\documentclass[10pt,reqno]{amsart}
\RequirePackage[margin=1in,dvips]{geometry}
\RequirePackage{float}
\RequirePackage{graphicx}
\RequirePackage{amsmath, amsthm, amssymb, amsfonts,slashed}
\RequirePackage[colorlinks=true]{hyperref}
\hypersetup{urlcolor=blue, citecolor=blue}
\RequirePackage[nameinlink]{cleveref}
\RequirePackage{xifthen}
\RequirePackage{subcaption}
\RequirePackage{indentfirst}
\RequirePackage{microtype}

\newenvironment{Eq}[1][]{\begin{equation}\ifthenelse{\equal{#1}{}}{}{\tag{#1}}\begin{aligned}}{\end{aligned}\end{equation}\ignorespacesafterend}
\newenvironment{Eq*}[1][]{\begin{equation*}\ifthenelse{\equal{#1}{}}{}{\tag{#1}}\begin{aligned}}{\end{aligned}\end{equation*}\ignorespacesafterend}
\numberwithin{equation}{section}
\newtheorem{theorem}{Theorem}[section]
\newtheorem{proposition}[theorem]{Proposition}
\newtheorem{corollary}[theorem]{Corollary}
\newtheorem{btsz}{Bootstrap ansatz}
\newtheorem{lemma}[theorem]{Lemma}
\newtheorem{claim}[theorem]{Claim}
\theoremstyle{remark}%
\newtheorem{remark}{Remark}[section] 
\theoremstyle{definition}

\DeclareMathOperator{\I}{i}
\DeclareMathOperator{\fd}{d}	\renewcommand{\d}{\fd}
\DeclareMathOperator{\diag}{diag}
\newcommand{\const}{\mathrm{const}.}
\newcommand{\Roma}[1]{\uppercase\expandafter{\romannumeral#1}}
\newcommand{\varE}{{\mathfrak E}}
\newcommand{\varF}{\varE'}
\newcommand{\Ma}[1]{\max\{#1,1\}}
\newcommand{\Mi}[1]{\min\{#1,1\}}

\newcommand{\mm}{{\mathbf{m}}}

\newcommand{\bZ}{{\mathbf Z}}

\newcommand{\bL}{{\mathbf L}}

\newcommand{\bOmega}{{\boldsymbol \Omega}}
\newcommand{\bY}{{\mathbf Y}}

\newcommand{\rH}{{\mathcal H}}
\newcommand{\rHb}{{\underline{\mathcal H}}}
\newcommand{\rD}{{\mathcal D}}
\newcommand{\rR}{{\mathbb{R}}}
\newcommand{\rS}{{\mathcal{S}}}
\newcommand{\rN}{{\mathbb{N}}}

\newcommand{\rC}{{\mathbb{C}}}
\newcommand{\rDa}{{\rD_{U,U_*}^{V,\infty}}}
\newcommand{\rDb}{{\rD_{l;U,U_V}^{V,V_U}}}

\newcommand{\LD}{{\mathcal L}}
\newcommand{\hD}{{\hat D}}
\DeclareMathOperator{\divv}{div}\renewcommand{\div}{\divv}
\DeclareMathOperator{\Div}{Div}

\newcommand{\Lb}{{\underline L}}
\newcommand{\alphab}{{\underline\alpha}}

\newcommand{\sdiv}{{\slashed \div}}
\newcommand{\spartial}{{\slashed \partial}}
\newcommand{\snabla}{{\slashed \nabla}}
\newcommand{\sD}{{\slashed D}}

\newcommand{\svE}{{\slashed{\mathcal{E}}}}

\newcommand{\sJ}{{\slashed J}}
\newcommand{\sA}{{\slashed A}}

\newcommand{\sm}{{\slashed m}}
\newcommand{\sGamma}{{\slashed \Gamma}}

\newcommand{\rsdiv}{{\mathring \sdiv}}

\newcommand{\rsvE}{{\mathring \svE}}

\newcommand{\Alphab}{{\underline{\mathcal A}}}
\newcommand{\varC}{{\mathcal C}}

\newcommand{\INF}[1]{{\kl[#1\kr]^{(\infty)}}}
\newcommand{\INFt}[1]{{\kl[#1\kr]^{(\infty)}}}

\newcommand{\q}{{\mathbf q_0}}

\newcommand{\rF}{{\mathring F}}
\newcommand{\rrho}{{\mathring \rho}}
\newcommand{\rsigma}{{\mathring \sigma}}
\newcommand{\ralpha}{{\mathring \alpha}}
\newcommand{\ralphab}{{\mathring \alphab}}

\renewcommand{\tt}{{\tilde t}}
\newcommand{\tx}{{\tilde x}}
\newcommand{\tr}{{\tilde r}}
\newcommand{\tomega}{{\tilde\omega}}
\newcommand{\tu}{{\tilde u}}
\newcommand{\tv}{{\tilde v}}
\newcommand{\tU}{{\tilde U}}

\newcommand{\tsquare}{{\tilde \square}}
\newcommand{\tpartial}{{\tilde \partial}}
\newcommand{\tspartial}{{\tilde \spartial}}
\newcommand{\tD}{{\tilde D}}

\newcommand{\tOmega}{{\tilde \Omega}}
\newcommand{\tLb}{{\tilde\Lb}}
\newcommand{\tL}{{\tilde L}}
\newcommand{\tT}{{\tilde T}}
\newcommand{\tS}{{\tilde S}}
\newcommand{\tK}{{\tilde K}}
\newcommand{\te}{{\tilde e}}

\newcommand{\tA}{{\tilde A}}
\newcommand{\tsA}{{\tilde \sA}}
\newcommand{\tF}{{\tilde F}}
\newcommand{\tphi}{{\tilde \phi}}
\newcommand{\tPsi}{{\tilde \Psi}}
\newcommand{\talpha}{{\tilde \alpha}}
\newcommand{\tJ}{{\tilde J}}

\newcommand{\trS}{{\tilde\rS}}
\newcommand{\trH}{{\tilde\rH}}
\newcommand{\trHb}{{\tilde\rHb}}
\newcommand{\trD}{{\tilde\rD}}

\newcommand{\tbZ}{{\tilde\bZ}}

\newcommand{\kl}[1]{\mathopen{}\left#1}
\newcommand{\kr}[1]{\right#1}
\newcommand{\Low}[1]{\limits_{\substack{#1}}}
\renewcommand{\emph}{}

\crefname{equation}{}{}
\newcommand{\Tm}[1]{Theorem \ref{#1}}

\newcommand{\La}[1]{Lemma \ref{#1}}

\newcommand{\Sn}[1]{Section \ref{#1}}
\newcommand{\Cm}[1]{Claim \ref{#1}}
\newcommand{\Cy}[1]{Corollary \ref{#1}}
\newcommand{\Pn}[1]{Proposition \ref{#1}}
\newcommand{\Bz}[1]{Bootstrap ansatz \ref{#1}}

\newcounter{part0}[subsection]
\renewcommand{\part}[1][]{\noindent\refstepcounter{part0}{\bfseries Part \number\value{part0}:} #1.\par}

\newcounter{part1}[part0]

\newcounter{part2}[part1]

\title[MKG equation with scattering data]
{The Maxwell-Klein-Gordon equation with scattering data}

\author{Wei Dai}
\address{School of Mathematical Sciences\\ Zhejiang University of Technology\\ Zhejiang, China}
\email{daiw23@zjut.edu.cn}

\author{He Mei}
\address{School of Mathematical Sciences\\ Shenzhen University\\ Shenzhen 518060, China}
\email{meihe@szu.edu.cn}

\author{Dongyi Wei}
\address{School of Mathematical Sciences \\ Peking University\\ Beijing, China}
\email{jnwdyi@pku.edu.cn}

\author{Shiwu Yang}
\address{Beijing International Center for Mathematical Research\\ Peking University\\ Beijing, China}
\email{shiwuyang@math.pku.edu.cn}

\date{\today}

\begin{document}

\bibliographystyle{siam}

\begin{abstract}
It has been shown in \cite{MR4030741} that general large solutions to the Cauchy problem for the Maxwell-Klein-Gordon system (MKG) in the Minkowski space $\mathbb{R}^{1+3}$ decay like linear solutions. One hence can define the associated radiation field on the future null infinity as  the limit of $(r\underline{\alpha}, r\phi)$ along the out going null geodesics. In this paper, we show the existence of a global solution to the MKG system which scatters to any given sufficiently localized radiation field with arbitrarily large size and total charge. The result follows by studying the characteristic initial value problem for the MKG system with general large data by using gauge invariant vector field method. We in particular extend the small data result of He in \cite{MR4299134} to a class of general large data. 
\end{abstract}

\keywords{Maxwell-Klein-Gordon system; Characteristic initial value problem; Scattering}

\subjclass[2020]{35L05, 35Q61, 35Q70}

\maketitle

\section{Introduction}
\label{Sn:I}

The classical Maxwell-Klein-Gordon (MKG) equation describes the motion of a  charged particle driven by the electromagnetic field. Let $A$ be a real valued $1$-form on $\mathbb{R}^{1+3}$. The covariant derivative $D_\mu$ acts on a complex valued scalar field $\phi$ via the relation 
\begin{Eq*}
D_\mu \phi =\partial_\mu \phi+\I A_\mu \phi,\qquad \I:=\sqrt{-1}. 
\end{Eq*}
The commutator of this covariant derivative gives the Maxwell field $F$ with components
\begin{Eq*}
F_{\mu\nu}=\partial_\mu A_\nu-\partial_\nu A_\mu,
\end{Eq*}
which can also be viewed as the $2$-form $F=dA$. 
The MKG equation is  a system for the connection field $A$ and the scalar field $\phi$
\begin{Eq}
\label{Eq:MKGe-Fphi}
\begin{cases}
\partial^\mu F_{\mu\nu}=-J_\nu[\phi]:=-\Im(\phi\cdot\overline{D_\nu \phi}),\\
\square_A\phi:=D^\mu D_\mu \phi=0.
\end{cases}
\end{Eq}
Here and throughout the paper, raising and lowering the indices is with respect to the flat Minkowski metric $\mm=\diag\{-1, 1, 1, 1\}$.
The MKG system is one of the simplest models in gauge theory
and enjoys the gauge freedom,
that is, 
 $(A-\d\xi, e^{\I\xi}\phi)$ solves the same system with $(A,\phi)$ for any smooth enough function $\xi(t,x)$.
However, we may note that the Maxwell field $F$ is gauge invariant and the total charge for the particle 
\begin{Eq}\label{Eq:Do_q0} 
\q:=\frac{1}{4\pi}\int_{\rR^3}J_0[\phi]\d x=\frac{1}{4\pi}\int_{\rR^3}\partial^iF_{0i}\d x
\end{Eq}
is conserved, which in particular implies that the electric field $E=(F_{0i})_{i=1}^3$ has a nontrivial tail at the spatial infinity.

The Cauchy problem with initial data given on $\{t=0\}$ has been well studied. It has been shown in \cite{MR649158}, \cite{MR649159}, \cite{MR1271462}, \cite{MR2051613} that for sufficiently smooth initial data or even initial data in the energy space or below, the solution exists globally in time. One of the central themes  regarding the MKG system in the past decades is to investigate the long time dynamics of the solutions. Various decay estimates have been obtained for the solutions firstly in \cite{MR802944}, \cite{MR654042}  for sufficiently small and rapid decaying initial data (hence the total charge vanishes). As pointed above, the charge has a long range effect on the asymptotic behaviors of the solutions. An essential improvement was contributed by Lindblad-Sterbenz in \cite{MR2253534}, in which they obtained decay estimates for the solutions for general small initial data with nonzero total charge. This small data result has been extended to general large data recently in  \cite{MR4030741},  \cite{wei2022global}. More precisely, for initial data bounded in some weighted energy space, the solution verifies the pointwise decay estimate
\begin{Eq*}
|\underline{\alpha}|\leq C t^{-1},\quad |\phi|\leq Ct^{-1}.
\end{Eq*}
Here $\alphab=(F_{\Lb e_i})_{i=1,2}$ (see the precise definition in \Sn{Sn:P}). Now, for $u:=\frac{t-r}{2}$, $v:=\frac{t+r}{2}$, $r:=|x|$ and $\omega\in \mathbb{S}^2$, one can take the limit
\begin{equation}
\label{eq:radiation}
\Alphab(u, \omega)=\lim_{v\rightarrow\infty} (r\alphab),\qquad
\Phi(u, \omega)=\lim_{v\rightarrow\infty} (r\phi),
\end{equation}
which can be viewed as functions on the future null infinity. In particular, for data on the initial Cauchy hypersurface $\{t=0\}$ with total charge $\q$, we obtained the associated scattering data $(\Alphab, \Phi; \q)$ on the future null infinity in the sense of Friedlander \cite{MR583989}, \cite{MR1846782}, \cite{MR0460898}. It is then natural to ask for given $(\Alphab, \Phi; \q)$ on the future null infinity, whether there is a global solution to the MKG system such that the corresponding scattering data agree with $(\Alphab, \Phi; \q)$.

Such a problem is partially motivated by the above mentioned classical results of Friedlander for linear wave equations, in which he established the one to one correspondence between the Cauchy data and the scattering data in energy space for solutions of linear wave equations. 
After conformal transformation, the future null infinity becomes a conic light cone. 
The problem is then reduced to a characteristic initial value problem for the MKG system with data on a light cone.  
The study of the characteristic initial value problem for linear and nonlinear wave equations also attracts extensive attention. 
From a physical point of view, an observer is able to measure the initial data on his past light cone, 
but is not able to instantaneously measure the data on a nearby space-like surface, 
see for example the discussions in \cite{MR3528223}, \cite{MR3585918}. 
For some geometric wave equations like the Einstein equation in general relativity and Yang-Mills equations, 
the initial data has to satisfy the associated constraint equations,
which are nonlinear elliptic equations for the Cauchy problem while they are transport equations for the characteristic initial value problem. 
The easier constraint equation allows one to specify the initial data and helps to simplify the full problem. 
In the monumental work \cite{MR2488976} of Christodoulou (also see his early works \cite{MR860312}, \cite{MR848643}, \cite{MR1680551}), 
a type of short pulse data for the characteristic initial value problem for the vacuum Einstein equation leads to the dynamic formation of trapped surfaces for the spacetime as solution to the Einstein equation.

However, results for the characteristic initial value problem for general linear and nonlinear wave equation are less fruitful compared to the Cauchy problem, partially due to the lack of Fourier analysis or effective fundamental solution for the characteristic initial value problem. One of our main goals in this paper is to study the characteristic initial value problem for the MKG system. We then establish a scattering theory for the nonlinear model of MKG system for general large initial data on the future null infinity.

\subsection{Statement of the main results}
To state our main theorems,
let's define some necessary notations. 
Additional to the standard Cartesian coordinates $(t, x)$ of $\mathbb{R}^{1+3}$, we may also use the polar coordinate system $\{t, r=|x|, \omega=r^{-1}x\}$. Let $u=\frac{t-r}{2}$, $v=\frac{t+r}{2}$ and  $\Lb=\partial_u$, $L=\partial_v$. Then $\{\Lb,L,e_1,e_2\}$ forms a null frame with $\{e_1,e_2\}$  an orthonormal basis of the tangent space of the two sphere with radius $r$.
 
Denote $\mathcal{I}^+$ as the future null infinity, parameterized by  $(u, \omega)\in \mathbb{R}\times \mathbb{S}^2$. The scattering data as our initial data on $\mathcal{I}^+$ are a 1-form $\Alphab$ on the unit sphere for each fixed $u$ and a complex scalar field $\Phi$. Our goal is to find a solution $(F, \phi)$ to the MKG system \eqref{Eq:MKGe-Fphi} such that 
\eqref{eq:radiation} holds with $\alphab= (F_{\Lb e_i})_{i=1,2}$. To define the covariant derivatives for the scalar field on $\mathcal{I}^+$, 
we need to extract a 1-form on $\mathcal{I}^+$ which should be considered as the limitation of $A$. 
Let $[r^2\sigma]^{(\infty)}:=\lim_{v\rightarrow\infty}r^2\sigma$ with $\sigma:=F_{e_1 e_2}$.
We are able to calculate it through $\Alphab$ and the transport equation
\begin{equation}
\label{eq:trans:sigma}
\Lb[r^2\sigma]^{(\infty)}=-\rsdiv({}^*\Alphab),
\end{equation}
which is induced by \eqref{Eq:MKGe-ud_FZ},
where $\rsdiv$ is the sphere divergence and ${}^*$ is the Hodge star operator on $\mathbb{S}^2$.
Then, we see that 
\begin{Eq*}
\mathcal{F}^+:=[r^2\sigma]^{(\infty)}\rsvE+\d u \wedge \Alphab
\end{Eq*}
is a closed 2-form on $\mathcal{I}^{+}$, where $\rsvE$ is the volume form of $\mathbb{S}^2$. We hence can find a 1-form $B$ on $\mathcal{I}^{+}$ such that $\mathcal{F}^+=dB$. Now, we get the covariant derivatives $D_{u}:=\partial_u+\I B_u$, $D_{\omega}:=\partial_\omega+\I B_\omega$ on $\mathcal{I}^{+}$. 

Our first result concerns the continuation criteria for the characteristic initial value problem for the MKG system. For any $U_*\in\rR$, we define the following weighted energy $\|\cdot\|_{S\!N_{U_*}}$ for the data $(\Alphab,\Phi)$ on $\mathcal{I}^{+}$ 
\begin{Eq*}
\|(\Alphab,\Phi)\|_{S\!N_{U_*}}^2:=&\hspace{-5pt}\sum\Low{n\leq 2\\|\beta|\leq 6-2n}\hspace{-5pt}\int_{U_*}^\infty \int_{\mathbb{S}^2} \kl<u\kr>^{6+2n}|\partial_u^n \partial_\omega^\beta\Alphab|^2\d\omega\d u
+\hspace{-20pt}\sum\Low{n\leq 3\\|\beta|\leq 7-2\max\{n,1\}}\hspace{-20pt}\int_{U_*}^\infty \int_{\mathbb{S}^2} \kl<u\kr>^{4+2n}|D_u^n D_\omega^\beta\Phi|^2\d\omega\d u.
\end{Eq*} 
Here $\kl<u\kr>:=\sqrt{|u|^2+1}$. We emphasize here that the above norm is gauge invariant, which is independent of the choice of the connection field $B$ on $\mathcal{I}^{+}$.
\begin{theorem}
\label{Tm:Main1}
For any given scattering data $(\Alphab,\Phi)$ such that  
 $ \|(\Alphab,\Phi)\|_{S\!N_{U_*}}$ is finite for any $U_*\in \mathbb{R}$, there exists a solution $(F,\phi)$ to the MKG system \eqref{Eq:MKGe-Fphi} which scatters to $(\Alphab,\Phi)$ in the sense of \eqref{eq:radiation}. Moreover the solution is unique if for any $U_*\in \mathbb{R}$
\begin{Eq}
\label{eq:uniqueness}
\lim_{v\rightarrow\infty}\|(r\alphab,r\phi)\|_{S\!N_{U_*}'}=\|(\Alphab,\Phi)\|_{S\!N_{U_*}},
\end{Eq}
in which the weighted energy $\|(r\alphab,r\phi)\|_{S\!N_{U_*}'}$ is similarly defined on the light cone  $\{v=\const\}$. 
\end{theorem}

Our second result gives quantitative estimates for the solution $(F, \phi)$. Due to the existence of total charge $\q$, we need the following compatibility condition 
 \begin{Eq}
 \label{Eq:crosd}
\const\equiv\int_{-\infty}^\infty \Im\kl(\Phi\overline{D_u\Phi}\kr)- \rsdiv \Alphab\d u=:\q,\qquad
0\equiv \int_{-\infty}^\infty  \rsdiv ({}^*\!\Alphab)\d u.
\end{Eq}
This condition is implied by the asymptotic behaviors for the solutions for the Cauchy problem. Indeed the second condition follows from the transport equation \eqref{eq:trans:sigma} for the component $\sigma$ of the Maxwell field together with the pointwise decay properties of the solution obtained in \cite{MR4030741}. 
Similarly for the first identity, 
for $[r^2\rho]^{(\infty)}:=\lim_{v\rightarrow\infty}r^2\rho$ with $\rho:=\frac{1}{2}F_{\Lb L}$, we have the transport equation 
\begin{Eq*}
\Lb [r^2\rho]^{(\infty)}=\rsdiv\Alphab-\Im (\Phi \cdot \overline{D_{u}\Phi}),
\end{Eq*}
which is induced by equation \eqref{Eq:MKGe-ud_FZ}. 
The asymptotic behavior of $\rho$ then leads to the above compatibility condition \eqref{Eq:crosd}. See \Sn{Sn:PoCEo_FphiZ_oeni} for more detailed discussions.
 
For any constant $0\leq \varepsilon_1<1 $, we define the weighted energy for the initial scattering data
\begin{Eq*}
\mathcal{E}_{\varepsilon_1}=\|(\Alphab,\Phi)\|_{S\!N_{0}}^2+\hspace{-5pt}
\sum\Low{n\leq 2\\|\beta|\leq 6-2n}\hspace{-5pt}\int_{-\infty}^0 \int_{\mathbb{S}^2}   \kl<u\kr>^{2+2n-\varepsilon_1}|\partial_u^n \partial_\omega^\beta\Alphab|^2\d\omega\d u +
\hspace{-20pt}
\sum\Low{n\leq 3\\|\beta|\leq 7-2\max\{n,1\}}\hspace{-20pt}\int_{-\infty}^0 \int_{\rS^2}   \kl<u\kr>^{2n-\varepsilon_1}|D_u^n D_\omega^\beta\Phi|^2\d\omega\d u.
\end{Eq*} 
Then we have:
\begin{theorem}\label{Tm:Main2}
Assume that the initial scattering data $(\Alphab,\Phi)$ satisfy the above compatibility condition with total charge $\q$ and the weighted energy $\mathcal{E}_{\varepsilon_1}$ is finite for some constant $0\leq\varepsilon_1<1$. Then there exists a solution $(F, \phi)$ to the MKG equation \eqref{Eq:MKGe-Fphi} with this given radiation field $(\Alphab,\Phi)$  and a constant $C$ depending only on  $\mathcal{E}_{\varepsilon_1}$, $\q$, $\varepsilon_1$ and $\varepsilon_2$  such that the weighted energy on the Cauchy hypersurface $\{t=0\}$
\begin{Eq}
\label{Eq:Main2}
\sum_{|\beta|\leq 2}\int_{\rR^3}\kl<r\kr>^{2-\varepsilon_2+2|\beta|}\kl|\kl(DD^\beta\phi,\partial^\beta \rF\kr)\kr|^2\d x \leq C
\end{Eq}
 for all $\varepsilon_1<\varepsilon_2<1$. Here $\rF$ is the chargeless part of the Maxwell field 
\begin{Eq*}
\rF=F-F[\q]=F- \q r^{-2}\d t\wedge\d r\cdot{\bf 1}_{\{u\leq -1\}}.
\end{Eq*}
Moreover the solution is unique if we assume that the solution is continuous in the sense of \eqref{eq:uniqueness} for all $U_*$ as in theorem \ref{Tm:Main1}.
\end{theorem}
A couple of remarks are in order.

\begin{remark}
  The boundedness of the weighted energy on the Cauchy hypersurface $\{t=0\}$ is sufficient to conclude the pointwise decay estimates for the solutions near the spatial infinity.
\end{remark}

\begin{remark}
 The decay assumption on the scattering data at time infinity ($u\rightarrow +\infty$) ensures that the data on the conic light cone after a conformal transformation is regular up to the conic point. This assumption will be significantly improved in our forthcoming paper \cite{DaiYang2024} by studying the MKG system with singular data on the initial light cone. 
\end{remark}

\begin{remark}
 The scattering data problem  studied here is closely related to the theory of scattering for linear and nonlinear fields, see discussion for example in \cite{MR0217440}. Combined with the decay estimates for the Cauchy problem, we can establish a map from the scattering data on the future null infinity to that on the past null infinity. However, we lose a little bit regularity ($\varepsilon_1<\varepsilon_2$) due to the existence of large total charge $\q$. 
\end{remark}

\subsection{Related results}
 The most related work is the small data result obtained by He in \cite{MR4299134}, in which the scattering data are given for the scalar field and the connection field in Lorenz gauge. Under this particular gauge condition, the scalar field can not be freely assigned as there is a phase correction based on the asymptotic behaviors for the solutions for the Cauchy problem in \cite{MR3936130}, \cite{MR2253534}. 
Another related work was given by Wang in \cite{wang2010radiation}. In this work, he considered the scattering problem for Einstein equations in harmonic coordinates, and also used the conformal compactification method and discussed the characteristic initial value problem, see also \cite{wang2014} for the higher dimension cases.  
General nonlinear wave equation verifying the null condition or weak null condition with small scattering data on the future null infinity has been studied earlier by Lindblad-Schlue in \cite{lindblad2017scattering}. Another deep and remarkable result is the construction of a large class of black hole spacetimes as solutions to the vacuum Einstein equation with scattering data on the  future null infinity and the event horizon of a fixed Schwarzschild spactime in \cite{dafermos2013scattering}  by Dafermos-Holzegel-Rodnianski. For scattering theory for linear fields on various black hole spacetimes, we refer to the recent progresses in \cite{MR2070126}, \cite{MR3494169}, \cite{MR3798305}, \cite{MR4166630} and references therein. 

Scattering theory for nonlinear fields with large data is in general more difficult. 
Baez-Segal-Zhou in \cite{MR1073286}  constructed the scattering operator for the cubic defocusing semilinear wave equation in the energy space in Minkowski space, which has later been extended  to non-stationary asymptotic flat backgrounds by Joudioux in \cite{MR2910977}, \cite{MR4109802}.  Similar result has been obtained by Baskin-Barreto in \cite{MR3324913} for the energy critical defocusing semilinear wave equations. We remark here that the good sign and  the superconformal power (see discussions for the Cauchy problem in \cite{MR4411879}) of the nonlinearity is of particular importance to establish the scattering theory. Hence the MKG system studied here provides another nonlinear model for which  one can find a continuous map from a class of scattering data on the future null infinity to that on the past null infinity.

There are mainly two types of initial hypersurfaces for characteristic initial value problem for nonlinear wave equations: a characteristic null cone and two transversely  intersecting null hypersurfaces, one of which may be part of the null infinity. Local existence and uniqueness of the solutions for general quasilinear wave equations with smooth data on a characteristic null cone were shown by Cagnac in \cite{MR648323}.  Then Rendall in \cite{MR1032984} extended this  result to two intersecting null hypersurfaces by using  a different approach with smooth data. Generalizations on these classical results for general nonlinear wave equation could be found for example in \cite{MR1079777}, \cite{dossa1992problemes}, \cite{MR1267067}, \cite{MR1851621}, \cite{MR2168739}, \cite{MR2153385}. We emphasize here that the local existence result means that the solution exists in a small neighborhood of the conic point or the intersecting sphere of the two null hypersurfaces. Local existence in a uniform neighborhood of the null hypersurfaces as needed in our case (see details in the next subsection) holds true under certain conditions on the nonlinearity (\cite{MR2124075}, \cite{MR2475332}, \cite{MR3528223},  \cite{MR3974128}).
 
An important application of the above local existence results is to the Einstein equation in general relativity. Local existence in a neighborhood of conic light cone was shown in \cite{MR2785136}, \cite{MR3129312}. For the case of two intersecting null hypersurfaces or data on part of the past null infinity, we refer to the discussions in  \cite{MR661723}, \cite{MR2719376}, \cite{MR3116325} and references therein. The Einstein equation under certain gauge condition is a complicated system of quasilinear wave equation for the metric components, which does not verify the conditions in \cite{MR2124075}, \cite{MR2475332}. However Luk in \cite{MR2989616} was able to extend the local existence result in a neighborhood of the intersecting two sphere to the uniform future of the characteristic initial null cones, see further discussions and applications in \cite{MR4148409}, \cite{MR3851745}, \cite{MR3318018}, \cite{MR3758993}.

\subsection{Strategy for the proof}
In this subsection, we sketch the main ideas and novelties for the proof. For the Cauchy problem investigated in \cite{MR4030741}, the solution is first constructed in the exterior region out side of a forward light cone $\{u\leq U_*\}$. In this region, the difficulty is the long range effect of the large total charge. But one can choose  $U_*$ large enough such that the chargeless part of the data are sufficiently small. Then inside the light cone, the conformal compactification method was applied to show the decay estimates for the solutions. For the inverse problem studied in this paper, the strategy is to construct the solution in the reverse order.

We first need to show the existence of the solution in a neighborhood of the time infinity. Similar to the Cauchy problem, we can make use of the conformal structure of the MKG system in Minkowski space $\mathbb{R}^{1+3}$. For some $U_*\in \mathbb{R}$, the region $\{(t,x):u>U_*\}$ in $\mathbb{R}^{1+3}$ can be conformally compactified to a region bounded by two light cones as shown in the following picture:
\begin{figure}[H]
\centering
\includegraphics[width=0.5\textwidth]{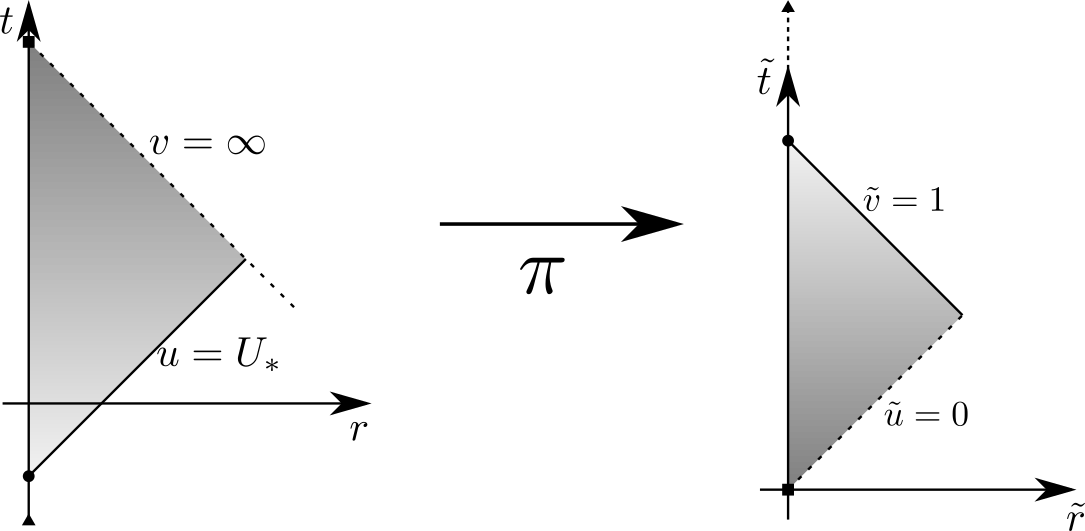}
\end{figure}
Here we may note that the part of the future null infinity $\{v=\infty\}$ is mapped to $\{\tilde{u}=0\}$ and the light cone $\{u=U_*\}$ corresponds to $\{\tilde{v}=1\}$. We are then led to study the characteristic initial value problem for the MKG system with data on the conic light cone $\{\tilde{u}=0\}$. Our assumption on the scattering data on the future null infinity implies that the data on the cone $\{\tilde{u}=0\}$ is sufficiently regular. The difficulty now is to show the existence of the solution to the MKG system in the maximal Cauchy development $\{\tilde{u}\geq 0, \tilde{v}\leq 1\}$ with general  sufficiently regular data on the light cone $\{\tilde{u}=0,  \tilde{v}\leq 1\}$.  The idea is to first show that the solution exists in a neighborhood of the light cone $\{\tilde{u}=0\}$, that is, for some constant $\epsilon>0$ there is a unique solution to the MKG system in the region $\{0\leq \tilde{u}\leq \epsilon,  \tilde{v}\leq 1\}$. This will be carried out by studying the MKG system under the Lorenz gauge and the proof is inspired by the work \cite{MR2989616} of Luk for the local existence result for the characteristic initial value problem for the Einstein equations. 

To extend the local solution to a global one in the whole region $\{0\leq \tilde{u}\leq 1, \tilde{v}\leq 1\}$, we rely on the classical results \cite{MR649158}, \cite{MR649159} of Eardley-Moncreif for the Cauchy problem (also see the work \cite{MR1271462} of Klainerman-Machedon). We can extract a hyperboloid which is space like in the region $\{0\leq \tilde{u}\leq \epsilon, 0\leq \tilde{v}\leq 1\}$. Then the solution exists in the maximal Cauchy development of this hyperboloid as long as the solution is sufficiently regular on the hyperboloid. This is standard, see details in \Sn{Sn:Taiir}. %

 The above discussion also implies that the solution exists in the full Minkowski space, that is, Theorem \ref{Tm:Main1} holds. But we do not have quantitative control on the solution as $U_*$ goes to $-\infty$. The existence of nonzero total charge does not allow us to use the conformal mapping method in the whole Minkowski space. To show that the solution is also bounded in some weighted energy space on any Cauchy hypersurface as stated in Theorem \ref{Tm:Main2}, we turn to gauge invariant vector field method. First we can choose $-U_*$, $V_*$ large enough such that the data on the later part of the future null infinity $\{u\leq U_*, v=\infty\}$ and the out going null cone $\{u=U_*, v\geq V_*\}$ are sufficiently small. It is then sufficient to study the solutions in the exterior region $\{u\leq U_*, v\geq V_*\}$. To see the new difficulties, let's review the $r$-weighted energy identity which played a significant role for the Cauchy problem in the region $\mathcal{D}$ bounded by the out going null hypersurface $\mathcal{H}_{U_*}$, the in coming null hypersurface $\underline{\mathcal{H}}_{V_*}$ and the Cauchy hypersurface $\{t=0\}$  for all $0\leq p\leq 2$
\begin{align*}
  &\iint_{\mathcal{D}}r^{p-1}\left(p(|\hat{D}_L\phi|^2+|\alpha|^2)+(2-p)(|\sD\phi|^2+|\rrho|^2+ \sigma^2)\right)dxdt\\
&+\int_{\mathcal{H}_{U_*} }r^{p+2}(|\hat{D}_L\phi|^2+| \alpha|^2)dvd\omega+\int_{\underline{\mathcal{H}}_{V_*} }r^{p+2}(|\sD\phi|^2+|\rrho|^2+\sigma^2)dud\omega\\
=&\frac{1}{2}\int_{\{t=0\}} r^p(|\hat{D}_L\phi|^2+|\sD\phi|^2)+r^{p+2}(|\alpha|^2+|\rrho|^2+\sigma^2)drd\omega-\iint_{\mathcal{D}}  \q r^{p-2} \Im(\phi\cdot \overline{D_L\phi}) dxdt.
\end{align*}
Here $\hat{D}_L\phi=D_L\phi+r^{-1}\phi $ and $\rrho=\rho-\q r^{-2}$. We still have the same issue to control the error term arising from the existence of large total charge $\q$ (the last term with indefinite sign).  The bulk term in the first line contributes a good sign for the Cauchy problem when the data are assigned on $\{t=0\}$ while it becomes negative for the characteristic initial value problem with data on $\mathcal{H}_{U_*} $ and the future null infinity (the limit of $\underline{\mathcal{H}}_{V_*}$ as $V_*$ goes to infinity). This means that the robust $r$-weighted energy method introduced by Dafermos-Rodnianski in \cite{MR2730803} may not work for the characteristic initial value problem considered here.

In our paper, we first adopt the weighted conformal vector field $|u|^{-\varepsilon_1} K$, where $K=2u^2 \Lb+2 v^2 L$ has been used in \cite{MR4299134} for the small data case.
Then, we apply the energy method to the region bounded by $\rH_{U_*}$, $\rH_{U}$, $\rHb_{\infty}$ and $\rHb_{V}$ with $U<U_*<0$ and $V_*<V$, and get that
\begin{Eq*}
&\int_{\mathcal{H}_{U} }|u|^{-\varepsilon_1}  v^2 |\hat{D}_L\phi|^2
\leq C(\mathcal{E}_{\varepsilon_1}) +C_\q\int_{U}^{U_*}\int_{\rH_u} v^{-1}\kl(|u|^{-\varepsilon_1}v^2 |r^{-1}\phi||\hD_L\phi|\kr)\d u+else.
\end{Eq*}
Here we only show this reduced version of energy estimate,
since the error term arising from the nonzero total charge depends only on the scalar field, 
where we have omitted the volume element for simplicity.
As discussed in \cite{MR3758434}, 
integrating from the future null infinity, 
we can bound $\phi$ in terms of $\hD_{L}\phi$, 
that is, we can view  $|r^{-1}\phi|$ as $|\hD_{L}\phi|$.
However, such error term still can not be absorbed since we only have $v^{-1}\leq |u|^{-1}$ and that $\q$ is not small.
Now, the key observation is that this error term could be controlled by using Gronwall's inequality if we restrict the above identity to the region $v\geq |u|^{1+\varepsilon_3}$ for some $\varepsilon_3>0$. 
Then, noticing $v^{-1}\leq |u|^{-1-\varepsilon_3}$ and $|u|\geq |U_*|$, the $v^{-1}$ in error term can provide integrability and smallness simultaneously, for $|U_*|$ large enough. See \Sn{Sn:Taierp1} for the detailed discussion. 

For the remaining part where $v\leq |u|^{1+\varepsilon_3}$, we change the weighted multiplier to $|u|^{-R-\varepsilon_2}v^{R} K$ for some large constant $R$ and apply it to the region $\mathcal{D}$ bounded by $\{v=|u|^{1+\varepsilon_3}\}$ with $\varepsilon_3:=({\varepsilon_2-\varepsilon_1})/{R}$, $\rHb_{V} $ and $\rH_{U}$ as depicted below:
\begin{figure}[H]
\centering
\includegraphics[width=0.3\textwidth]{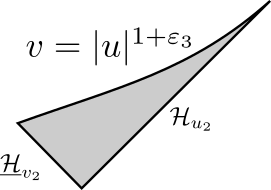}
\end{figure}
Noticing that the previous step has also shown that the weighted energy on the space like hypersurface $\{v=|u|^{1+\varepsilon_3}\}$ is finite. We then can derive the energy estimate
\begin{align*}
R\iint_{\mathcal{D}}   |u|^{2-R-\varepsilon_2}v^{-1+R} |\hD_{\Lb}\phi|^2
\leq C(\mathcal{E}_{\varepsilon_1}) +C_\q\iint_{\rD} |u|^{2-R-\varepsilon_2}v^{-1+R}|r^{-1}\phi||\hD_L\phi|+else,
\end{align*}
where the factor $R$ in front of the bulk term comes from the skillfully selection of multiplier.
Again $r^{-1}|\phi|$ is comparable to $|\hat{D}_L\phi|$.
Now, for sufficiently large constant $R$, depending only on the total charge $\q$, the error term arising from the total charge can be absorbed by the bulk term on the left hand side. See \Sn{Sn:Taierp2} for detailed discussion.

\subsection{Organization of the paper}
The paper is organized as follows.
In \Sn{Sn:P} we make some preparation which includes notations, basic calculations, energy method, and some induced formulas of MKG equations.
In \Sn{Sn:Taiir} we consider the MKG equation in the interior region and conclude   \Tm{Tm:Main1}. The argument in the exterior region is split into two part.   \Sn{Sn:Taierp1}
is devoted to the study of the solutions in the upper exterior region where $v\geq |u|^{1+\varepsilon_3}$ and   \Sn{Sn:Taierp2} gives the proof for the solution in the lower exterior region.   The last   \Sn{Sn:A} collects some computations and lemmas used in the proof.

\section{Preparation}\label{Sn:P}
\subsection{Notation}
In this subsection, we list some notations that will be used throughout the paper.

In this paper, $C$ and $C(\omega)$ are used to denote a positive constant and a bounded function respectively, 
depending on their indices, 
which may change from line to line. 
The convention $x\lesssim y$ and $y\gtrsim x$ mean $x\leq Cy$,
and $x\approx y$ means $x\lesssim y\lesssim x$.

We use $\Div$ as the time-space divergence, $\div$ as the space divergence,
and $\spartial$ (and similarly for $\snabla$, $\sD$, etc.,) to represent the spherical projection of $\partial$.
Besides, we shall mainly use the derivatives with respect to two sets of vector fields. They are $\mathcal{Y}:=\{\Lb,L\}\cup\{\Omega_{ij}\}$ and $\mathcal{Z}:=\{T,S,K\}\cup \{\Omega_{ij}\}$, where
\begin{Eq*}
\Omega_{ij}=&x_i\partial_j-x_j\partial_i=rC_{ij}^k(\omega)e_k,&\quad T=&\partial_t=\frac{1}{2}(\Lb+L),\\
S=&t\partial_t+r\partial_r=u\Lb+vL,&\quad K=&(t^2+r^2)\partial_t+2tr\partial_r=2(u^2\Lb+v^2L).
\end{Eq*}

The order of derivatives are sometimes important in this paper.
So we use the bold symbol, like $\bZ$ (and similarly for $\bY$, $\bOmega$, etc.,), 
to denote a collection of derivatives $(Z_1,Z_2,...,Z_m)$,
where $Z_i$ belongs to its corresponding set $\mathcal{Z}$ (or $\mathcal{Y}$, $\{\Omega_{ij}\}$, etc.,).
We use $\chi_{X}(\bZ)$ to denote the amount of derivatives in $\bZ$ that belongs to $X$, 
and $\chi(\bZ)$ to denote the total amount of derivatives in $\bZ$.
For example, setting $\bZ=(\Omega_{12},K,\Omega_{31})$,
then we have $\chi_\Omega(\bZ)=2$, $\chi_K(\bZ)=1$ and $\chi(\bZ)=3$.
Finally, we define $\bZ=(Z_1,Z_2,...,Z_m)$ acting on any $f$ as $Z_1Z_2...Z_mf$ (similarly for $D_{\bZ}$, $\LD_{\bZ}$, etc.,).
We also use the convention $\partial^{\leq n}$ 
(and similar for $\Omega^{\leq n}$, $\spartial^{\leq n}$, etc.,)  to
stand for the collection $\kl\{\partial^{\beta}\kr\}_{|\beta|\leq n}$.

For any $2$-form $G=G_{\mu\nu}\d x^\mu \d x^\nu$, we define its null decomposition
\begin{Eq*}
\rho[G]:=\frac{1}{2}\Lb^\mu L^\nu G_{\mu\nu},\quad \alphab[G]_i:=\Lb^\mu e_i{}^\nu G_{\mu\nu},\quad
\alpha[G]_i:=L^\mu e_i{}^\nu G_{\mu\nu},\quad \sigma[G]_{ij}:=e_i{}^\mu e_j{}^\nu G_{\mu\nu}.
\end{Eq*}
Note that $(\sigma[G]_{ij})$ can be considered as a spherical $2$-form. 
Thus, we have $\sigma[G]_{ij}=\sigma[G]\svE_{ij}$ 
where $\svE$ is the volume form of its sphere 
and $\sigma[G]:=2^{-1}\sigma[G]_{ij}\svE^{ij}$.

We use the shorthands $\rho$, $\alphab$, etc., while $G=F$; 
$\rrho$, $\ralphab$, etc., while $G=\rF$;
$\rho^{(\bZ)}$, $\alphab^{(\bZ)}$, etc., while $G=F^{(\bZ)}$, where $G^{(\bZ)}:=\LD_\bZ G$;
and $\rrho^{(\bZ)}$, $\ralphab^{(\bZ)}$, etc., while $G=\rF^{(\bZ)}$.
For the scalar field part, we use the shorthand $f^{(\bZ)}:=\hD_{\bZ}f$ where $\hD f:=r^{-1}D(rf)$.
Finally, we also use the shorthand $J^{(\bZ)}:=r^{-2}\LD_\bZ(r^2J)=r^{-2}\LD_{\bZ}\Im\big((r\phi)\cdot\overline{D(r\phi)}\big)$.

In this paper, we will mainly consider the region that corresponds to the light cone.
Under this consideration, we define
\begin{Eq*}
\rS_{U}^V:=&\{(t,x):u=U,~v=V\},\quad&
\rH_{U}^{V_1,V_2}:=&\{(t,x):u=U,~V_1\leq v\leq V_2\},\\
\rHb_{U_1,U_2}^{V}:=&\{(t,x):U_1\leq u\leq U_2,~v=V\},\quad&
\rD_{U_1,U_2}^{V_1,V_2}:=&\{(t,x):U_1\leq u\leq U_2,~V_1\leq v\leq V_2\}.
\end{Eq*}
We also define $\rD_{+;U_1,U_2}^{V_1,V_2}:=\rD_{U_1,U_2}^{V_1,V_2}\cap\{t\geq 0\}$.
When considering the integral in such regions, we omit the usual volume element.
For example, we denote
\begin{Eq*}
\int_{\rS_U^V}f:=\kl.\int_{\rS^2}f\cdot r^2\d\omega\kr|_{u=U,v=V}.
\end{Eq*}

\subsection{Basic calculation results}
In this subsection, we give some basic calculation results that will be used later.

Firstly, we record here the standard formulas of the \emph{Levi-Civita} connection $\nabla$ with respect to the null frame.
\begin{Eq*}
&\nabla_\Lb(\Lb,L,e_i)=\nabla_L(\Lb,L,e_i)=0,\quad&
&\nabla_{e_i}\Lb=-\nabla_{e_i}L=-r^{-1}e_i,\\
&\nabla_{e_i}e_j=\sGamma^{k}{}_{ij}e_k+\frac{1}{2r}\sm_{ij}(\Lb-L),
\end{Eq*}
where $\sm$ is the induced sphere metric of $m$ and $\sGamma$ is the sphere frame-Christoffel symbol.
Using them, we easily know
\begin{Eq*}
\rho=&\frac{1}{2}(\Lb A_L-L A_\Lb),&\qquad \alphab_i=&r^{-1}\Lb(rA_{e_i})-e_i A_\Lb,\\
\alpha_i=&r^{-1}L(rA_{e_i})-e_i A_L,&\qquad \sigma=&e_1A_{e_2}-e_2A_{e_1}+(\sGamma^i{}_{21}-\sGamma^i{}_{12})A_{e_i}.
\end{Eq*}

Next, we consider the Lie derivative of $F$ in $\{u\leq -1\}$.
We can easily calculate that $\LD_ZF[\q]=0$ for any $Z\in \mathcal{Z}$.
It means that excepting $\rho=\rrho+\rho_0$ with $\rho_0:=r^{-2}\q$,
there are $\kl(\alphab^{(\bZ)},\rho^{(\bZ)},\sigma^{(\bZ)},\alpha^{(\bZ)}\kr)=\kl(\ralphab^{(\bZ)},\rrho^{(\bZ)},\rsigma^{(\bZ)},\ralpha^{(\bZ)}\kr)$ 
 for any $\bZ$.

Then, we consider the communicator.
The only non-vanishing $[Y_1,Y_2]$ for $Y_1,Y_2\in \mathcal{Y}$ are 
\begin{Eq*}
&[\Omega_{i_1i_2},\Omega_{j_1j_2}]=C_{i_1i_2j_1j_2}^{k_1k_2}\Omega_{k_1k_2}.
\end{Eq*}
The non-vanishing $[Z_1,Z_2]$ for $Z_1,Z_2\in \mathcal{Z}$ excepting $[\Omega,\Omega]$ are 
\begin{Eq*}
&[T,S]=T,\quad [T,K]=2S,\quad [S,K]=K.
\end{Eq*}
Besides, we also have
\begin{Eq}\label{Eq:CofdaZ}
[\Lb,T]=&0,&\quad[\Lb,S]=&\Lb,&\quad [\Lb,K]=&4u\Lb,&\quad [\Lb,\Omega_{ij}]=&0,\\
[L,T]=&0,&\quad[L,S]=&L,&\quad [L,K]=&4vL,&\quad [L,\Omega_{ij}]=&0,\\
[e_i,T]=&0,&\quad[e_i,S]=&e_i,&\quad [e_i,K]=&2te_i,&\quad [e_i,\Omega_{jk}]=&C^l_{ijk}(\omega) e_l.
\end{Eq}
Excepting these, for the covariant derivative $D$, we have
\begin{Eq}\label{Eq:Cocd}
[D_\Lb,D_L]=2\I\rho,\quad [D_\Lb,D_{\Omega_{ij}}]=\I C_{ij}^k(\omega)r\alphab_k,\quad [D_L,D_{\Omega_{ij}}]=\I C_{ij}^k(\omega)r\alpha_k.
\end{Eq}
Finally, using the relation between \emph{Lie} derivatives and the usual derivatives, we can find that for any $N\in\rN_0$,
\begin{Eq}\label{Eq:EbudaLd}
\big|\Omega^{\leq N}\alphab[G]\big|\approx& \sum_{\chi(\bOmega)\leq N}\big|\alphab[\LD_\bOmega G]\big|,
\end{Eq}
and similar for that which $\alphab$ replaced by $\rho$, $\sigma$ or $\alpha$.


\subsection{The multiplier vector fields and energy equality}\label{Tmvfaee}
For any closed $\rR$-valued $2$-form $G$ and $\rC$-valued scalar field $f$, 
their momentum $2$-tensor $T[G,f]$ in $(\rR^{1+3},\mm)$ is defined by
\begin{Eq*}
T[G,f]_{\mu\nu}:=G_{\mu\gamma}G_{\nu}{}^\gamma-\frac{1}{4}m_{\mu\nu}G_{\gamma\iota}G^{\gamma\iota}+\Re(\overline{D_\mu f}D_\nu f)-\frac{1}{2}m_{\mu\nu}\overline{D^\gamma f}D_\gamma f.
\end{Eq*}
Given an $\rR$-valued function $\chi$ and $\rR$-valued vector fields $W$ and $X$, we can define the associated current,
\begin{Eq*}
{}^{(X)}\tilde J[G,f]_\mu:=&T[G,f]_{\mu\nu}X^\nu-\frac{1}{2}|f|^2\partial_\mu\chi+\frac{1}{2}\chi\partial_\mu(|f|^2)+W_\mu.
\end{Eq*}
Then, we are able to calculate that
\begin{Eq*}
\Div\kl( {}^{(X)}\tilde J[G,f]\kr)
=&X_\nu G^{\nu\gamma}\nabla^\mu G_{\mu\gamma}+\Re(\overline{\square_A f}(D_Xf+\chi f))+X^\nu\Im(f\overline{D^\gamma f})F_{\nu\gamma}\\
&+T[G,f]_{\mu\nu}{}^{(X)}\pi^{\mu\nu}-\frac{1}{2}|f|^2\square\chi+\chi D^\mu f\overline{D_\mu f}+\nabla^\mu W_\mu, 
\end{Eq*}
where ${}^{(X)}\pi_{\mu\nu}:=2^{-1}\LD_X\mm_{\mu\nu}$ is the deformation tensor of $X$. 
Integrating it in a given domain $\rD$, by \emph{Gauss}' law, we get the energy equality
\begin{Eq}\label{Eq:Ee}
\int_\rD \Div\kl( {}^{(X)}\tilde J[G,f]\kr)=\int_{\partial\rD}{}^{(X)}\tilde J[G,f]_{\vec n} 
\end{Eq}
where $\vec n$ is the outer normal vector of $\partial\rD$.

In this paper, 
we will use a special kind of multiplier vector fields.
Consider $X=2^{-1}\kappa K$ with $\kappa=\kappa(u,v)$ which will be fixed later,
$\chi=t\kappa$ and $W=(2r)^{-1}|f|^2\kl(u^2L\kappa\Lb-v^2\Lb\kappa L\kr)$. 
We can calculate that 
\begin{Eq}\label{Eq:Eet}
{}^{(X)}\pi_{\Lb\Lb}=&-2v^2\Lb\kappa,\quad{}^{(X)}\pi_{L\Lb}=-2t\kappa-2^{-1}K\kappa,\quad{}^{(X)}\pi_{LL}=-2u^2L\kappa,\quad{}^{(X)}\pi_{e_ie_j}=t\kappa\sm_{ij},\\
{}^{(X)}\tilde J[G,f]_\Lb
=&\kappa\cdot\kl(u^2(|\alphab[G]|^2+|\hD_{\Lb} f|^2)+v^2(|\rho[G]|^2+|\sigma[G]|^2+|\sD f|^2)\kr)+ \frac{1}{2r^2}\Lb\kl(r(u^2+v^2)\kappa|f|^2\kr),\\
{}^{(X)}\tilde J[G,f]_L
=&\kappa\cdot\kl(u^2(|\rho[G]|^2+|\sigma[G]|^2+|\slashed D f|^2)+v^2(|\alpha[G]|^2+|\hD_{L} f|^2)\kr)- \frac{1}{2r^2}L\kl(r(u^2+v^2)\kappa|f|^2\kr),\\
\Div\kl( {}^{(X)}\tilde J[G,f]\kr)
=&\frac{1}{2}\kappa\cdot\kl(G_K{}^{\nu}\nabla^{\mu}G_{\mu\nu}+\Re(\overline{\square_A f}\hD_{K}f)+\Im(f\overline{D^\mu f})F_{K\mu}\kr)\\
&\quad-\frac{u^2L\kappa}{2}\cdot\kl(|\alphab[G]|^2+|\hD_{\Lb}f|^2\kr)-\frac{v^2\Lb\kappa}{2}\cdot\kl(|\alpha[G]|^2+|\hD_{L}f|^2\kr)\\
&\quad-\frac{K\kappa}{4}\cdot\kl(|\rho[G]|^2+|\sigma[G]|^2+|\sD f|^2\kr).
\end{Eq}

\subsection{Some induced formulas of the MKG equation}
In the null frame, for the scalar field, we have,
\begin{Eq}\label{Eq:MKGe-ud_phi}
D_LD_\Lb(r\phi)=\sD^2(r\phi)-\I\rho\cdot(r\phi).
\end{Eq}%
As for the \emph{Maxwell} field, as shown in Section 2.2 of \cite{MR4030741},
for any $\chi(\bZ)\leq 2$, we have
\begin{Eq}\label{Eq:MKGe-ud_FZ}
\begin{cases}
L(r^2\rho^{(\bZ)})=-\sdiv(r^2\alpha^{(\bZ)})+r^2J^{(\bZ)}_L,\\
\Lb(r^2\rho^{(\bZ)})=\sdiv(r^2\alphab^{(\bZ)})-r^2J^{(\bZ)}_\Lb,\\
L(r^2\sigma^{(\bZ)})=-\sdiv(r^2{}^*\alpha^{(\bZ)}),\quad 
\Lb(r^2\sigma^{(\bZ)})=-\sdiv(r^2{}^*\alphab^{(\bZ)}),\\
L(r\alphab^{(\bZ)}_i)=-\snabla_i(r\rho^{(\bZ)})+{}^*\snabla_i(r\sigma^{(\bZ)})+rJ^{(\bZ)}_{e_i},\\
\Lb(r\alpha^{(\bZ)}_i)=\snabla_i(r\rho^{(\bZ)})+{}^*\snabla_i(r\sigma^{(\bZ)})+rJ^{(\bZ)}_{e_i},
\end{cases}
\end{Eq}
where $({}^*\slashed \Xi)_j:=\slashed\Xi^i\svE_{ij}$.
Note that here $\rho$ can be freely replaced by $\rrho$ in the region $\{u\leq -1\}$.
Meanwhile, here $J^{(\bZ)}$ has the expansion formula
\begin{Eq}\label{Eq:Efo_JZ}
J_\mu
=&\Im(\phi\cdot \overline{ D_\mu\phi})=\Im(\phi\cdot \overline{ \hD_\mu\phi}),\\
J^{(Z)}_\mu
=&\Im(\phi^{(Z)}\cdot \overline{ \hD_\mu\phi})+\Im(\phi\cdot \overline{ \hD_\mu \phi^{(Z)}})-|\phi|^2F_{Z\mu},\\
J^{(Z_1,Z_2)}_\mu
=&\Im(\phi^{(Z_1,Z_2)}\cdot \overline{ \hD_\mu\phi})+\Im(\hD_{Z_2}\phi\cdot \overline{ \hD_\mu \phi^{(Z_1)}})+\Im(\phi^{(Z_1)}\cdot \overline{ \hD_\mu \phi^{(Z_2)}})+\Im(\phi\cdot \overline{ \hD_\mu\phi^{(Z_1,Z_2)}})\\
&-2\Re(\phi\cdot\overline{\phi^{(Z_1)}})F_{Z_2\mu}-2\Re(\phi\cdot\overline{\phi^{(Z_2)}})F_{Z_1\mu}-|\phi|^2F^{(Z_1)}_{Z_2\mu}-|\phi|^2F_{[Z_2,Z_1]\mu}.
\end{Eq}
As shown in that paper, we also know
\begin{Eq}\label{Eq:MKGe_phiZ}
\square_A\phi^{(\bZ)}=
\begin{cases}
0,&\bZ=\emptyset,\\
Q(\phi,F;Z),&\bZ=Z,\\
\begin{aligned}
&Q(\phi^{(Z_1)},F;Z_2)+Q(\phi^{(Z_2)},F;Z_1)\\
&~ +Q(\phi,F;[Z_1,Z_2])+Q(\phi,F^{(Z_1)};Z_2)\\
&~-2F_{Z_1\mu}F_{Z_2}{}^\mu\phi,
\end{aligned}&\bZ=(Z_1,Z_2),
\end{cases}
\end{Eq}
with 
\begin{Eq*}
Q(f,G;Z):=&2\I G_{\mu\nu}Z^\nu D^\mu f+\I\nabla^{\mu}(Z^\nu G_{\mu\nu})f.
\end{Eq*}
Here, for $u\leq -1$, the $Q(f,G;Z)$ can be controlled as follow:
\begin{Eq}\label{Eq:Eo_Q}
|u|^{-\chi_K(\bZ)}|Q(f,G;Z)|
\lesssim& \big(|u||(\Div G)_\Lb|+v|\slashed{\Div G}|+|u|^{-1}v^2|(\Div G)_L|\big)|f|\\
&+\big(|u||\rho[G]|+v|\alpha[G]|\big)|\hD_\Lb f|\\
&+\big(|u||\alphab[G]|+v|\sigma[G]|+|u|^{-1}v^2|\alpha[G]|\big)|\sD f|\\
&+\big(v|\alphab[G]|+|u|^{-1}v^2|\rho[G]|\big)|\hD_Lf|\\
&+\big(|\rho[G]|+|\sigma[G]|\big)|f|.
\end{Eq}

\section{The analysis in interior region}\label{Sn:Taiir}
\subsection{The conformal transform}
To begin with, we introduce a conformal transform result 
which will be used in the interior region.
The whole picture of this theory to MKG equations can be found in, e.g., 
Chapter 4 of \cite{MR2391586} or Section 6.1 of \cite{MR4030741}. 

\begin{proposition}\label{Pn:Toct}
Consider $\pi:(\rD,m)\rightarrow (\tilde \rD,\tilde m)$ is a conformal bijection between two domains of \emph{Minkowski} space with a scalar field $\Lambda$  that satisfying
\begin{Eq*}
\pi^{*}\tilde m=\Lambda^{-2}m,\qquad \square (\Lambda^{-1})=0.
\end{Eq*}
Then, $(A,\phi)$ is a solution of \eqref{Eq:MKGe-Fphi} on $(\rD,m)$ 
if and only if $(\tA,\tphi):=(\pi^{-1})^*(A,\Lambda\phi)$ is a solution of \eqref{Eq:MKGe-Fphi} on $(\tilde\rD,\tilde m)$.
\end{proposition}

In $(\tt,\tx)$ coordinate system, we define $\tu$, $\tL$, etc., similar to that in $(t,x)$ coordinate system.
Now, for any fixed $U_*$, 
we define $T_*:=2U_*-2^{-1}$ and $\Lambda:=(t-T_*)^2-|x|^2$.
It is easy to check that the map $\pi$ defined by
\begin{Eq*}
\pi:\quad&D_{U_*,\infty}^{U_*,\infty}=\{(t,x):u\geq U_*\}&\rightarrow \quad&\trD_{0,1}^{0,1}=\{(\tt,\tx):\tu>0,\tv\leq 1\},\\
&\qquad\quad(t,x)&\mapsto\quad&\qquad(\tt,\tx)=\Lambda^{-1}(t-T_*,x)
\end{Eq*}
satisfies the requirements of \Pn{Pn:Toct}. 
Meanwhile, $\pi$ maps the $u\geq U_*$ part of scattering data
to the cone data on $\trH_{0}^{0,1}$ (see the figure below).
\begin{figure}[H]
\centering
\includegraphics[width=0.7\textwidth]{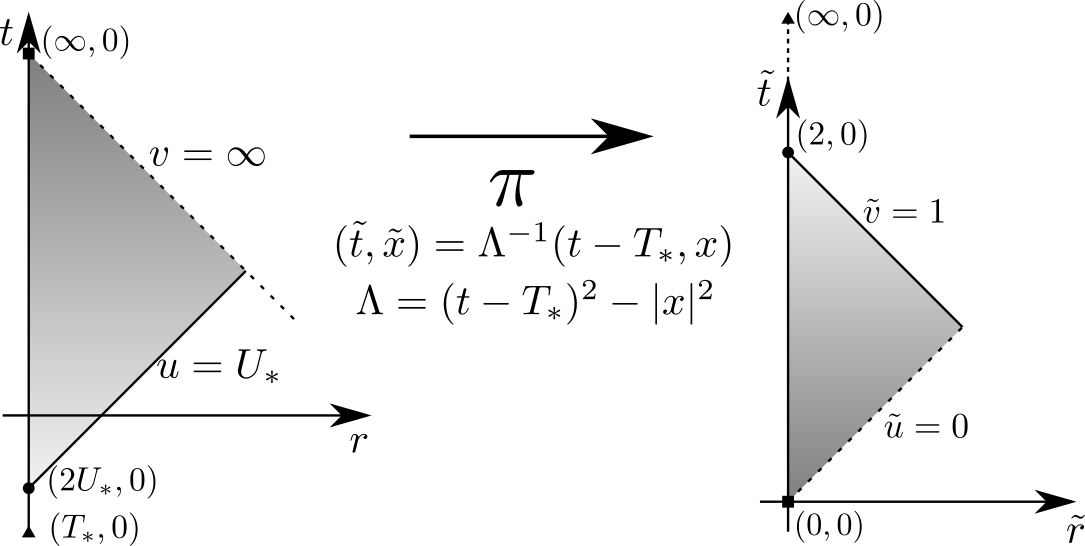}
\end{figure}
Thus, to construct the solution of \eqref{Eq:MKGe-Fphi} in $D_{+;U_*,\infty}^{U_*,\infty}$ with any given $U_*$, 
we only need to construct the solution of \eqref{Eq:MKGe-Fphi} in $\trD_{0,1}^{0,1}$ with the corresponding data on $\trH_{0}^{0,1}$.

\subsection{The equation on $\trD_{0,1}^{0,1}$ and its data}
To solve the MKG equation in $\trD_{0,1}^{0,1}$, we need to temporarily choose a gauge 
so that $\trH_{0}^{0,1}$ becomes a characteristic hypersurface.
For this purpose, we adopt the \emph{Lorenz} gauge, that is, 
\begin{Eq}\label{Eq:gci_tD}
\tilde\lambda=\tpartial^\mu \tA_\mu\equiv 0, \qquad ~for~(\tt,\tx)\in \tD_{0,1}^{0,1}.
\end{Eq}
In this gauge, \eqref{Eq:MKGe-Fphi} reduce to the semilinear wave system
\begin{Eq}\label{Eq:MKGe-swe}
\begin{cases}
\tsquare \tA_\mu=-\tJ_\mu=-\Im(\tphi\cdot\overline{\tpartial_\mu \tphi})+\tA_\mu|\tphi|^2,\\
\tsquare \tphi=-2\I \tA^\mu\tpartial_\mu\tphi+\tA^\mu \tA_\mu\tphi.
\end{cases}
\end{Eq}

Conversely, once $(\tA,\tphi)$ solves \eqref{Eq:MKGe-swe}, then $\tilde\lambda$ satisfies the wave equation
\begin{Eq*}
\tsquare\tilde \lambda=\tilde\lambda|\tphi|^2.
\end{Eq*}
This means $\tilde\lambda\equiv 0$ in $\trD_{0,1}^{0,1}$ as long as $\tilde\lambda=0$ on $\trH_{0}^{0,1}$.

Next, we construct the initial data of $(\tA,\tphi)$ on $\trH_0^{0,1}$.
To begin with, we introduce some relations under mapping $\pi$. 
Defining $u_*=\frac{t-T_*-r}{2}$ and $v_*=\frac{t-T_*+r}{2}$, we have
\begin{Eq}\label{Eq:Rbofatf}
&\tu=(4v_*)^{-1},&\quad &\tv=(4u_*)^{-1}, &\quad& \Lambda=4u_*v_*=\tr^{-1}r,\\
&\tLb=-4v_*^2L,&\quad& \tL=-4u_*^2\Lb,&\quad& \te_i=4u_*v_* e_i,&\quad&\tOmega_{ij}=\Omega_{ij}.
\end{Eq}
Then, we are able to calculate
\begin{Eq*}
\talpha_i(0,\tv,\tomega)=&\lim_{\tu\rightarrow 0}\tF_{\tL\te_i}(\tu,\tv,\tomega)=-16u_*^3\lim_{v\rightarrow\infty}[v_*F_{\Lb e_i}](u,v,\omega)=-16u_*^3\Alphab_i(u,\omega),\\
\tphi(0,\tv,\tomega)=&\tv^{-1}\lim_{\tu\rightarrow 0}[\tr\tphi](\tu,\tv,\tomega)=4u_*\lim_{v\rightarrow \infty}[r\phi](u,v,\omega)=4u_*\Phi(u,\omega).
\end{Eq*}
Now, using \eqref{Eq:Rbofatf} and noticing $\|(\Alphab,\Phi)\|_{S\!N_{U_*}}<\infty$, we have
\begin{Eq}\label{Eq:Eo_talpha_o_tH}
&\sum\Low{n\leq 2}\hspace{0pt}\int_0^1\int_{\rS^2}\tv^{-2+2n}\kl|\tL^n\tOmega^{\leq 6-2n} \talpha\kr|_{\tu=0}^2\d\tomega\d \tv\\
\lesssim&\sum\Low{n\leq 2}\int_{U_*}^\infty \int_{\rS^2} u_*^{2-2n}\kl|(u_*^2\Lb)^n\Omega^{\leq 6-2n} (u_*^3\Alphab)\kr|^2\d\omega\frac{\d u}{u_*^2}\\
\lesssim&\sum\Low{n\leq 2}\int_{U_*}^\infty \int_{\rS^2} \kl<u\kr>^{6+2n}\kl|\Lb^n\Omega^{\leq 6-2n} \Alphab\kr|^2\d\omega\d u<\infty,
\end{Eq}
and similarly
\begin{Eq}\label{Eq:Eo_tsAphi_o_tH}
\sum_{n\leq 3}\int_0^1\int_{\trS^2}\tv^{-4+2n}\kl|\tD_\tL^n\tD_\tOmega^{\leq 7-2\Ma{n}}\tphi\kr|_{\tu=0}^2\d\tomega\d \tv<\infty.
\end{Eq}

Armed with these estimates,
we make the following claim.
\begin{claim}\label{Cm:Eo_tAphi_o_tH}
Assume that \eqref{Eq:Eo_talpha_o_tH} and \eqref{Eq:Eo_tsAphi_o_tH} are satisfied.
Then, in \emph{Lorenz} gauge, there exists a collection $(\tA,\tphi)$ on $\trH_0^{0,1}$  corresponding to $(\talpha,\tphi)$, which satisfying $\tA_\tL=0$ and 
\begin{Eq}\label{Eq:Eo_tAphi_o_tH}
\sum\Low{|\alpha|\leq 1\\|\beta|\leq 2}\int_{\trH_{0}^{0,1}}\kl|(\tspartial,\tL)^\alpha\tpartial^\beta\big(\tA,\tphi\big)\kr|^2<\infty.
\end{Eq}
Moreover, this collection is the unique collection (up to a gauge transform under \emph{Lorenz} gauge) with respect to $(\talpha,\tphi)$.
\end{claim}
\begin{remark}
Here, $\tA$ is in \emph{Lorenz} gauge means that it satisfies the compatibility condition \eqref{Eq:c_eq_ALb} which in turn comes from \eqref{Eq:gci_tD} and \eqref{Eq:MKGe-swe}.
The proof of this Claim will be postponed to \Sn{Sn:PoCEo_tAphi_o_tH}.
\end{remark}

\subsection{The preparation for local existence}
In this and the next two subsections, we omit the tilde on $\tL$, $\tA$, etc., 
since that we will only discuss in $(\tt,\tx)$ coordinate system.
As we discussed in the \Sn{Sn:I}, there are already several results of the local existence of wave equation with characteristic data. 
However, to fit our frame and for the convenience of readers, 
we give the detailed process to construct the local solution of \eqref{Eq:MKGe-swe} 
with characteristic data satisfying \eqref{Eq:Eo_tAphi_o_tH}.

To begin with, we transfer \eqref{Eq:MKGe-swe} to a unified system
\begin{Eq}\label{Eq:Eo_Psi}
\begin{cases}
\square \Psi^I=C_{JK}^{I\mu}\Psi^J\partial_\mu \Psi^K+C_{JKL}^I\Psi^J\Psi^K\Psi^L,\\
\Psi|_{\rH_0^{0,1}}=\Psi_{(data)},
\end{cases}
\end{Eq}
with $\Psi:=(A,\phi)$ satisfying \eqref{Eq:Eo_tAphi_o_tH}.
In the beginning, we assume $\Phi_{(data)}\in C_0^\infty(\rH_0^{0,1}\backslash \rS_0^0)$.

Then, we construct the iterative system
\begin{Eq}\label{Eq:c_it_Phi}
\begin{cases}
\square \Psi_{j+1}^I=C_{JK}^{I\mu}\Psi_{j}^J\partial_\mu \Psi_{j}^K+C_{JKL}^I\Psi_{j}^J\Psi_{j}^K\Psi_{j}^L,\\
\Psi_{j+1}|_{\rH_0^{0,1}}=\Psi_{(data)},
\end{cases}
\end{Eq}
with $j\in\rN_0$ and $\Psi_0:=0$.
Using the result shown in \cite{MR0460898} and the iteration,
we know $\Psi_j$ is smooth in $\rD_{0,1}^{0,1}$ for any $j$.

Next, we are going to show that $\{\Psi_j\}$ converges in $\rD_{0,\delta}^{0,1}$ with some $\delta$ small enough.
For this purpose, for $U\leq\delta\ll1$, $V\leq1$ and $n=0,1,2$, we define
\begin{Eq*}
\varE^{(n)}[\Psi](U,V):=&\sup_{0\leq u\leq U}\|(L,\spartial)\partial^{\leq n}\Psi\|_{L^2(\rH_{u}^{0,V})}^2+\sup_{0\leq v\leq V}\|(\Lb,\spartial)\partial^{\leq n}\Psi\|_{L^2(\rHb_{0,U}^{v})}^2,\\
\varE_0^{(n)}[\Psi](U,V):=&\sup_{0\leq u\leq U}\|\partial^{\leq n}\Psi\|_{L^2(\rH_{u}^{0,V})}^2.
\end{Eq*}
Following \eqref{Eq:Eo_tAphi_o_tH}, we know that there exists a constant $M_1\gg 1$, such that for any $j$, there is
\begin{Eq}\label{Eq:ci_EE0_j}
\varE_0^{(2)}[\Psi_j](0,1)+\varE^{(2)}[\Psi_j](0,1)\leq M_1^2.
\end{Eq}
We also find that
\begin{Eq*}
\|\partial^{\leq n}\Psi\|_{L^2(\rH_{U}^{0,V})}^2-\|\partial^{\leq n}\Psi\|_{L^2(\rH_{0}^{0,V})}^2
\lesssim& \int_{\rD_{0,U}^{0,V}}|\partial_t\partial^{\leq n}\Psi||\partial^{\leq n}\Psi|\\
\lesssim& \int_0^V \|\Lb\partial^{\leq n}\Psi\|_{L^2(\rHb_{0,U}^{v})}^2\d v+\int_0^U\|(L,1)\partial^{\leq n}\Psi\|_{L^2(\rH_{u}^{0,V})}^2\d u\\
\lesssim& \varE^{(n)}[\Psi](U,V)+\delta \varE_0^{(n)}[\Psi](U,V),
\end{Eq*}
which means 
\begin{Eq}\label{Eq:c_E0}
\varE_0^{(n)}[\Psi](U,V)\lesssim \varE_0^{(n)}[\Psi](0,V)+\varE^{(n)}[\Psi](U,V).
\end{Eq}

\subsection{The uniformly boundedness of $\Psi_j$ in $\rD_{0,\delta}^{0,1}$}
In this subsection we will show that, 
for some $M_2\gg 1$ and $\delta$ to be fixed later, 
we have
\begin{Eq}\label{Eq:c_E_j}
\varE^{(2)}[\Psi_j](U,V)< 4M_1^2e^{2M_2V}
\end{Eq}
for any $j$, $U\leq \delta$ and $V\leq 1$.

It is obvious that \eqref{Eq:c_E_j} holds for $j=0$. Next, we assume that it holds for some specific $j$. 
Now, using \eqref{Eq:c_E0} and \eqref{Eq:ci_EE0_j}, we know
\begin{Eq}\label{Eq:c_E0_j}
\varE_0^{(2)}[\Psi_j](U,V)\lesssim M_1^2e^{2M_2V}.
\end{Eq}
Meanwhile, using Sobolev inequality on cone, \eqref{Eq:c_E_j} and \eqref{Eq:c_E0_j}, we get that
\begin{Eq}\label{Eq:c_Dinf_j}
\|\partial^{\leq 1}\Psi_j\|_{L^\infty(\rD_{0,U}^{0,V})}\leq& \sup_{0\leq u\leq U}\|\partial^{\leq 1}\Psi_j\|_{L^\infty(\rH_{u}^{0,V})}\\
\lesssim& \sup_{0\leq u\leq U}\|(L,\spartial)^{\leq 2}\partial^{\leq 1}\Psi_j\|_{L^2(\rH_{u}^{0,V})}\\
\lesssim&\varE^{(2)}[\Psi_j](U,V)^{\frac{1}{2}}+\varE_0^{(2)}[\Psi_j](U,V)^{\frac{1}{2}}\\
\lesssim&M_1e^{M_2V}.
\end{Eq}
Thus, once $\delta\lesssim e^{-M_2}$, we have
\begin{Eq}\label{Eq:c_inf_j}
\|\Psi_j\|_{L^\infty(\rD_{0,U}^{0,V})}\leq& \|\Psi_j\|_{L^\infty(\rH_{0}^{0,V})}+U\|\Lb\Psi_j\|_{L^\infty(\rD_{0,U}^{0,V})}\\
\lesssim& \|(L,\spartial)^{\leq 2}\Psi_j\|_{L^2(\rH_{0}^{0,V})}+U M_1e^{M_2V}\\
\lesssim&\varE_0^{(2)}[\Psi_j](0,V)^{\frac{1}{2}}+\delta M_1e^{M_2V}\\
\lesssim& M_1.
\end{Eq}

Then, using the standard energy estimate of \eqref{Eq:c_it_Phi}, we know
\begin{Eq*}
\varE^{(2)}[\Psi_{j+1}](U,V)-\varE^{(2)}[\Psi_{j+1}](0,V)
\leq& C\int_{\rD_{0,U}^{0,V}}\kl|\partial^{\leq 2}\kl(\Psi_j\partial\Psi_j,\Psi_j^3\kr)\kr|\kl|\partial_t\partial^{\leq 2}\Psi_{j+1}\kr|\\
\leq& C_1\int_{\rD_{0,U}^{0,V}}\kl(\underbrace{|\Lb\partial^{\leq 2}\Psi_j||\Psi_j|}_{I_1}+\underbrace{|(L,\spartial)^{\leq 1}\partial^{\leq 2}\Psi_j||\partial^{\leq 1}\Psi_j|\kl<\Psi_j\kr>}_{I_2}\kr)\\
&\phantom{C_1\int_{\rD_{0,U}^{0,V}}}\quad\times \kl(\underbrace{|\Lb\partial^{\leq 2}\Psi_{j+1}|}_{I_3}+\underbrace{|(L,\spartial)\partial^{\leq 2}\Psi_{j+1}|}_{I_4}\kr)
\end{Eq*}
for some specific $C_1$.
Among them, using \eqref{Eq:c_E_j} and \eqref{Eq:c_inf_j}, we calculate
\begin{Eq*}
\int_{\rD_{0,U}^{0,V}} I_1I_3 \leq& \|\Lb\partial^{\leq 2}\Psi_j\|_{L_v^1\kl((0,V); L^2(\rHb_{0,U}^v)\kr)} \|\Psi_j\|_{L^\infty(\rD_{0,U}^{0,V})}\|\Lb\partial^{\leq 2}\Psi_{j+1}\|_{L_v^\infty\kl((0,V); L^2(\rHb_{0,U}^v)\kr)}\\
\leq &C \|M_1e^{M_2v}\|_{L_v^1(0,V)}\cdot M_1\cdot \varE^{(2)}[\Psi_{j+1}](U,V)^{\frac{1}{2}}\\
\leq &C M_2^{-2}M_1^4e^{2M_2V}+\frac{1}{8C_1}\varE^{(2)}[\Psi_{j+1}](U,V).
\end{Eq*}
Using \eqref{Eq:c_E_j}, \eqref{Eq:c_E0_j}, \eqref{Eq:c_Dinf_j} and \eqref{Eq:c_inf_j}, we see
\begin{Eq*}
\int_{\rD_{0,U}^{0,V}} I_2I_3 \leq& \|(L,\spartial)^{\leq 1}\partial^{\leq 2}\Psi_j\|_{L_u^2\kl((0,U);L^2(\rH_{u}^{0,V})\kr)}\|\partial^{\leq 1}\Psi_j\|_{L^\infty(\rD_{0,U}^{0,V})}\\
&\times \|\kl<\Psi_j\kr>\|_{L^\infty(\rD_{0,U}^{0,V})}\|\Lb\partial^{\leq 2}\Psi_{j+1}\|_{L_v^2\kl((0,V); L^2(\rHb_{0,U}^v)\kr)}\\
\leq &CU^{1/2}M_1e^{M_2V}\cdot M_1e^{M_2V}\cdot M_1\cdot \varE^{(2)}[\Psi_{j+1}](U,V)^{\frac{1}{2}}\\
\leq &C\delta M_1^6e^{4M_2V}+\frac{1}{8C_1}\varE^{(2)}[\Psi_{j+1}](U,V).
\end{Eq*}
Similarly, we have
\begin{Eq*}
\int_{\rD_{0,U}^{0,V}} I_1I_4 
\leq& C\delta M_2^{-1}M_1^4e^{2M_2V}+\frac{1}{8C_1}\varE^{(2)}[\Psi_{j+1}](U,V),\\
\int_{\rD_{0,U}^{0,V}} I_2I_4 
\leq &C\delta^2 M_1^6e^{4M_2V}+\frac{1}{8C_1}\varE^{(2)}[\Psi_{j+1}](U,V).
\end{Eq*}

Mixing these estimates together, we get that
\begin{Eq*}
\varE^{(2)}[\Psi_{j+1}](U,V)\leq M_1^2+C_2\kl(M_2^{-2}M_1^4e^{2M_2V}+\delta M_1^6e^{4M_2V}\kr)+\frac{1}{2}\varE^{(2)}[\Psi_{j+1}](U,V)
\end{Eq*}
always holds for some specific $C_2$.
Now, we fix $M_2=(2C_2)^{\frac{1}{2}}M_1$ and $\delta=(2C_2)^{-1}M_1^{-4}e^{-2M_2}$. 
It can be derived from the above inequality that \eqref{Eq:c_E_j} holds for $j+1$.
This finishes the proof of uniformly boundedness by iteration.

\subsection{The existence and uniqueness in $\rD_{0,\delta}^{0,1}$}
In this step, we will show that the sequence $\{\Psi_j\}$ is convergent.
Using the energy method again, we know
\begin{Eq*}
&\varE^{(0)}[\Psi_{j+1}-\Psi_j](U,V)\\
\lesssim &\int_{\rD_{0,U}^{0,V}} \kl|\kl(\Psi_j\partial\Psi_j-\Psi_{j-1}\partial\Psi_{j-1},\Psi_j^3-\Psi_{j-1}^3\kr)\kr|\kl|\partial_t(\Psi_{j+1}-\Psi_j)\kr|\\
\lesssim &\int_{\rD_{0,U}^{0,V}} \bigg({\kl|\Lb^{\leq 1}(\Psi_j-\Psi_{j-1})\kr|\kl|\partial^{\leq 1}(\Psi_j,\Psi_{j-1})\kr|\kl<\Psi_j,\Psi_{j-1}\kr>}+{\kl|(L,\spartial)(\Psi_j-\Psi_{j-1})\kr|\kl|(\Psi_j,\Psi_{j-1})\kr|}\bigg)\\
&\phantom{\int_{\rD_{0,U}^{0,V}}}\quad\times \kl({\kl|\Lb(\Psi_{j+1}-\Psi_j)\kr|}+{\kl|(L,\spartial)(\Psi_{j+1}-\Psi_j)\kr|}\kr).
\end{Eq*}

Similarly to the process in last subsection, using \eqref{Eq:c_E0} to \eqref{Eq:c_inf_j}
and noticing $\delta$ is sufficiently small,
we know
\begin{Eq*}
\varE^{(0)}[\Psi_{j+1}-\Psi_j](U,V)\leq&  M_3\int_0^V\varE^{(0)}[\Psi_{j}-\Psi_{j-1}](U,v)\d v+\frac{1}{2} \varE^{(0)}[\Psi_{j}-\Psi_{j-1}](U,V)
\end{Eq*}
with some $M_3\geq 4M_1^2 e^{2M_2}\geq \varE^{(0)}[\Psi_{1}-\Psi_{0}](U,V)$ and large enough. By iteration, we know
\begin{Eq*}
\varE^{(0)}[\Psi_{j+1}-\Psi_j](U,V)\leq& M_3\sum_{i=0}^j\frac{1}{i!}2^{i-j}(VM_3)^{i},\\
\sum_{j=0}^\infty \varE^{(0)}[\Psi_{j+1}-\Psi_j](U,V)\leq &M_3\sum_{i=0}^\infty\sum_{j=i}^\infty \frac{1}{i!}2^{i-j}(VM_3)^{i}
\leq 2M_3e^{VM_3}.
\end{Eq*}

This shows that $\{\Psi_j\}$ is convergent, 
and gives the existence result in $\rD_{0,\delta}^{0,1}$ with $C_0^\infty$ data. 
By density, this also gives the existence result with general data.
A similar process also gives the uniqueness of such solution. 
We left the detailed proof to the interested reader.
\begin{remark}
The processes in these three subsections can be transplanted to the local existence of more general wave equation $\square \Psi=f(\Psi,\partial\Psi)$ with smooth $f$. 
The only difference is that in \eqref{Eq:c_inf_j} we need to control $\|\partial^{\leq 1}\Psi_j\|_{L^\infty}$.
So, we should require $\|(\spartial,L)^{\leq 1}\partial^{\leq 3}\Psi\|_{L^2(H_0^{0,1})}<\infty$ compared to \eqref{Eq:Eo_tAphi_o_tH}.
\end{remark}

\subsection{The existence and uniqueness in $\rD_{0,1}^{0,1}$}
The next step is to construct the solution in whole $\trD_{0,1}^{0,1}$.
The key point here is that we have already known the global existence result of MKG equation with the \emph{Cauchy} data.
Meanwhile, we find that $\pi^{-1}$ maps the hyperboloid 
$\tilde\varC:=\{\tt/(\tt^2-\tr^2)=(2\delta)^{-1},~\tv\leq 1\}\subset \trD_{0,\delta}^{0,1}$
to the hyperplane $\varC:=\{t=(2\delta)^{-1}+T_*,~u\geq U_*\}$.
\begin{figure}[H]
\centering
\includegraphics[width=0.7\textwidth]{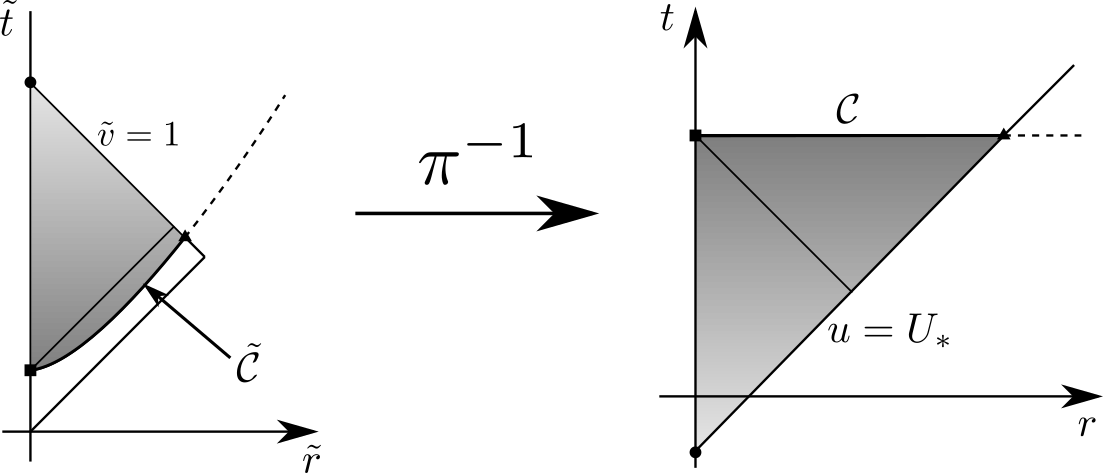}
\end{figure}
Now, notice that have we already known $(\tA,\tphi)$ and their derivatives on $\tilde\varC$. 
By mapping $\pi^{-1}$, this gives the initial data of \eqref{Eq:MKGe-Fphi} on $\varC$.
Meanwhile, since both $\tilde\varC$ and $\varC$ are bounded regions in their coordinates,
it is easy to verify that all needed energy on $\varC$ is finite.
Then, according to the theory showed in \cite{MR1271462}, 
there exists a unique solution in $\{t\leq(2\delta)^{-1}+T_*,~u\geq U_*\}$, 
and thus a unique solution in $\{\tt/(\tt^2-\tr^2)\leq(2\delta)^{-1},~\tv\leq 1\}$.

By splicing the solution in this part with the solution in $\trD_{0,\delta}^{0,1}$,
we get the solution in $\trD_{0,1}^{0,1}$, 
where the uniqueness follows from the uniqueness in each region.

\subsection{The proof of \Tm{Tm:Main1}}
Finally, notice that the above process gives a unique solution $(\tA,\tphi)$ under \emph{Lorentz} gauge.
It means that its corresponding $(A,\phi)$ is a unique solution in $\rD_{+;U_*,\infty}^{U_*,\infty}$ under some corresponded gauge which depend on $U_*$.
It also means that, up to a gauge transform, $(A,\phi)$ is the unique solution in $\rD_{+;U_*,\infty}^{U_*,\infty}$.

Now, passing $U_*\rightarrow\infty$, we reach the solution in $\rR_+\times \rR^3$,
and finishes the proof of \Tm{Tm:Main1}.

\subsection{Additional discussion on $\rH_{U_*}^{-U_*,\infty}$}\label{Sn:Ado_H}
In order to give the boundary energy in the exterior region, 
that is, the energy on $\rH_{U_*}^{-U_*,\infty}$, 
we give some additional discussions.

Here we consider $U_*\leq -1$ as a constant. 
Then, on $\rH_{U_*}^{-U_*,\infty}$ we have $r\approx v\approx v_*$.
The image of $\rH_{U_*}^{-U_*,\infty}$ under mapping $\pi$ is $\trHb_{0,\tU_*}^1$ with $\tU_*:=(1-8U_*)^{-1}$.

\part[Some preparation]
To begin with, we mention that under mapping $\pi$, we have \eqref{Eq:Rbofatf}. Thus we have
\begin{Eq}\label{Eq:Rb_Z_a_tZ}
T=-\tK,\quad
S=-(\tS+T_*\tK),\quad
K=-(\tT+2T_*\tS+T_*^2\tK).
\end{Eq}
For the writing convenience, in this subsection we set $\tPsi=(\tA_0,\tA_1,\tA_2,\tA_3,\tphi)$ and
\begin{Eq*}
I_1:=&\int_0^{\tU_*}\int_{\trS^2}|(\tLb,\tspartial)\tpartial^{\leq 2}\tPsi|_{\tv=1}^2\d\tomega\d\tu,\quad&
I_2:=&\int_0^{\tU_*}\int_{\trS^2}|\tpartial^{\leq 2}\tPsi|_{\tv=1}^2\d\tomega\d\tu,\\
I_3:=&\sup_{\tu\in (0,\tU_*]}\int_{\trS^2}|\tpartial^{\leq 2}\tPsi|_{\tv=1}^2\d\tomega.
\end{Eq*}
Here, noticing $\tr\approx 1$ on $\trHb_{0,\tU_*}^1$, by the discussions in last few subsections, we easily know 
\begin{Eq*}
I_1\lesssim \|(\tLb,\tspartial)\tpartial^{\leq 2}\tPsi\|_{L^2(\trHb_{0,1}^1)}^2<\infty.
\end{Eq*}
For $I_2$ and $I_3$, using \eqref{Eq:key_Eob1_2} and \eqref{Eq:key_Eob3_2}, we also know
\begin{Eq*}
I_2+I_3\lesssim I_1+\kl.\int_{\trS^2}|\tpartial^{\leq 2}\tPsi|\d\tomega\kr|_{\tu=0,\tv=1}^2\lesssim I_1+\|\tL^{\leq 1}\tpartial^{\leq 2}\tPsi\|_{L^2(\trH_0^{0,1})}^2<\infty.
\end{Eq*}

\part[The energy corresponding to $\rho$, $\sigma$ and $\alpha$]
For $\rho$, using \eqref{Eq:Rbofatf} and \eqref{Eq:Rb_Z_a_tZ}, under mapping $\pi$, we have
\begin{Eq*}
\sum_{\chi(\bZ)\leq 2}|r^2\rho^{(\bZ)}|=\sum_{\chi(\bZ)\leq 2}|r^2\Lb^\mu L^\nu\LD_\bZ F_{\mu\nu}|
\approx\sum_{\chi(\tbZ)\leq 2}|\tr^2\tL^\mu \tLb^\nu\LD_\tbZ \tF_{\mu\nu}|
\lesssim|\tpartial^{\leq 2}\tL\tA_\tLb|+|\tpartial^{\leq 2}\tLb\tA_\tL|.
\end{Eq*}
Using the \emph{Lorentz} gauge condition \eqref{Eq:Ro_LbAL_ug},
the relation between the null frame and the \emph{Cartesian} frame, 
and noticing $\tr\approx 1$,
we find
\begin{Eq*}
\sum_{\chi(\bZ)\leq 2}|r^2\rho^{(\bZ)}|\lesssim |(\tLb,\tspartial)^{\leq 1}\tpartial^{\leq 2}\tPsi|.
\end{Eq*}
Then, we have
\begin{Eq*}
\sum_{\chi(\bZ)\leq 2}\int_{\rH_{U_*}^{-U_*,\infty}}|\rho^{(\bZ)}|^2
\approx\sum_{\chi(\bZ)\leq 2}\int_{-U_*}^\infty\int_{\rS^2}|r^2\rho^{(\bZ)}|_{u=U_*}^2\d\omega\frac{\d v}{v_*^2}
\lesssim I_1+I_2<\infty.
\end{Eq*}

The energy estimate for $\sigma$ and $\alpha$ is similar.
At first we can calculate
\begin{Eq*}
\sum_{\chi(\bZ)\leq 2}|r^2\sigma^{(\bZ)}|
\lesssim|\tspartial^{\leq 1}\tpartial^{\leq 2}\tPsi|,\qquad
\sum_{\chi(\bZ)\leq 2}|rv_*^2\alpha^{(\bZ)}|
\lesssim|(\tLb,\tspartial)^{\leq 1}\tpartial^{\leq 2}\tPsi|.
\end{Eq*}
Then, we know
\begin{Eq*}
\sum_{\chi(\bZ)\leq 2}\int_{\rH_{U_*}^{-U_*,\infty}}\kl(|u|^2|\sigma^{(\bZ)}|^2+v^2|\alpha^{(\bZ)}|^2\kr)<\infty.
\end{Eq*}

\part[The energy corresponding to $\hD_L\phi$ and $\sD\phi$]
For $\hD_L\phi$, under mapping $\pi$, we have
\begin{Eq*}
\sum_{\chi(\bZ)\leq 2}|rv_*^2\hD_L\phi^{(\bZ)}|=\sum_{\chi(\bZ)\leq 2}|v_*^2D_LD_\bZ(r\phi)|
\approx\sum_{\chi(\tbZ)\leq 2}|\tD_\tLb \tD_\tbZ(\tr\tphi)|
\lesssim|\tLb\tpartial^{\leq2}\tPsi|+|\tpartial^{\leq 2}\tPsi|\langle \tPsi\rangle^3+\kl|\tpartial^{\leq 1}\tPsi\kr|^2.
\end{Eq*}
Then, similar to above and using \emph{H\"older} and \emph{Sobolev} inequality, we have
\begin{Eq*}
\sum_{\chi(\bZ)\leq 2}\int_{\rH_{U_*}^{-U_*,\infty}}v^2|\hD_L\phi^{(\bZ)}|^2
\lesssim& \int_0^{\tU_*}\int_{\trS^2}\kl(|\tLb\tpartial^{\leq2}\tPsi|+|\tpartial^{\leq 2}\tPsi|\langle \tPsi\rangle^3+\kl|\tpartial^{\leq 1}\tPsi\kr|^2\kr)_{\tv=1}^2\d\tomega\d\tu\\
\lesssim& I_1+I_2(1+I_3)^3+I_2^2<\infty.
\end{Eq*}
Similarly, for $\sD\phi$, we have
\begin{Eq*}
\sum_{\chi(\bZ)\leq 2}\int_{\rH_{U_*}^{-U_*,\infty}}|\sD\phi^{(\bZ)}|^2<\infty.
\end{Eq*}

\part[The energy corresponding to $J_L$ and $\sJ$]
For $J_L$, similar to above and using \eqref{Eq:CofdaZ}, we know
\begin{Eq*}
\sum_{\chi(\bZ)\leq 2}|r^2v_*^2J^{(\bZ)}_L|=&\sum_{\chi(\bZ)\leq 2}\kl|v_*^2 L^\mu \LD_\bZ\Im\big((r\phi)\cdot  \overline{D(r\phi)}\big)_\mu\kr|
\approx\sum_{\chi(\tbZ)\leq 2}\kl|\tLb^\mu\LD_{\tbZ}\Im\kl((\tr\tphi)\cdot \overline{\tD(\tr\tphi)}\kr)_\mu\kr|\\
\lesssim&\kl(|\tLb^{\leq 1}\tpartial^{\leq2}\tPsi||\tPsi|+|\tLb\tpartial^{\leq 1}\tPsi||\tpartial^{\leq 1}\tPsi|+|\tLb\tPsi||\tpartial^{\leq 2}\tPsi|\kr)\langle \tPsi\rangle.
\end{Eq*}
Then, we similarly have
\begin{Eq*}
\sum_{\chi(\bZ)\leq 2}\int_{\rH_{U_*}^{-U_*,\infty}}v^4|J^{(\bZ)}_L|^2
\lesssim(I_1+I_2)I_3(1+I_3)<\infty.
\end{Eq*}
Similarly, for $\sJ$, we know
\begin{Eq*}
\sum_{\chi(\bZ)\leq 2}\int_{\rH_{U_*}^{-U_*,\infty}}v^2|\sJ^{(\bZ)}|^2<\infty.
\end{Eq*}

\subsection{Additional discussion on null infinity}
In this subsection, we give some discussions of the limit behavior of $(F,\phi)$ on null infinity.
Here we always assume $u\geq U_*:=-1$, then $u_*\geq 1/4$.
At first, we see a calculation trick that
\begin{Eq}\label{Eq:spts_3}
\lim_{s\rightarrow\infty}\sum_{n\leq N}|c|^{-n}\kl|{\kl((s+c)^2\partial_s\kr)^n\kl((s+c)^2f\kr)}\kr|
\approx\lim_{s\rightarrow\infty}\sum_{n\leq N}|c|^{-n}\kl|{\kl(s^2\partial_s\kr)^n(s^2f)}\kr|
\end{Eq}
is held for the constant independent of $c$.
Then, using the relations under mapping $\pi$,
we know
\begin{Eq}\label{Eq:Eo_rho_oini}
\lim_{v\rightarrow\infty}\sum_{n\leq 1}u_*^{-2n}\int_{\rS^2}\kl|{(r^2L)^n\Omega^{\leq 5-2n}(r^2\rho)}\kr|^2\d\omega
\approx&\lim_{v\rightarrow\infty}\sum_{n\leq 1}u_*^{-2n}\int_{\rS^2}\kl|{(v_*^2L)^n\Omega^{\leq5-2n}(r^2\rho)}\kr|^2\d\omega\\
\lesssim&\sum_{n\leq 1}\tv^{2n}\int_{\trS^2}\kl|\tLb^n\kl(\tr^2\tOmega^{\leq 5-2n}(\tLb \tA_\tL,\tL\tA_\tLb)\kr)\kr|_{\tu=0}^2\d\tomega\\
\lesssim&\sum_{n\leq 1}\tv^{5} \varF[(\tA_\tLb,\tA_\tL);n+1]\lesssim u_*^{-5}
\end{Eq}
with $\varF[\ \cdot\ ;n]$ is defined and controlled in \Sn{Sn:PoCEo_tAphi_o_tH}. Similarly we know
\begin{Eq}\label{Eq:Eo_sigma_oini}
\lim_{v\rightarrow\infty}\sum_{n\leq 2}u_*^{-2n}\int_{\rS^2}\kl|{(r^2L)^{n}\Omega^{\leq 5-2n}(r^2\sigma)}\kr|^2\d\omega
\lesssim&\sum_{n\leq 2}\tv^{2n}\int_{\trS^2}\kl|\tLb^n\kl(\tr\tOmega^{\leq 6-2n}\tsA\kr)\kr|_{\tu=0}^2\d\tomega\\
\lesssim&\sum_{n\leq 2}\tv^{5} \varF[\tsA;n]\lesssim u_*^{-5}
\end{Eq}
and
\begin{Eq}\label{Eq:Eo_alpha_oini}
\lim_{v\rightarrow\infty}\sum_{n\leq 1}u_*^{-2n}\int_{\rS^2}\kl|{(r^2L)^n\Omega^{\leq 4-2n}(r^3\alpha)}\kr|^2\d\omega
\lesssim&\sum_{n\leq 1}\tv^{2n}\int_{\trS^2}\kl|\tLb^n\tOmega^{\leq 4-2n}(\tOmega\tA_\tLb,\tLb(\tr\tsA))\kr|_{\tu=0}^2\d\omega\\
\lesssim&\sum_{n\leq 1}\tv^{3}\varF[(\tA_\tLb,\tsA);n+1]\lesssim u_*^{-3}.
\end{Eq}

As for $\phi$, similar to the process in the last section, we know
\begin{Eq}\label{Eq:Eo_phi_oini}
\lim_{v\rightarrow\infty}\sum_{n\leq 2}u_*^{-2n}\int_{\rS^2}\kl|{(r^2D_L)^n D_\Omega^{\leq 5-2n}(r\phi)}\kr|^2\d\omega
\lesssim&\sum_{n\leq 2}\tv^{2n}\int_{\trS^2}\kl|{\tD_\tLb^n \tD_\tOmega^{\leq 5-2n}(\tr\tphi)}\kr|^2\d\tomega\\
\lesssim&\sum_{n\leq 2}\tv^{5}\varF[\tphi;n]\kl<\varF[(\tA_\tLb,\tsA),n]\kr>^{5-n}\lesssim u_*^{-5}.
\end{Eq}


\section{The analysis in exterior region, part 1}\label{Sn:Taierp1}
In this and the next sections, we will focus on a subset of $\rD_{-\infty,-1}^{1,\infty}$.
Thus, in these sections we always assume $1\leq -u\leq v\approx r$ (and similar for $U$, $V$).
We use the convention $\INF{f}$ to denote $\lim_{v\rightarrow\infty}f$. 
Meanwhile, we always allow the upcoming constant in $\lesssim$ to depend on the initial data.
Finally, we fix a $\delta$ that is small enough, say, $(1-\varepsilon_2)/100$.

\subsection{Preparation}
As that in \Tm{Tm:Main2}, for $u\leq -1$, we denote 
\begin{Eq*}
\rF:=F-F[\q],\quad F[\q]=\q r^{-2}\d t\wedge\d r . 
\end{Eq*}
Then we can easily calculate that
\begin{Eq}\label{Eq:Ro_F0Z}
|u|^2\kl|F[\q]_{Z\Lb}\kr|+v^2\kl|F[\q]_{ZL}\kr|\lesssim& |u|^{\chi_K(Z)+1}.
\end{Eq}

To discuss the solution in $\rD_{-\infty,-1}^{1,\infty}$, we should first give its limit behavior on the null infinity. Here we give the following claim but postpone its proof to \Sn{Sn:PoCEo_FphiZ_oeni}.
\begin{claim}\label{Cm:Eo_FphiZ_oeni}
Assume that \eqref{Eq:crosd} and $\mathcal{E}_{\varepsilon_1}<\infty$ are satisfied. Then, for any $\chi(\bZ)\leq 2$, there are
\begin{Eq}\label{Eq:Eo_FphiZ_oeni}
&\sum_{i=1,2}\int_{-\infty}^{-1}\int_{\rS^2}|u|^{-2\chi_K(\bZ)}\kl|\INF{r(Col_i^{(\bZ)})}\kr|^2 \d\omega\d u\lesssim 1,\\
&(Col_1^{(\bZ)}):=|u|^{-\frac{\varepsilon_1}{2}}\kl(|u|(\alphab^{(\bZ)},\hD_\Lb\phi^{(\bZ)}), v(\rrho^{(\bZ)},\sigma^{(\bZ)},\sD\phi^{(\bZ)})\kr),\\
&(Col_2^{(\bZ)}):=|u|^{1-\varepsilon_1}v\kl(|u|J_\Lb^{(\bZ)},v\sJ^{(\bZ)}\kr).
\end{Eq}
For $\chi(\bZ)\leq 1$ and $u\leq -1$, there are
\begin{Eq}\label{Eq:Eo_FZ_onis}
&\lim_{v\rightarrow\infty}\kl(|u|^2\|{r\alphab^{(\bZ)}}\|_{L_\omega^2}+|u|\kl\|r^2(\rrho^{(\bZ)},\sigma^{(\bZ)})\kr\|_{L_\omega^2}+\|{r^3\alpha^{(\bZ)}}\|_{L_\omega^2}\kr)\lesssim|u|^{\chi_K(\bZ)+\frac{1+\varepsilon_1}{2}}.
\end{Eq}
Moreover, for $\chi(\bZ)\leq 2$, $2\leq p<\infty$ and $u\leq -1$, there are
\begin{Eq}\label{Eq:Eo_phiZ_onis}
\lim_{v\rightarrow\infty}\|r\phi^{(\bZ)}\|_{L_\omega^p}\lesssim |u|^{\chi_K(\bZ)+\frac{-1+\varepsilon_1}{2}}.
\end{Eq}
\end{claim}

On the other hand, 
by discussions in \Sn{Sn:Ado_H} and the fact $\rrho=\rho-\rho_0$ with $\rho_0= \q r^{-2}$,
we can also find that for any $\chi(\bZ)\leq 2$,   
there are
\begin{Eq}\label{Eq:Eo_FphiZ_o_HU*}
&\sum_{i=3,4}\int_{\rH_{U_*}^{-U_*,\infty}}|u|^{-2\chi_K(\bZ)}|(Col_i^{(\bZ)})|^2\lesssim C_{U_*},\\
&(Col_3^{(\bZ)}):=|u|^{-\frac{\varepsilon_1}{2}}\kl(|u|(\rrho^{(\bZ)},\sigma^{(\bZ)},\sD\phi^{(\bZ)}),v(\alpha^{(\bZ)},\hD_L\phi^{(\bZ)})\kr),\\
&(Col_4^{(\bZ)}):=|u|^{1-\varepsilon_1}v\kl(|u|\sJ^{(\bZ)},vJ^{(\bZ)}_L\kr).
\end{Eq}

Now, we are prepared to use the energy method to analyze the solution in $\rD_{+;-\infty,U_*}^{V_*,\infty}$ with some $U_*$ and $V_*$ fixed later.
In order to have a more detailed discussion, we split this region to an upper region $\rD_{u;-\infty,U_*}^{V_*,\infty}$ and a lower region $\rD_{l;-\infty,U_*}^{V_*,\infty}$, where $\rD_{u;U_1,U_2}^{V_1,V_2}:=\rD_{+;U_1,U_2}^{V_1,V_2}\cap \{v\geq |u|^{1+(\varepsilon_2-\varepsilon_1)/R}\}$ and $\rD_{l;U_1,U_2}^{V_1,V_2}:=\rD_{+;U_1,U_2}^{V_1,V_2}\cap \{v\leq |u|^{1+(\varepsilon_2-\varepsilon_1)/R}\}$ with $\varepsilon_1$, $\varepsilon_2$ are given by \Tm{Tm:Main2} and $R\gg 1$ will be fixed later (see the figure below).
\begin{figure}[H]
\centering
\includegraphics[width=0.7\textwidth]{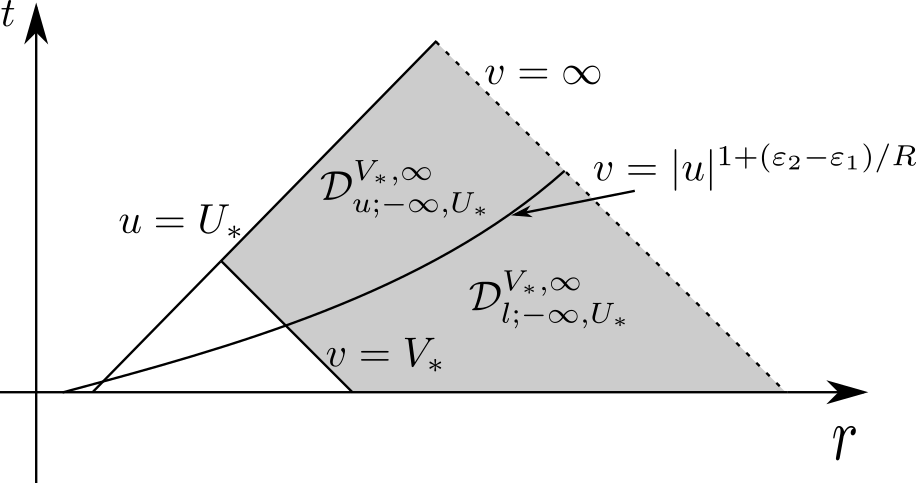}
\end{figure}

\subsection{The bootstrap ansatz}\label{Sn:Tbaiur}
In the rest of this section, we are going to discuss $(F,\phi)$ on $\rD_{u;-\infty,U_*}^{V_*,\infty}$,
and postpone the discussion on $\rD_{l;-\infty,U_*}^{V_*,\infty}$ to the next section.
Thus,
in what follows,
we always assume that $(u,v,\omega)\in \rD_{u;-\infty,U_*}^{V_*,\infty}$,
and fix the $\kappa$ in \Sn{Tmvfaee} to be $|u|^{-\varepsilon_1}$. 
Then, in \eqref{Eq:Eet} we have
\begin{Eq*}
-\frac{u^2L\kappa}{2}=0,\qquad -\frac{v^2\Lb\kappa}{2}=-\frac{\varepsilon_1}{2}|u|^{-1-\varepsilon_1}v^2<0,\qquad -\frac{K\kappa}{4}=-\frac{\varepsilon_1}{2}|u|^{1-\varepsilon_1}<0.
\end{Eq*}

We define the energy corresponding to fields $[f,G]$ to be
\begin{Eq*}
\varE_u[f,G](U,V):=&\int_{\rHb_{U,U_*}^{V}}|u|^{-\varepsilon_1}\kl(|u|^2\big|(\alphab[G],\hD_\Lb f)\big|^2+v^2\big|(\rho[G],\sigma[G],\sD f)\big|^2\kr)\\
&\quad+\int_{\rH_U^{V,\infty}}|u|^{-\varepsilon_1}\kl(|u|^2\big|(\rho[G],\sigma[G],\sD f)\big|^2+v^2\big|(\alpha[G],\hD_{L} f)\big|^2\kr),
\end{Eq*}
and define the energy corresponding to the current $J$ to be
\begin{Eq*}
\varE_u[J](U,V)
:=&\int_{\rHb_{U,U_*}^V}|u|^{2-2\varepsilon_1}v^2 \kl(|u|^2|J_{\Lb}|^2+v^2|\sJ|^2\kr)+\int_{\rH_U^{V,\infty}}|u|^{2-2\varepsilon_1}v^2\kl(|u|^2|\sJ|^2+v^2|J_L|^2\kr).
\end{Eq*}

Then, for any given $\varepsilon_1$, $\varepsilon_2$, $R$ and $\chi(\bZ)\leq 2$, 
we set the  bootstrap ansatz to be
\begin{btsz}\label{Bz:u}
\begin{Eq}\label{Eq:Eoeiur-FphiJ}
\varE_u[\phi^{(\bZ)},\rF^{(\bZ)}](U,V)\leq& 8M_1|U|^{2\chi_K(\bZ)},\\
\varE_u[J^{(\bZ)}](U,V)\leq& 2M_1|U|^{2\chi_K(\bZ)+1+4\delta}.
\end{Eq}
\end{btsz}
In the beginning, we fix a constant $M_1$ to be much larger than the constants in \eqref{Eq:Eo_FphiZ_oeni}.
Meanwhile,  for any fixed $U_*$, we are able to find a $V_*$ such that the integral in \eqref{Eq:Eo_FphiZ_o_HU*} with $\rH_{U_*}^{-U_*,\infty}$ replaced by $\rH_{U_*}^{V_*,\infty}$ is smaller than $M_1/2$. 
These ensure that \eqref{Eq:Eoeiur-FphiJ} is held for $U$ sufficiently close to $U_*$ and $V$ sufficiently large.

Then, we will show that 
there exist a $U_1\leq -1$ depends on the initial data,
such that the constant $8$ and $2$ in \eqref{Eq:Eoeiur-FphiJ} could be improved to $4$ and $1$ 
while $U\leq U_*\leq U_1$, $V\geq V_*$.
This means that \eqref{Eq:Eoeiur-FphiJ} will be held for all $(U,V)$ in such region.

To show this result, we first show several results while \eqref{Eq:Eoeiur-FphiJ} are held.

\subsection{Estimate of the Maxwell field on sphere}\label{Sn:EotMfosiur}
In this subsection, we will give the $L_\omega^p$ estimate of the Maxwell field on $\rS_u^v$.
\begin{lemma}\label{La:Eo_FZ_osiur}
Assume that \eqref{Eq:Eo_FZ_onis} and \eqref{Eq:Eoeiur-FphiJ} are satisfied.
For $\chi(\bZ)=1$, $p=2$, or $\chi(\bZ)=0$, $2\leq p<\infty$, we have
\begin{Eq}\label{Eq:Eo_FZ_osiur}
&|u|^2\|\alphab^{(\bZ)}\|_{L_\omega^p}+|u|v\kl\|\kl(\rrho^{(\bZ)},\sigma^{(\bZ)}\kr)\kr\|_{L_\omega^p}+v^2\|\alpha^{(\bZ)}\|_{L_\omega^p}
\lesssim|u|^{\chi_K(\bZ)+\frac{1+\varepsilon_1}{2}}v^{-1}.
\end{Eq}
\end{lemma}
\begin{corollary}\label{Cy:Eo_FrZ_osiur}
Assume that \eqref{Eq:Eo_FZ_onis} and \eqref{Eq:Eoeiur-FphiJ} are satisfied.
For $\chi(\bZ)=1$, $p=2$, or $\chi(\bZ)=0$, $2\leq p<\infty$, we have
\begin{Eq}\label{Eq:Eo_FrZ_osiur}
|u|^2\|\rF^{(\bZ)}_{Z\Lb}\|_{L_\omega^p}+|u|v\sum_i\|\rF^{(\bZ)}_{Ze_i}\|_{L_\omega^p}+v^2\|\rF^{(\bZ)}_{ZL}\|_{L_\omega^p}\lesssim  |u|^{\chi_K(\bZ,Z)+\frac{1+\varepsilon_1}{2}}.
\end{Eq}
\end{corollary}

\begin{proof}[Proof of \La{La:Eo_FZ_osiur}]
\part[The estimate of $\alphab^{(\bZ)}$]
Using \eqref{Eq:MKGe-ud_FZ}, \eqref{Eq:EbudaLd} and \eqref{Eq:Eoeiur-FphiJ}, for $\chi(\bZ)\leq 1$, we know 
\begin{Eq*}
\int_{\rH_{U}^{V,\infty}} |L(r\alphab^{(\bZ)})|^2\lesssim&\int_{\rH_U^{V,\infty}}\sum_{|\bOmega|\leq 1}\kl|\kl(\rrho^{(\bOmega,\bZ)},\sigma^{(\bOmega,\bZ)}\kr)\kr|^2+\kl|r\sJ^{(\bZ)}\kr|^2\\
\lesssim& |U|^{2\chi_K(\bZ)-2+\varepsilon_1}+ |U|^{2\chi_K(\bZ)-3+2\varepsilon_1+4\delta}\\
\lesssim&|U|^{2\chi_K(\bZ)-3+\varepsilon_1}V.
\end{Eq*}
Then, using \eqref{Eq:key_Eob1_2} and \eqref{Eq:Eo_FZ_onis}, for $\chi(\bZ)\leq 1$, we find 
\begin{Eq*}
\int_{\rS^2}|r\alphab^{(\bZ)}|_{v=V}^2\d\omega
\lesssim& \lim_{v\rightarrow\infty}\int_{\rS^2}|r\alphab^{(\bZ)}|^2\d\omega+V^{-1}\int_V^\infty\int_{\rS^2} v^2|L(r\alphab^{(\bZ)})|^2\d \omega\d v\\
\lesssim& |u|^{2\chi_K(\bZ)-3+\varepsilon_1}.
\end{Eq*}
This, \eqref{Eq:EbudaLd} and the \emph{Sobolev} inequality gives the $\alphab$ part of \eqref{Eq:Eo_FZ_osiur}.

\part[The estimate of $\rrho^{(\bZ)}$ and $\sigma^{(\bZ)}$]
Using \eqref{Eq:MKGe-ud_FZ}, \eqref{Eq:EbudaLd} 
and \eqref{Eq:Eoeiur-FphiJ}, for $\chi(\bZ)\leq 1$, we similarly have  
\begin{Eq*}
\int_{\rH_{U}^{V,\infty}} \kl|L\kl(r^2(\rrho^{(\bZ)},\sigma^{(\bZ)})\kr)\kr|^2
\lesssim&\int_{\rH_U^{V,\infty}} \sum_{|\bOmega|\leq 1}\kl|r\alpha^{(\bOmega,\bZ)}\kr|^2+|r^2J^{(\bZ)}_L|^2
\lesssim |U|^{2\chi_K(\bZ)-1+\varepsilon_1}V.
\end{Eq*}
Then, using \eqref{Eq:key_Eob1_2} and \eqref{Eq:Eo_FZ_onis},  for $\chi(\bZ)\leq 1$,  we similarly find 
\begin{Eq*}
\int_{\rS^2}\kl|r^2(\rrho^{(\bZ)},\sigma^{(\bZ)})\kr|_{v=V}^2\d\omega\lesssim& |u|^{2\chi_K(\bZ)-1+\varepsilon_1}.
\end{Eq*}
This gives $\rrho$ and $\sigma$ parts of \eqref{Eq:Eo_FZ_osiur}.

\part[The estimate of $\alphab^{(\bZ)}$]
At first, for any $2$-form $G$, we have
\begin{Eq*}
\frac{1}{2}v\alpha[\LD_KG]_i=r^{-1}uv^2\alpha[G]_i+L\kl(v^3\alpha[G]_i\kr)+r^{-1}u^2v\Lb(r\alpha[G]_i).
\end{Eq*}
Then, by \eqref{Eq:MKGe-ud_FZ} and \eqref{Eq:EbudaLd}, we know 
\begin{Eq*}
|L(v^3\alpha[G])|\lesssim v|\alpha^{(K,\bZ)}|+|u|v \big|\alpha^{(\bZ)}\big|+|u|^2\sum_{|\bOmega|\leq 1}\kl|(\rrho^{(\bOmega,\bZ)},\sigma^{(\bOmega,\bZ)})\kr|+|u|^2|r\sJ^{(\bZ)}|.
\end{Eq*}
Using \eqref{Eq:Eoeiur-FphiJ}, for $\chi(\bZ)\leq 1$, we similarly have
\begin{Eq*}
\int_{\rH_{U}^{V,\infty}} |L(v^3\alpha^{(\bZ)}_i)|^2
\lesssim&\int_{\rH_U^{V,\infty}}v^2|\alpha^{(K,\bZ)}|^2+|U|^2\int_{\rH_U^{V,\infty}}v^2|\alpha^{(\bZ)}|^2\\
&\quad+|U|^4\int_{\rH_U^{V,\infty}}\sum_{|\bOmega|\leq 1}\kl|\kl(\rrho^{(\bOmega,\bZ)},\sigma^{(\bOmega,\bZ)}\kr)\kr|^2+\kl|r\sJ^{(\bZ)}\kr|^2\\
\lesssim&|U|^{2\chi_K(\bZ)+1+\varepsilon_1}V.
\end{Eq*}
Now, using \eqref{Eq:key_Eob1_2} and \eqref{Eq:Eo_FZ_onis}, for $\chi(\bZ)\leq 1$, we also find 
\begin{Eq*}
\int_{\rS^2}\kl|v^3\alpha^{(\bZ)}\kr|_{v=V}^2\d\omega\lesssim& |u|^{2\chi_K(\bZ)+1+\varepsilon_1}.
\end{Eq*}
This gives the $\alpha$ part of \eqref{Eq:Eo_FZ_osiur}, and finishes the proof.
\end{proof}

\begin{proof}[Proof of \Cy{Cy:Eo_FrZ_osiur}]
At first, for any $2$-form $G$ and $Z\in\mathcal{Z}$, we easily know
\begin{Eq*}
|G_{Z\Lb}|\lesssim& \begin{cases}
r\big|\alphab[G]\big|,&Z=\Omega,\\
v^{\chi_K(Z)+1}\big|\rho[G]\big|, &else.
\end{cases}\\
|G_{Ze_i}|\lesssim& \begin{cases}
r|\sigma[G]|,&Z=\Omega,\\
|u|^{\chi_K(Z)+1}\big|\alphab[G]\big|+v^{\chi_K(Z)+1}\big|\alpha[G]\big|, &else,
\end{cases}\\
|G_{ZL}|\lesssim& \begin{cases}
r|\alpha[G]|,&Z=\Omega,\\
|u|^{\chi_K(Z)+1}\big|\rho[G]\big|, &else.
\end{cases}\\
\end{Eq*}
Then, for any $\bZ$ and $p$ in \Cy{Cy:Eo_FrZ_osiur}, by \eqref{Eq:Eo_FZ_osiur}, we know
\begin{Eq*}
\|\rF^{(\bZ)}_{Z\Lb}\|_{L_\omega^p}\lesssim&v\|\alphab^{(\bZ)}\|_{L_\omega^p}+v^{\chi_K(Z)+1}\|\rrho^{(\bZ)}\|_{L_\omega^p}
\lesssim  |u|^{\chi_K(\bZ,Z)-\frac{3}{2}+\frac{\varepsilon_1}{2}},\\
\|\rF^{(\bZ)}_{Ze_i}\|_{L_\omega^p}\lesssim&|u|^{\chi_K(Z)+1}\|\alphab^{(\bZ)}\|_{L_\omega^p}+v\|\sigma^{(\bZ)}\|_{L_\omega^p}+v^{\chi_K(Z)+1}\|\alpha^{(\bZ)}\|_{L_\omega^p}
\lesssim|u|^{\chi_K(\bZ,Z)+\frac{-1+\varepsilon_1}{2}}v^{-1},\\
\|\rF^{(\bZ)}_{ZL}\|_{L_\omega^p}\lesssim&v\|\alpha^{(\bZ)}\|_{L_\omega^p}+|u|^{\chi_K(Z)+1}\|\rrho^{(\bZ)}\|_{L_\omega^p}
\lesssim |u|^{\chi_K(\bZ,Z)+\frac{1+\varepsilon_1}{2}}v^{-2}.
\end{Eq*}
This finishes the proof.
\end{proof}

\subsection{Estimate of the scalar field on sphere}\label{Sn:Eotsfosiur}
Setting $p_\delta:=\frac{4}{2-\delta}>2$, we will show that
\begin{lemma}\label{La:Eo_phiZ_osiur}
Assume that \eqref{Eq:Eo_phiZ_onis} and \eqref{Eq:Eoeiur-FphiJ} are satisfied.
Then, we have 
\begin{Eq}\label{Eq:Eo_phiZ_osiur}
\|\phi^{(\bZ)}\|_{L_\omega^p}
\lesssim&\begin{cases}
 |u|^{\chi_K(\bZ)+\frac{-1+\varepsilon_1}{2}}v^{-1},&
\begin{aligned}
&\chi(\bZ)=2,~p=2,~or~\chi(\bZ)=1,~2\leq p<\infty,\\
&\quad or~\chi(\bZ)=0,~2\leq p\leq\infty,
\end{aligned}\\
|u|^{\chi_K(\bZ)+\frac{-1+\varepsilon_1}{2}+\delta}v^{-1},& \chi(\bZ)=2,~2<p\leq p_\delta,~or~\chi(\bZ)=1,~p=\infty.
\end{cases}
\end{Eq}
We also have
\begin{Eq}\label{Eq:Eo_DphiZ_osiur}
&|u|^2\|\hD_\Lb\phi^{(\bZ)}\|_{L_\omega^p}+|u|v\|\sD\phi^{(\bZ)}\|_{L_\omega^p}+v^2\|\hD_L\phi^{(\bZ)}\|_{L_\omega^p}\\
\lesssim& \begin{cases}
|u|^{\chi_K(\bZ)+\frac{1+\varepsilon_1}{2}}v^{-1},&\chi(\bZ)=1,~p=2,~or~\chi(\bZ)=0,~2\leq p<\infty,\\
|u|^{\chi_K(\bZ)+\frac{1+\varepsilon_1}{2}+\delta}v^{-1},& \chi(\bZ)=1,~2<p\leq p_\delta,~or~\chi(\bZ)=0,~p=\infty.
\end{cases}
\end{Eq}
\end{lemma}
\begin{corollary}\label{Cy:So_J_osiur}
Assume that \eqref{Eq:Eo_phiZ_onis} and \eqref{Eq:Eoeiur-FphiJ} are satisfied.
For $\chi(\bZ)=1$, $p=2$ or $\chi(\bZ)=0$, $2\leq p<\infty$, we have
\begin{Eq}\label{Eq:Eo_J_osiur}
|u|^2\|J^{(\bZ)}_\Lb\|_{L_\omega^p}+|u|v\|\sJ^{(\bZ)}\|_{L_\omega^p}+v^2\|J^{(\bZ)}_L\|_{L_\omega^p}\lesssim |u|^{\chi_K(\bZ)+\varepsilon_1}v^{-2}.
\end{Eq}
\end{corollary}
\begin{proof}[Proof of \La{La:Eo_phiZ_osiur}]
At first, using \eqref{Eq:key_Eob1_2}, \eqref{Eq:Eo_phiZ_onis} and \eqref{Eq:Eoeiur-FphiJ}, 
for $\chi(\bZ)\leq 2$, we have
\begin{Eq*}
&V^{-1}\int_{\rS^2}|r\phi^{(\bZ)}|^2_{v=V}\d\omega+ \int_V^{\infty}\int_{\rS^2}v^{-2}|r\phi^{(\bZ)}|^2\d\omega\d v\\
\lesssim & V^{-1}\lim_{v\rightarrow\infty}\int_{\rS^2}|r\phi^{(\bZ)}|^2\d\omega+V^{-2}\int_V^\infty\int_{\rS^2}v^2|D_L(r\phi^{(\bZ)})|^2\d\omega\d v\\
\lesssim&  |u|^{2\chi_K(\bZ)-1+\varepsilon_1}V^{-1}.
\end{Eq*}
Here we replace the $\partial$ by $D$ when using \eqref{Eq:key_Eob1_2} since obviously $\big|\partial_\mu|f|\big|\leq |D_\mu f|$.
This and the \emph{Sobolev} inequality give the first part of \eqref{Eq:Eo_phiZ_osiur}. 

To get another part of \eqref{Eq:Eo_phiZ_osiur}, we first estimate $\|\phi^{(\bZ)}\|_{L_\omega^4}$.
Using \eqref{Eq:key_42}, \eqref{Eq:Eo_phiZ_onis}, the above inequality and \eqref{Eq:Eoeiur-FphiJ},
for $\chi(\bZ)=2$, we find
\begin{Eq}\label{Eq:Eo_L4_phiZ2_iur}
\kl(\int_{\rS^2}|\phi^{(\bZ)}|^4\d \omega\kr)^{\frac{1}{2}}_{v=V}
\lesssim&\int_V^{\infty}\int_{\rS^2}|\phi^{(\bZ)}|^2+|D_\Omega \phi^{(\bZ)}|^2+|D_L \phi^{(\bZ)}|^2\d\omega\d v\\
\lesssim&   |u|^{2\chi_K(\bZ)-1+\varepsilon_1}V^{-1}+\int_{\rH_{u}^{V,\infty}}|\sD \phi^{(\bZ)}|^2+v^{-2}|\hD_L \phi^{(\bZ)}|^2\\
\lesssim& |u|^{2\chi_K(\bZ)-2+\varepsilon_1}.
\end{Eq}
Next, to make up the decay in $v$ direction, for $2<p\leq p_\delta$, we further calculate that 
\begin{Eq*}
&\kl(\int_{\rS^2}|r\phi^{(\bZ)}|^{p}\d \omega\kr)_{v=V}-\lim_{v\rightarrow\infty}\int_{\rS^2}|r\phi^{(\bZ)}|^{p}\d \omega\\
\lesssim& \int_V^\infty\int_{\rS^2}\kl|D_L(r\phi^{(\bZ)})\kr|\big|r\phi^{(\bZ)}\big|^{p-1}\d\omega\d v\\
\lesssim&\kl(\int_{\rH_u^{V,\infty}}v^2|\hD_L\phi^{(\bZ)}|^2 \kr)^{\frac{1}{2}}\kl(\int_V^\infty\int_{\rS^2}r^{-2}|r\phi^{(\bZ)}|^{2p-2}\d\omega\d v\kr)^{\frac{1}{2}}\\
\equiv&I_1\cdot I_2.
\end{Eq*}
Using \eqref{Eq:Eoeiur-FphiJ}, we know $I_1^2\lesssim |u|^{2\chi_K(\bZ)+\varepsilon_1}$.
As for $I_2$, using the results above, we know
\begin{Eq*}
I_2^2\lesssim& \int_V^\infty r^{2p-4}\|\phi^{(\bZ)}\|_{L_\omega^{2p-2}}^{2p-2}\d v
\lesssim\int_V^\infty v^{2p-4}\|\phi^{(\bZ)}\|_{L_\omega^2}^{6-2p}\|\phi^{(\bZ)}\|_{L_\omega^4}^{4p-8}\d v\\
\lesssim &|u|^{(2p-2)\chi_K(\bZ)+(p-1)\varepsilon_1-3p+5}\int_V^\infty v^{4p-10}\d v \\
\lesssim& |u|^{(2p-2)\chi_K(\bZ)+(p-1)\varepsilon_1+p-4},
\end{Eq*}
since $4p-10\leq 4p_\delta-10=(10\delta-4)/(2-\delta)<-1$.
Combing these with \eqref{Eq:Eo_phiZ_onis}, we then know
\begin{Eq*}
\|\phi^{(\bZ)}\|_{L_\omega^{p}}
\lesssim& \kl(|u|^{\chi_K(\bZ)+\frac{-1+\varepsilon_1}{2}}+|u|^{\chi_K(\bZ)+\frac{\varepsilon_1}{2}+\frac{p-4}{2p}}\kr)v^{-1}
\lesssim  |u|^{\chi_K(\bZ)+\frac{-1+\varepsilon_1}{2}+\frac{2p-4}{2p}}v^{-1}.
\end{Eq*}
Since that $(2p-4)/(2p)\leq(2p_\delta-4)/(2p_\delta)=\delta/2$,
this gives the second part of \eqref{Eq:Eo_phiZ_osiur}.

Finally, noticing that $-ruD_\Lb=D_K-vD_S$ and $rvD_L=D_K-uD_S$, for any $\bZ$ and $p$, we have
\begin{Eq*}
\|\hD_\Lb\phi^{(\bZ)}\|_{L_\omega^p}\leq& |u|^{-1}v^{-1}\kl(\|\phi^{(K,\bZ)}\|_{L_\omega^p}+v\|\phi^{(S,\bZ)}\|_{L_\omega^p}\kr),\\
\|\sD\phi^{(\bZ)}\|_{L_\omega^p}\leq& v^{-1}\|\phi^{(\Omega,\bZ)}\|_{L_\omega^p},\\
\|\hD_L\phi^{(\bZ)}\|_{L_\omega^p}\leq& v^{-2}\kl(\|\phi^{(K,\bZ)}\|_{L_\omega^p}+|u|\|\phi^{(S,\bZ)}\|_{L_\omega^p}\kr).
\end{Eq*}
These give the \eqref{Eq:Eo_DphiZ_osiur}, and finish the proof.
\end{proof}

\begin{proof}[Proof of \Cy{Cy:So_J_osiur}]
For $\bZ=Z$, using \eqref{Eq:Efo_JZ}, we find
\begin{Eq*}
\|J^{(Z)}_\mu\|_{L_\omega^2}\lesssim&\|\phi^{(Z)}\|_{L_\omega^4}\|\hD_\mu\phi\|_{L_\omega^4}+\|\phi\|_{L_\omega^\infty}\|\hD_\mu \phi^{(Z)}\|_{L_\omega^2}
+\|\phi\|_{L_\omega^\infty}^2 \|\rF_{Z\mu}\|_{L_\omega^2}+\|\phi\|_{L_\omega^4}^2 |F[\q]_{Z\mu}|.
\end{Eq*}
Then, using \eqref{Eq:Ro_F0Z}, \eqref{Eq:Eo_phiZ_osiur} and \eqref{Eq:Eo_DphiZ_osiur}, we have
\begin{Eq*}
\|J^{(Z)}_\Lb\|_{L_\omega^2}\lesssim& (1+ |u|^{-\frac{1}{2}})|u|^{\chi_K(Z)-2+\varepsilon_1}v^{-1}
\lesssim  |u|^{\chi_K(Z)-2+\varepsilon_1}v^{-1},\\
\|J^{(Z)}_{e_i}\|_{L_\omega^2}\lesssim& (1+ |u|^{-\frac{1}{2}})|u|^{\chi_K(Z)-1+\varepsilon_1}v^{-2}
\lesssim |u|^{\chi_K(Z)-1+\varepsilon_1}v^{-2},\\
\|J^{(Z)}_L\|_{L_\omega^2}\lesssim& (1+ |u|^{-\frac{1}{2}})|u|^{\chi_K(Z)+\varepsilon_1}v^{-3}
\lesssim  |u|^{\chi_K(Z)+\varepsilon_1}v^{-3}.
\end{Eq*}
This gives the $\chi(\bZ)=1$ part of \eqref{Eq:Eo_J_osiur}. 
The proof of the case that $\chi(\bZ)=0$ can be obtained similarly. We leave it to the interested reader.
\end{proof}

\subsection{Energy estimates of Maxwell field}
Here, we use the \eqref{Eq:Ee} with $\rD=\rD_{U,U_*}^{V,\infty}$, $[f,G]=[0,\rF^{(\bZ)}]$ and $\kappa=|u|^{-\varepsilon_1}$.
Through the discussions in \Sn{Tmvfaee} and \Sn{Sn:Tbaiur},
using \eqref{Eq:Eoeiur-FphiJ} and noticing $V\geq|U|\geq|U_*|$, we have
\begin{Eq*}
&\varE_u[0,\rF^{(\bZ)}](U,V)-M_1|U|^{2\chi_K(\bZ)}
\lesssim\int_\rDa |u|^{-\varepsilon_1}|(\rF^{(\bZ)})_K{}^{\mu}\nabla^{\nu}(\rF^{(\bZ)})_{\nu\mu}|\\
\lesssim&\int_\rDa |u|^{2-\varepsilon_1} \kl(|\rrho^{(\bZ)}||J^{(\bZ)}_\Lb|+|\alphab^{(\bZ)}||\sJ^{(\bZ)}|\kr)+|u|^{-\varepsilon_1}v^2\kl(|\rrho^{(\bZ)}||J^{(\bZ)}_L|+|\alpha^{(\bZ)}||\sJ^{(\bZ)}|\kr)\\
\lesssim&\int_V^{\infty}|U|^{\frac{\varepsilon_1}{2}}v^{-2}\kl(\int_{\rHb_{U,U_*}^{v}}|u|^{-\varepsilon_1}v^2|\rrho^{(\bZ)}|^2\kr)^{\frac{1}{2}}\kl(\int_{\rHb_{U,U_*}^{v}}|u|^{4-2\varepsilon_1}v^2|J^{(\bZ)}_\Lb|^2\kr)^{\frac{1}{2}}\d v\\
&\quad+\int_V^{\infty}|U|^{\frac{\varepsilon_1}{2}}v^{-2}\kl(\int_{\rHb_{U,U_*}^{v}}|u|^{2-\varepsilon_1}|\alphab^{(\bZ)}|^2\kr)^{\frac{1}{2}}\kl(\int_{\rHb_{U,U_*}^{v}}|u|^{2-2\varepsilon_1}v^4|\sJ^{(\bZ)}|^2\kr)^{\frac{1}{2}}\d v\\
&\quad+\int_{U}^{U_*}|u|^{-2+\frac{\varepsilon_1}{2}}\kl(\int_{\rH_{u}^{V,\infty}}|u|^{2-\varepsilon_1}|\rrho^{(\bZ)}|^2\kr)^{\frac{1}{2}}\kl(\int_{\rH_{u}^{V,\infty}}|u|^{2-2\varepsilon_1}v^4|J^{(\bZ)}_L|^2\kr)^{\frac{1}{2}}\d u\\
&\quad+\int_{U}^{U_*}|u|^{-2+\frac{\varepsilon_1}{2}}\kl(\int_{\rH_{u}^{V,\infty}}|u|^{-\varepsilon_1}v^2|\alpha^{(\bZ)}|^2\kr)^{\frac{1}{2}}\kl(\int_{\rH_{u}^{V,\infty}}|u|^{4-2\varepsilon_1}v^2|\sJ^{(\bZ)}|^2\kr)^{\frac{1}{2}}\d u\\
\lesssim&\int_V^{\infty}|U|^{2\chi_K(\bZ)+\frac{1+\varepsilon_1}{2}+2\delta}v^{-2}\d v+\int_U^{U_*}|u|^{2\chi_K(\bZ)+\frac{-3+\varepsilon_1}{2}+2\delta}\d u\\
\lesssim& |U_*|^{\frac{-1+\varepsilon_1}{2}+2\delta}|U|^{2\chi_K(\bZ)}.
\end{Eq*}
This means that for some constant $C_1$ independent with $U_*$, we have
\begin{Eq*}
\varE_u[0,\rF^{(\bZ)}](U,V)\leq \kl(1+C_1|U_*|^{\frac{-1+\varepsilon_1}{2}+2\delta}\kr)M_1|U|^{2\chi_K(\bZ)}.
\end{Eq*}

\subsection{Energy estimates of scalar field}
Here we only consider the case that $\bZ=(Z_1,Z_2)$.
The cases $\chi(\bZ)\leq 1$ can be obtained similarly and more simply, which are left to the interested reader. 
We use the \eqref{Eq:Ee} with $\rD=\rD_{U,U_*}^{V,\infty}$, $[f,G]=[\phi^{(\bZ)},0]$ and $\kappa=|u|^{-\varepsilon_1}$.
By the discussions in \Sn{Tmvfaee} and \Sn{Sn:Tbaiur}, we know
\begin{Eq*}
\varE_u[\phi^{(\bZ)},0](U,V)-M_1|U|^{2\chi_K(\bZ)}
\leq&\int_\rDa |u|^{-\varepsilon_1}\kl(|\Re(\overline{\square_A\phi^{(\bZ)}}\hD_K\phi^{(\bZ)})|+|\Im(\phi^{(\bZ)}\overline{\hD^\mu \phi^{(\bZ)}})F_{K\mu}|\kr)\\
\lesssim&\int_\rDa|u|^{2-\varepsilon_1}\kl(\kl|\kl(\square_A\phi^{(\bZ)},\phi^{(\bZ)}\rho\kr)\kr||\hD_\Lb\phi^{(\bZ)}|+|\phi^{(\bZ)}\alphab||\sD \phi^{(\bZ)}|\kr)\\
&\phantom{\int_\rDa}\quad +|u|^{-\varepsilon_1} v^2\kl(\kl|\kl(\square_A\phi^{(\bZ)},\phi^{(\bZ)}\rho\kr)\kr||\hD_L \phi^{(\bZ)}|+|\phi^{(\bZ)}\alpha||\sD \phi^{(\bZ)}|\kr).
\end{Eq*}
Among this, by \eqref{Eq:MKGe_phiZ}, we have
\begin{Eq}\label{Eq:Efo_boxphiZ2}
|\square_A\phi^{(\bZ)}|\lesssim&|Q(\phi^{(Z_1)},F;Z_2)|+|Q(\phi^{(Z_2)},F;Z_1)|+|Q(\phi,F;[Z_1,Z_2])|\\
&+|Q(\phi,F^{(Z_1)};Z_2)|+|F_{Z_1\mu}F_{Z_2}{}^\mu\phi|
\equiv\sum_{j=1}^5 I_j.
\end{Eq}
Here, using \eqref{Eq:Eo_Q} and results in \Sn{Sn:EotMfosiur} and \Sn{Sn:Eotsfosiur}, setting $q_\delta^{-1}=2^{-1}-p_\delta^{-1}$, we know
\begin{Eq*}
|u|^{-\chi_K(Z_2)}\|Q(\phi^{(Z_1)},F;Z_2)\|_{L_\omega^2}
\lesssim&\kl\|\kl(|u|J_\Lb,v\slashed{J},|u|^{-1}v^2J_L\kr)\kr\|_{L_\omega^4}\|\phi^{(Z_1)}\|_{L_\omega^4}
+\kl\|\kl(|u|\rrho,v\alpha\kr)\kr\|_{L_\omega^{q_\delta}}\|\hD_\Lb \phi^{(Z_1)}\|_{L_\omega^{p_\delta}}\\
&+|u||\rho_0|\|\hD_\Lb \phi^{(Z_1)}\|_{L_\omega^2}
+\kl\|\kl(|u|\alphab,v\sigma,|u|^{-1}v^2\alpha\kr)\kr\|_{L_\omega^{q_\delta}}\|\sD \phi^{(Z_1)}\|_{L_\omega^{p_\delta}}\\
&+\kl\|\kl(v\alphab,|u|^{-1}v^2\rrho\kr)\kr\|_{L_\omega^{q_\delta}}\|\hD_L\phi^{(Z_1)}\|_{L_\omega^{p_\delta}}
+|u|^{-1}v^2|\rho_0|\|\hD_L\phi^{(Z_1)}\|_{L_\omega^2}\\
&+\kl\|\kl(\rrho,\sigma\kr)\kr\|_{L_\omega^4}\|\phi^{(Z_1)}\|_{L_\omega^4}
+|\rho_0|\|\phi^{(Z_1)}\|_{L_\omega^2}\\
\lesssim&|u|^{\chi_K(Z_1)+\frac{-3+3\varepsilon_1}{2}}v^{-3}
+|u|^{\chi_K(Z_1)-1+\varepsilon+\delta}v^{-3}
+ |u|^{\chi_K(Z_1)+\frac{-1+\varepsilon_1}{2}}v^{-3},
\end{Eq*}
which implies that $\|I_1\|_{L_\omega^2}\lesssim  |u|^{\chi_K(\bZ)+\frac{-1+\varepsilon_1}{2}}v^{-3}$. Same estimates also hold for $I_2$ and for $I_3$ since that $\chi_K([Z_1,Z_2])\leq \chi_K(Z_1,Z_2)$.  
For $I_4$, we similarly know
\begin{Eq*}
&|u|^{-\chi_K(Z_2)}\|Q(\phi,F^{(Z_1)};Z_2)\|_{L_\omega^2}\\
\lesssim&\kl\|\kl(|u|J^{(Z_1)}_\Lb,v\slashed{J}^{(Z_1)},|u|^{-1}v^2J^{(Z_1)}_L\kr)\kr\|_{L_\omega^2}\|\phi\|_{L_\omega^\infty}
+\kl\|\kl(|u|\rho^{(Z_1)},v\alpha^{(Z_1)}\kr)\kr\|_{L_\omega^2}\|\hD_\Lb \phi\|_{L_\omega^\infty}\\
&+\kl\|\kl(|u|\alphab^{(Z_1)},v\sigma^{(Z_1)},|u|^{-1}v^2\alpha^{(Z_1)}\kr)\kr\|_{L_\omega^2}\|\sD \phi\|_{L_\omega^\infty}\\
&+\kl\|\kl(v\alphab^{(Z_1)},|u|^{-1}v^2\rho^{(Z_1)}\kr)\kr\|_{L_\omega^2}\|\hD_L\phi\|_{L_\omega^\infty}
+\kl\|\kl(\rho^{(Z_1)},\sigma^{(Z_1)}\kr)\kr\|_{L_\omega^2}\|\phi\|_{L_\omega^\infty}\\
\lesssim&|u|^{\chi_K(Z_1)+\frac{-3+3\varepsilon_1}{2}}v^{-3}+|u|^{\chi_K(Z_1)-1+\varepsilon_1+\delta}v^{-3},
\end{Eq*}
which implies that $\|I_4\|_{L_\omega^2}\lesssim  |u|^{\chi_K(\bZ)+\frac{-1+\varepsilon_1}{2}}v^{-3}$.
Finally, using \eqref{Eq:Eo_FrZ_osiur}, \eqref{Eq:Ro_F0Z} and \eqref{Eq:Eo_phiZ_osiur}, we can also find $\|I_5\|_{L_\omega^2}\lesssim |u|^{\chi_K(\bZ)+\frac{-1+\varepsilon_1}{2}}v^{-3}$.
Mixing these estimates, and using \eqref{Eq:Eo_FZ_osiur} and \eqref{Eq:Eo_phiZ_osiur}, we know
\begin{Eq*}
P_0:=\kl\|\kl(\square_A\phi^{(\bZ)},\phi^{(\bZ)}\rho\kr)\kr\|_{L_\omega^2}\lesssim \|\square_A\phi^{(\bZ)}\|_{L_\omega^2}+\|\phi^{(\bZ)}\|_{L_\omega^2}|\rho_0|+\|\phi^{(\bZ)}\|_{L_\omega^{p_\delta}}\|\rrho\|_{L_\omega^{q_\delta}}\lesssim  |u|^{\chi_K(\bZ)+\frac{-1+\varepsilon_1}{2}}v^{-3}.
\end{Eq*}

Now, by this estimate, \eqref{Eq:Eo_FZ_osiur}, \eqref{Eq:Eo_phiZ_osiur} and \eqref{Eq:Eoeiur-FphiJ}, we get that
\begin{Eq*}
&\varE_u[\phi^{(\bZ)},0](U,V)-M_1|U|^{2\chi_K(\bZ)}\\
\lesssim&  \int_V^{\infty}\kl(\int_{U}^{U_*} |u|^{2-\varepsilon_1}v^2P_0^2\d u\kr)^\frac{1}{2}\kl(\int_{\rHb_{U,U_*}^{v}}|u|^{2-\varepsilon_1}|\hD_\Lb\phi^{(\bZ)}|^2\kr)^{\frac{1}{2}}\d v\\
&\quad+\int_U^{U_*} \kl(\int_V^{\infty}|u|^{2-\varepsilon_1}v^2\|\phi^{(\bZ)}\|_{L_\omega^{p_\delta}}^2\|\alphab\|_{L_\omega^{q_\delta}}^2\d v\kr)^{\frac{1}{2}}\kl(\int_{\rH_{u}^{V,\infty}}|u|^{2-\varepsilon_1}|\slashed D\phi^{(\bZ)}|^2\kr)^{\frac{1}{2}}\d u\\
&\quad+\int_U^{U_*}\kl(\int_V^\infty |u|^{-\varepsilon_1}v^4P_0^2\d v\kr)^\frac{1}{2}\kl(\int_{\rH_u^{V,\infty}}|u|^{-\varepsilon_1}v^2|\hD_L\phi^{(\bZ)}|^2\kr)^{\frac{1}{2}}\d u\\
&\quad+\int_V^\infty \kl(\int_U^{U_*}|u|^{-\varepsilon_1}v^4\|\phi^{(\bZ)}\|_{L_\omega^{p_\delta}}^2\|\alpha\|_{L_\omega^{q_\delta}}^2\d u\kr)^{\frac{1}{2}}\kl(\int_{\rHb_{U,U_*}^v}|u|^{-\varepsilon_1}v^2|\slashed D\phi^{(\bZ)}|^2\kr)^{\frac{1}{2}}\d v\\
\lesssim& \int_V^{\infty}\kl(\int_{U}^{U_*} |u|^{2\chi_K(\bZ)+1}v^{-4}\d u\kr)^\frac{1}{2}\kl(|u|^{2\chi_K(\bZ)}\kr)^{\frac{1}{2}}\d v
 +\int_U^{U_*}\kl(\int_V^\infty |u|^{2\chi_K(\bZ)-1}v^{-2}\d v\kr)^\frac{1}{2}\kl(|u|^{2\chi_K(\bZ)}\kr)^{\frac{1}{2}}\d u\\
\lesssim& |U|^{2\chi_K(\bZ)+\frac{1}{2}}V^{-\frac{1}{2}}.
\end{Eq*}
Noticing that $V\geq |U|^{1+(\varepsilon_2-\varepsilon_1)/R}$ and $|U|>|U_*|$ in $\rD_{u;-\infty,U_*}^{V_*,\infty}$,
this inequality means that 
\begin{Eq*}
\varE_u[\phi^{(\bZ)},0](U,V)\leq \kl(1+C_1|U_*|^{-\frac{\varepsilon_2-\varepsilon_1}{2R}}\kr)M_1|U|^{2\chi_K(\bZ)}
\end{Eq*}
with some constant $C_1$ independent with $U_*$.

\subsection{Estimate of the current terms}
Again, we only give the discussion of the case that $\chi(\bZ)=2$. 
The cases that $\chi(\bZ)\leq 1$ are left to the interested readers.

In the beginning, by \eqref{Eq:Efo_JZ}, we know
\begin{Eq}\label{Eq:Efo_JZ2}
|J^{(Z_1,Z_2)}_\mu|\lesssim&|\phi^{(Z_1,Z_2)}|| \hD_\mu\phi|+|\phi^{(Z_2)}||\hD_\mu \phi^{(Z_1)}|+|\phi^{(Z_1)}|| \hD_\mu \phi^{(Z_2)}|+|\phi|| \hD_\mu \phi^{(Z_1,Z_2)}|\\
&\quad+|\phi||\phi^{(Z_1)}||F_{Z_2\mu}|+|\phi||\phi^{(Z_2)}||F_{Z_1\mu}|+|\phi|^2|F^{(Z_1)}_{Z_2\mu}|+|\phi|^2|F_{[Z_2,Z_1]\mu}| \\
\equiv& \sum_{j=1}^8 I_{j;\mu}(Z_1,Z_2).
\end{Eq}
Among this, using \eqref{Eq:Eo_phiZ_osiur}, \eqref{Eq:Eo_DphiZ_osiur}, \eqref{Eq:Ro_F0Z} and \eqref{Eq:Eo_FrZ_osiur}, we are able to find that
\begin{Eq*}
\kl(\sum_{j=1}^3+\sum_{j=5}^8\kr)\|I_{j;X}\|_{L_\omega^2}\lesssim&
\begin{cases}
|u|^{\chi_K(\bZ)-2+\varepsilon_1+\delta}v^{-2},&X=\Lb,\\
|u|^{\chi_K(\bZ)-1+\varepsilon_1+\delta}v^{-3},&X=e_i,~(i=1,2),\\
|u|^{\chi_K(\bZ)+\varepsilon_1+\delta}v^{-4},&X=L,
\end{cases}\\
\|I_{4;X}\|_{L_\omega^2}\lesssim& |u|^{\frac{-1+\varepsilon_1}{2}}v^{-1}\| \hD_X \phi^{(\bZ)}\|_{L_\omega^2}.
\end{Eq*}
Combining these and using \eqref{Eq:Eoeiur-FphiJ}, we get that
\begin{Eq*}
\varE_u[J^{(\bZ)}](U,V)
=&\int_{\rHb_{U,U_*}^V}|u|^{2-2\varepsilon_1}v^2 \kl(|u|^2|J^{(\bZ)}_{\Lb}|^2+v^2|\sJ^{(\bZ)}|^2\kr)+\int_{\rH_U^{V,\infty}}|u|^{2-2\varepsilon_1}v^2\kl(|u|^2|\sJ^{(\bZ)}|^2+v^2|J^{(\bZ)}_L|^2\kr)\\
\lesssim&\int_U^{U_*} |u|^{2\chi_K(\bZ)+2\delta}\d u+|U|\int_{\rHb_{U,U_*}^{V}}|u|^{-\varepsilon_1}\kl(|u|^2|\hD_\Lb \phi^{(\bZ)}|^2+v^2|\sD \phi^{(\bZ)}|^2\kr)\\
&\quad+\int_V^\infty |U|^{2\chi_K(\bZ)+2+2\delta}v^{-2}\d v+|U|\int_{\rH_{U}^{V,\infty}}|u|^{-\varepsilon_1}\kl(|u|^2|\sD\phi^{(\bZ)}|^2+v^2|\hD_\Lb\phi^{(\bZ)}|^2\kr)\\
\lesssim&|U|^{2\chi_K(\bZ)+1+2\delta}+|U|^{2\chi_K(\bZ)+2+2\delta}V^{-1}+|U|^{2\chi_K(\bZ)+1}\\
\lesssim & |U_*|^{-2\delta}|U|^{2\chi_K(\bZ)+1+4\delta}.
\end{Eq*}


\subsection{Close of \Bz{Bz:u}}
Through the discussions in last few subsections, we find that
\begin{Eq*}
\varE_u[\phi^{(\bZ)},\rF^{(\bZ)}](U,V)
\leq & \kl(2+C_1|U_*|^{\frac{-1+\varepsilon_1}{2}+2\delta}+C_1|U_*|^{-\frac{\varepsilon_2-\varepsilon_1}{2R}}\kr)M_1|U|^{2\chi_K(\bZ)},\\
\varE_u[J^{(\bZ)}](U,V)\leq&(C_1|U_*|^{-2\delta})M_1|U|^{2\chi_K(\bZ)+1+4\delta}
\end{Eq*}
with the constant $C_1$ does not depend on $U_*$. 
Thus, defining $U_1:=-1-{C_1}^{\max\kl\{\frac{2R}{\varepsilon_2-\varepsilon_1},\frac{1}{2\delta}\kr\}}$, we find
\begin{Eq*}
\varE_u[\phi^{(\bZ)},\rF^{(\bZ)}](U,V)\leq & 4M_1|U|^{2\chi_K(\bZ)},\\
\varE_u[J^{(\bZ)}](U,V)\leq&M_1|U|^{2\chi_K(\bZ)+1+4\delta},
\end{Eq*}
for all $U\leq U_*\leq U_1$, $V\geq V_*$.
This closes the \Bz{Bz:u}.

\subsection{Additional discussion on $\{v=|u|^{1+(\varepsilon_2-\varepsilon_1)/R}\}$}
To discuss the $(F,\phi)$ on $\rD_{l;-\infty,U_*}^{V_*,\infty}$, 
we should first give its behavior on the upper boundary of $\rD_{l;-\infty,U_*}^{V_*,\infty}$.
Thus, for any $U\leq U_*$, we use the \eqref{Eq:Ee} with $\rD=\rD_{u;U,U_*}^{V_*,\infty}$.
Using \eqref{Eq:Eoeiur-FphiJ}
and a simple modification of the processes in last few subsections, we easily find that for any $\chi(\bZ)\leq 2$ and $M_2$ much larger than $M_1$, we have
\begin{Eq}\label{Eq:Eo_FphiZ_o_uC}
&\int_{\varC_U} |u|^{-\varepsilon_1}\kl(|u|^2|(\alphab^{(\bZ)},\hD_\Lb \phi^{(\bZ)})|^2+v^2|(\rrho^{(\bZ)},\sigma^{(\bZ)},\sD\phi^{(\bZ)})|^2\kr)\vec{n}^\Lb\\
&\phantom{\int_{\rH_U}}\quad+|u|^{-\varepsilon_1}\kl(|u|^2|(\rrho^{(\bZ)},\sigma^{(\bZ)},\sD \phi^{(\bZ)})|^2+v^2|(\alpha^{(\bZ)},\hD_{L} \phi^{(\bZ)})|^2\kr)\vec{n}^L 
\leq M_2|U|^{2\chi_K(\bZ)},
\end{Eq}
where $\varC_U:=\{U\leq u\leq U_*, ~V_*\leq v=|u|^{1+(\varepsilon_2-\varepsilon_1)/R}\}$ and $\vec{n}$ is the normal vector of $\varC_U$.

\section{The analysis in exterior region, part 2}\label{Sn:Taierp2}
\subsection{The bootstrap ansatz}\label{Sn:Tbailr}
In this section, 
we give the discussion of $(F,\phi)$ in $\rD_{l;-\infty,U_*}^{V_*,\infty}$.
Thus, in what follows, 
we always assume $(u,v,\omega)\in \rD_{l;-\infty,U_*}^{V_*,\infty}$.
For the writing convenience, we denote $V_u=|u|^{1+(\varepsilon_2-\varepsilon_1)/R}$ and $U_v=-v^{R/(\varepsilon_2-\varepsilon_1+R)}$.
We also fix the $\kappa$ in \Sn{Tmvfaee} to be $|u|^{-R-\varepsilon_2}v^R$. 
Then, in \eqref{Eq:Eet} we have
\begin{Eq*}
-\frac{u^2L\kappa}{2}=&-\frac{R}{2}|u|^{2-R-\varepsilon_2}v^{-1+R}< 0,\quad -\frac{v^2\Lb\kappa}{2}=-\frac{R+\varepsilon_2}{2}|u|^{-1-R-\varepsilon_2}v^{2+R}<0,\\
 -\frac{K\kappa}{4}=&-\frac{1}{2}|u|^{-R-\varepsilon_2}v^R(Rv+(R+\varepsilon_2)|u|)<0.
\end{Eq*}

We define the energy corresponding to fields $[f,G]$ to be
\begin{Eq*}
\varE_l[f,G](U,V):=&\int_{\rHb_{U,U_V}^{V}}|u|^{-R-\varepsilon_2}v^R\kl(|u|^2\big|(\alphab[G],\hD_\Lb f)\big|^2+v^2\big|(\rho[G],\sigma[G],\sD f)\big|^2\kr)\\
&\quad+\int_{\rH_U^{V,V_U}}|u|^{-R-\varepsilon_2}v^R\kl(|u|^2\big|(\rho[G],\sigma[G],\sD f)\big|^2+v^2\big|(\alpha[G],\hD_{L} f)\big|^2\kr),
\end{Eq*}
and define the energy corresponding to the current $J$ to be
\begin{Eq*}
\varE_l[J](U,V)
:=&\int_{\rHb_{U,U_V}^V}|u|^{2-2R-2\varepsilon_2}v^{2+2R} \kl(|u|^2|J_{\Lb}|^2+v^2|\sJ|^2\kr)\\
&\quad+\int_{\rH_U^{V,V_U}}|u|^{2-2R-2\varepsilon_2}v^{2+2R}\kl(|u|^2|\sJ|^2+v^2|J_L|^2\kr).
\end{Eq*}

Then, for any given $\varepsilon_1$, $\varepsilon_2$ and $\chi(\bZ)\leq 2$, 
we set the  bootstrap ansatz to be
\begin{btsz}\label{Bz:l}
\begin{Eq}\label{Eq:Eoeilr-FphiJ}
\varE_l[\phi^{(\bZ)},\rF^{(\bZ)}](U,V)\leq& 8M_2|U|^{2\chi_K(\bZ)},\\
\varE_l[J^{(\bZ)}](U,V)\leq& 2M_2|U|^{2\chi_K(\bZ)+1+4\delta}.
\end{Eq}
\end{btsz} 
We easily find that \eqref{Eq:Eoeilr-FphiJ} is held for $(U,V)\in \rD_{l;-\infty,U_*}^{V_*,\infty}$ which sufficiently close to $\{V=V_U\}$, with $(U_*,V_*)$ as above.

Then, we will show that 
there exists $U_2\leq -1$ and $R$ which both depend on the initial data,
such that the constant $8$ and $2$ in \eqref{Eq:Eoeilr-FphiJ} could be improved to $4$ and $1$ 
while $U\leq U_*\leq U_2$, $V\geq V_*$.
This means that \eqref{Eq:Eoeilr-FphiJ} will hold for all $(U,V)$ in such region.

To show this result, we first show several results while \eqref{Eq:Eoeilr-FphiJ} are held.

\subsection{Estimate of the Maxwell field on sphere}\label{Sn:EotMfosilr}
Here, we will show that
\begin{lemma}\label{La:Eo_FZ_osilr}
Assume that \eqref{Eq:Eo_FZ_osiur} and \eqref{Eq:Eoeilr-FphiJ} are satisfied.
For $\chi(\bZ)=1$, $p=2$, or $\chi(\bZ)=0$, $2\leq p<\infty$, we have
\begin{Eq}\label{Eq:Eo_FZ_osilr}
&|u|^2\|\alphab^{(\bZ)}\|_{L_\omega^p}+|u|v\kl\|\kl(\rrho^{(\bZ)},\sigma^{(\bZ)}\kr)\kr\|_{L_\omega^p}+v^2\|\alpha^{(\bZ)}\|_{L_\omega^p}
\lesssim|u|^{\chi_K(\bZ)+\frac{1+R+\varepsilon_2}{2}}v^{-\frac{R}{2}-1}.
\end{Eq}
\end{lemma}
\begin{corollary}\label{Cy:Eo_FrZ_osilr}
Assume that \eqref{Eq:Eo_FZ_osiur} and \eqref{Eq:Eoeilr-FphiJ} are satisfied.
For $\chi(\bZ)=1$, $p=2$, or $\chi(\bZ)=0$, $2\leq p<\infty$, we have
\begin{Eq}\label{Eq:Eo_FrZ_osilr}
|u|^2\|\rF^{(\bZ)}_{Z\Lb}\|_{L_\omega^p}+|u|v\sum_i\|\rF^{(\bZ)}_{Ze_i}\|_{L_\omega^p}+v^2\|\rF^{(\bZ)}_{ZL}\|_{L_\omega^p}\lesssim  |u|^{\chi_K(\bZ,Z)+\frac{1+R+\varepsilon_2}{2}}v^{-\frac{R}{2}}.
\end{Eq}
\end{corollary}

\begin{proof}[Proof of \La{La:Eo_FZ_osilr}]
\part[The estimate of $\alphab^{(\bZ)}$]
Similar to the estimate of $\alphab^{(\bZ)}$ in upper region, using \eqref{Eq:Eoeilr-FphiJ}, for $\chi(\bZ)\leq 1$, we have 
\begin{Eq*}
\int_{\rH_U^{V,V_U}}v^R |L(r\alphab^{(\bZ)})|^2
\lesssim&\int_{\rH_U^{V,V_U}}v^R\kl(\sum_{|\bOmega|\leq 1}\kl|\kl(\rrho^{(\bOmega,\bZ)},\sigma^{(\bOmega,\bZ)}\kr)\kr|^2+\kl|r\sJ^{(\bZ)}\kr|^2\kr)\\
\lesssim&  |U|^{2\chi_K(\bZ)-2+R+\varepsilon_2}+ |U|^{2\chi_K(\bZ)-3+2R+2\varepsilon_2+4\delta}V^{-R}\\
\lesssim&|U|^{2\chi_K(\bZ)-3+R+\varepsilon_2}V.
\end{Eq*}
Then, using \eqref{Eq:key_Eob1_2} and \eqref{Eq:Eo_FZ_osiur}, for $\chi(\bZ)\leq 1$,  we find 
\begin{Eq*}
\kl.\int_{\rS^2}|r\alphab^{(\bZ)}|^2\d\omega\kr|_{v=V}
\lesssim& \kl.\int_{\rS^2}|r\alphab^{(\bZ)}|^2\d\omega\kr|_{v=V_u}+V^{-1-R}\int_V^{V_u}\int_{\rS^2} v^{2+R}|L(r\alphab^{(\bZ)})|^2\d \omega\d v\\
\lesssim& |u|^{2\chi_K(\bZ)-3+R+\varepsilon_2}V^{-R}.
\end{Eq*}
This gives the $\alphab$ part of \eqref{Eq:Eo_FZ_osilr}.

\part[The estimate of $\rrho^{(\bZ)}$ and $\sigma^{(\bZ)}$]
Similar to the process in upper region, using \eqref{Eq:Eoeilr-FphiJ}, for $\chi(\bZ)\leq 1$, we have 
\begin{Eq*}
\int_{\rH_U^{V,V_U}} v^R\kl|L\kl(r^2(\rrho^{(\bZ)},\sigma^{(\bZ)})\kr)\kr|^2
\lesssim&\int_{\rH_U^{V,V_U}} v^R\kl(\sum_{|\bOmega|\leq 1}\kl|r\alpha^{(\bOmega,\bZ)}\kr|^2+|r^2J^{(\bZ)}_L|^2\kr)\\
\lesssim& |U|^{2\chi_K(\bZ)-1+R+\varepsilon_2}V.
\end{Eq*}
Now, using \eqref{Eq:key_Eob1_2} and \eqref{Eq:Eo_FZ_osiur},  for $\chi(\bZ)\leq 1$,  we similarly find 
\begin{Eq*}
\int_{\rS^2}\kl|r^2(\rrho^{(\bZ)},\sigma^{(\bZ)})\kr|_{v=V}^2\d\omega\lesssim& |u|^{2\chi_K(\bZ)-1+R+\varepsilon_2}V^{-R}.
\end{Eq*}
This gives the $\rrho$ part and the $\sigma$ part of \eqref{Eq:Eo_FZ_osilr}.

\part[The estimate of $\alphab^{(\bZ)}$]
Similar to the process in the upper region, using \eqref{Eq:Eoeilr-FphiJ}, for $\chi(\bZ)\leq 1$, we have
\begin{Eq*}
\int_{\rH_U^{V,V_U}} v^R|L(v^3\alpha^{(\bZ)}_i)|^2
\lesssim&\int_{\rH_U^{V,V_U}}v^{2+R}|\alpha^{(K,\bZ)}|^2+|U|^2\int_{\rH_U^{V,V_U}}v^{2+R}|\alpha^{(\bZ)}|^2\\
&+|U|^4\int_{\rH_U^{V,V_U}}v^R\kl(\sum_{|\bOmega|\leq 1}\kl|\kl(\rrho^{(\bOmega,\bZ)},\sigma^{(\bOmega,\bZ)}\kr)\kr|^2+\kl|r\sJ^{(\bZ)}\kr|^2\kr)\\
\lesssim&|U|^{2\chi_K(\bZ)+1+R+\varepsilon_2}V.
\end{Eq*}
Now, using \eqref{Eq:key_Eob1_2} and \eqref{Eq:Eo_FZ_osiur}, for $\chi(\bZ)\leq 1$, we again  find 
\begin{Eq*}
\int_{\rS^2}\kl|v^3\alpha^{(\bZ)}\kr|_{v=V}^2\d\omega\lesssim& |u|^{2\chi_K(\bZ)+1+R+\varepsilon_2}V^{-R}.
\end{Eq*}
This gives the $\alpha$ part of \eqref{Eq:Eo_FZ_osilr}, and finishes the proof.
\end{proof}

\begin{proof}[Proof of \Cy{Cy:Eo_FrZ_osilr}]
The proof of \Cy{Cy:Eo_FrZ_osilr} is the most same with that of \Cy{Cy:Eo_FrZ_osiur}. So we omit it here.
\end{proof}

\subsection{Estimate of the scalar field on sphere}\label{Sn:Eotsfosilr}
Here, we will show that
\begin{lemma}\label{La:Eo_phiZ_osilr}
Assume that \eqref{Eq:Eo_phiZ_osiur} and \eqref{Eq:Eoeilr-FphiJ} are satisfied.
Then, we have 
\begin{Eq}\label{Eq:Eo_phiZ_osilr}
\|\phi^{(\bZ)}\|_{L_\omega^p}
\lesssim&\begin{cases}
 |u|^{\chi_K(\bZ)+\frac{-1+R+\varepsilon_2}{2}}v^{-1-\frac{R}{2}},&
\begin{aligned}
&\chi(\bZ)=2,~p=2,~ or~\chi(\bZ)=1,~2\leq p<\infty,\\
&\quad or~\chi(\bZ)=0,~2\leq p\leq\infty,
\end{aligned}\\
|u|^{\chi_K(\bZ)+\frac{-1+R+\varepsilon_2}{2}+\delta}v^{-1-\frac{R}{2}},& 
\begin{aligned}
&\chi(\bZ)=2,~2<p\leq p_\delta,~ or~\chi(\bZ)=1,~p=\infty.
\end{aligned}
\end{cases}
\end{Eq}
We also have
\begin{Eq}\label{Eq:Eo_DphiZ_osilr}
&|u|^2\|\hD_\Lb\phi^{(\bZ)}\|_{L_\omega^p}+|u|v\|\sD\phi^{(\bZ)}\|_{L_\omega^p}+v^2\|\hD_L\phi^{(\bZ)}\|_{L_\omega^p}\\
\lesssim& \begin{cases}
|u|^{\chi_K(\bZ)+\frac{1+R+\varepsilon_2}{2}}v^{-1-\frac{R}{2}},&\chi(\bZ)=1,~p=2,~or~\chi(\bZ)=0,~2\leq p<\infty,\\
|u|^{\chi_K(\bZ)+\frac{1+R+\varepsilon_2}{2}+\delta}v^{-1-\frac{R}{2}},& \chi(\bZ)=1,~2<p\leq p_\delta,~or~\chi(\bZ)=0,~p=\infty.
\end{cases}
\end{Eq}
\end{lemma}
\begin{corollary}\label{Cy:So_J_osilr}
Assume that \eqref{Eq:Eo_phiZ_osiur} and \eqref{Eq:Eoeilr-FphiJ} are satisfied.
For $\chi(\bZ)=1$, $p=2$ or $\chi(\bZ)=0$, $2\leq p<\infty$, we have
\begin{Eq}\label{Eq:Eo_J_osilr}
|u|^2\|J^{(\bZ)}_\Lb\|_{L_\omega^p}+|u|v\|\sJ^{(\bZ)}\|_{L_\omega^p}+v^2\|J^{(\bZ)}_L\|_{L_\omega^p}\lesssim |u|^{\chi_K(\bZ)+R+\varepsilon_2}v^{-2-R}.
\end{Eq}
\end{corollary}

\begin{proof}[Proof of \La{La:Eo_phiZ_osilr}]
At first, using \eqref{Eq:key_Eob1_1}, \eqref{Eq:Eo_phiZ_osiur} and \eqref{Eq:Eoeilr-FphiJ}, 
for $\chi(\bZ)\leq 2$, we have
\begin{Eq*}
&V^{R-1}\kl.\int_{\rS^2}|r\phi^{(\bZ)}|^2\d\omega\kr|_{v=V}+ R\int_V^{V_u}\int_{\rS^2}v^{-2+R}|r\phi^{(\bZ)}|^2\d\omega\d v\\
\lesssim & V_u^{R-1}\kl.\int_{\rS^2}|r\phi^{(\bZ)}|^2\d\omega\kr|_{v=V_u}+R^{-1}\int_V^{V_u}\int_{\rS^2}V^{-2}v^{2+R}|D_L(r\phi^{(\bZ)})|^2\d\omega\d v\\
\lesssim&  |u|^{2\chi_K(\bZ)-1+R+\varepsilon_2}V^{-1}.
\end{Eq*}
This gives the first part of \eqref{Eq:Eo_phiZ_osilr}. 

Next, using the \eqref{Eq:key_42}, the inequality just obtained and \eqref{Eq:Eoeilr-FphiJ},
for $\chi(\bZ)=2$ and some constant $C$, we find
\begin{Eq*}
&\kl(\int_{\rS^2}|r^{R/2}\phi^{(\bZ)}|^4\d \omega\kr)^{\frac{1}{2}}_{v=V}-C\kl(\int_{\rS^2}|r^{R/2}\phi^{(\bZ)}|^4\d \omega\kr)^{\frac{1}{2}}_{v=V_u}\\
\lesssim&\int_V^{V_u}\int_{\rS^2}|D_L^{\leq 1}(r^{R/2}\phi^{(\bZ)})|^2+|D_\Omega (r^{R/2}\phi^{(\bZ)})|^2\d\omega\d v\\
\lesssim&\int_V^{V_u}\int_{\rS^2}Rv^{-2+R}|r\phi^{(\bZ)}|^2+v^{2+R}|\sD \phi^{(\bZ)}|^2+v^R|\hD_L\phi^{(\bZ)}|^2\d\omega\d v\\
\lesssim& |u|^{2\chi_K(\bZ)-2+R+\varepsilon_2}.
\end{Eq*}
Now, combining this with \eqref{Eq:Eo_L4_phiZ2_iur}, we calculate that
\begin{Eq*}
\|\phi^{(\bZ)}\|_{L_\omega^4}
\lesssim |u|^{\chi_K(\bZ)-1+\frac{R+\varepsilon_2}{2}}v^{-\frac{R}{2}}\lesssim |u|^{\chi_K(\bZ)+\frac{R+\varepsilon_2}{2}+\frac{\varepsilon_2}{R}}v^{-1-\frac{R}{2}}.
\end{Eq*}
Using the interpolation, for $2<p<4$ and $\chi(\bZ)=2$, we find
\begin{Eq*}
\|\phi^{(\bZ)}\|_{L_\omega^{p}}\lesssim \|\phi^{(\bZ)}\|_{L_\omega^2}^{\frac{4}{p}-1}\|\phi^{(\bZ)}\|_{L_\omega^4}^{2-\frac{4}{p}}\lesssim  |u|^{\chi_K(\bZ)+\frac{-1+R+\varepsilon_2}{2}+\frac{(R+2\varepsilon_2)(p-2)}{Rp}}v^{-1-\frac{R}{2}}.
\end{Eq*}
Noticing that $\frac{(R+2\varepsilon_2)(p-2)}{Rp}\leq\frac{(R+2\varepsilon_2)(p_\delta-2)}{Rp_\delta}= \frac{(R+2\varepsilon_2)\delta}{2R}<\delta$, 
we finish the proof of \eqref{Eq:Eo_phiZ_osilr}. 

Finally, the proof of \eqref{Eq:Eo_DphiZ_osilr} is most the same with that of \eqref{Eq:Eo_DphiZ_osiur}. So we omit it here.
\end{proof}

\begin{proof}[Proof of \Cy{Cy:So_J_osilr}]
Again, the proof of \Cy{Cy:So_J_osilr} is most the same with that of \Cy{Cy:So_J_osiur}. We left it to the interested reader.
\end{proof}

\subsection{Energy estimates of Maxwell field}
Here, we use the \eqref{Eq:Ee} with $\rD=\rDb$, $[f,G]=[0,\rF^{(\bZ)}]$ and $\kappa=|u|^{-R-\varepsilon_2}v^R$.
By the discussions in \Sn{Tmvfaee} and \Sn{Sn:Tbailr}, 
noticing that $|u|^{-R-\varepsilon_2}v^R=|u|^{-\varepsilon_1}$ on $\{v=v_u\}$ 
and using \eqref{Eq:Eo_FphiZ_o_uC}, we know
\begin{Eq*}
&\varE_l[0,\rF^{(\bZ)}](U,V)-M_2|U|^{2\chi_K(\bZ)}
\lesssim\int_\rDb |u|^{-R-\varepsilon_2}v^{R}|(\rF^{(\bZ)})_K{}^{\mu}\nabla^{\nu}(\rF^{(\bZ)})_{\nu\mu}|\\
\lesssim&\int_\rDb |u|^{2-R-\varepsilon_2}v^{R}\kl(|\rrho^{(\bZ)}||J^{(\bZ)}_\Lb|+|\alphab^{(\bZ)}||\sJ^{(\bZ)}|\kr)+|u|^{-R-\varepsilon_2}v^{2+R}\kl(|\rrho^{(\bZ)}||J^{(\bZ)}_L|+|\alpha^{(\bZ)}||\sJ^{(\bZ)}|\kr).
\end{Eq*}
Then, similar to the upper region, using \eqref{Eq:Eoeilr-FphiJ} and noticing $V\geq|U|\geq|U_*|$, we find
\begin{Eq*}
&\varE_l[0,\rF^{(\bZ)}](U,V)-M_2|U|^{2\chi_K(\bZ)}\\
\lesssim &\int_V^{V_U}|U|^{\frac{R+\varepsilon_2}{2}}v^{-2-\frac{R}{2}}\kl(\int_{\rHb_{U,U_v}^v}|u|^{-R-\varepsilon_2}v^{2+R}|\rrho^{(\bZ)}|^2\kr)^{\frac{1}{2}}
\kl(\int_{\rHb_{U,U_v}^{v}}|u|^{4-2R-2\varepsilon_2}v^{2+2R}|J^{(\bZ)}_\Lb|^2\kr)^{\frac{1}{2}}\d v\\
&\quad+\int_V^{V_U}|U|^{\frac{R+\varepsilon_2}{2}}v^{-2-\frac{R}{2}}\kl(\int_{\rHb_{U,U_v}^v}|u|^{2-R-\varepsilon_2}v^R|\alphab^{(\bZ)}|^2\kr)^{\frac{1}{2}}
\kl(\int_{\rHb_{U,U_v}^v}|u|^{2-2R-2\varepsilon_2}v^{4+2R}|\sJ^{(\bZ)}|^2\kr)^{\frac{1}{2}}\d v\\
&\quad +\int_U^{U_V}|u|^{-2+\frac{R+\varepsilon_2}{2}}V^{-\frac{R}{2}}\kl(\int_{\rH_u^{V,V_u}}|u|^{2-R-\varepsilon_2}v^R|\rrho^{(\bZ)}|^2\kr)^{\frac{1}{2}}
\kl(\int_{\rH_u^{V,V_u}}|u|^{2-2R-2\varepsilon_2}v^{4+2R}|J^{(\bZ)}_L|^2\kr)^{\frac{1}{2}}\d u\\
&\quad+\int_U^{U_V}|u|^{-2+\frac{R+\varepsilon_2}{2}}V^{-\frac{R}{2}}\kl(\int_{\rH_u^{V,V_u}}|u|^{-R-\varepsilon_2}v^{2+R}|\alpha^{(\bZ)}|^2\kr)^{\frac{1}{2}}
\kl(\int_{\rH_u^{V,V_u}}|u|^{4-2R-2\varepsilon_2}v^{2+2R}|\sJ^{(\bZ)}|^2\kr)^{\frac{1}{2}}\d u\\
\lesssim&\int_V^{\infty}|U|^{2\chi_K(\bZ)+\frac{1+R+\varepsilon_2}{2}+2\delta}v^{-2-\frac{R}{2}}\d v
+\int_U^{U_*}|u|^{2\chi_K(\bZ)+\frac{-3+R+\varepsilon_2}{2}+2\delta}V^{-\frac{R}{2}}\d u\\
\lesssim& |U_*|^{\frac{-1+\varepsilon_2}{2}+2\delta}|U|^{2\chi_K(\bZ)}.
\end{Eq*}
This means that for some constant $C_2$ that does not depend on $U_*$ and $R$, we have
\begin{Eq*}
\varE_l[0,\rF^{(\bZ)}](U,V)\leq \kl(1+C_2|U_*|^{\frac{-1+\varepsilon_2}{2}+2\delta}\kr)M_2|U|^{2\chi_K(\bZ)}.
\end{Eq*}

\subsection{Energy estimates of scalar field}
Again, we only consider the case that $\bZ=(Z_1,Z_2)$.
We use the \eqref{Eq:Ee} with $\rD=\rDb$, $[f,G]=[\phi^{(\bZ)},0]$ and $\kappa=|u|^{-R-\varepsilon_2}v^R$.
Through the discussions in \Sn{Tmvfaee} and \Sn{Sn:Tbailr}, we know
\begin{Eq*}
&\varE_l[\phi^{(\bZ)},0](U,V)-M_2|U|^{2\chi_K(\bZ)}+\frac{R}{2}\int_{\rDb}|u|^{-1-R-\varepsilon_2}v^{-1+R}\kl(|u|^3|\hD_\Lb\phi^{(\bZ)}|^2+v^3|\hD_L\phi^{(\bZ)}|^2\kr)\\
\leq&\int_\rDb |u|^{-R-\varepsilon_2}v^R\kl(|\Re(\overline{\square_A\phi^{(\bZ)}}\hD_K\phi^{(\bZ)})|+|\Im(\phi^{(\bZ)}\overline{\hD^\mu \phi^{(\bZ)}})F_{K\mu}|\kr)\\
\lesssim&\int_\rDb|u|^{2-R-\varepsilon_2}v^R\kl(\kl|\kl(\square_A\phi^{(\bZ)},\phi^{(\bZ)}\rho\kr)\kr||\hD_\Lb\phi^{(\bZ)}|+|\phi^{(\bZ)}\alphab||\sD \phi^{(\bZ)}|\kr)\\
&\phantom{\int_\rDb}\quad +|u|^{-R-\varepsilon_2} v^{2+R}\kl(\kl|\kl(\square_A\phi^{(\bZ)},\phi^{(\bZ)}\rho\kr)\kr||\hD_L \phi^{(\bZ)}|+|\phi^{(\bZ)}\alpha||\sD \phi^{(\bZ)}|\kr)
\end{Eq*}
Using the notations $\{I_j\}_{j=1}^5$ defined in \eqref{Eq:Efo_boxphiZ2},
similar to the process in upper region,
by \eqref{Eq:Eo_Q}, \eqref{Eq:Ro_F0Z} and results in \Sn{Sn:EotMfosilr} and \Sn{Sn:Eotsfosilr}, 
we know
\begin{Eq*}
P_0:=\kl\|\kl(\square_A\phi^{(\bZ)},\phi^{(\bZ)}\rho\kr)\kr\|_{L_\omega^2}\lesssim  |u|^{\chi_K(\bZ)+\frac{-1+R+\varepsilon_2}{2}}v^{-3-\frac{R}{2}}.
\end{Eq*}

Now, by this estimate, \eqref{Eq:Eo_FZ_osilr}, \eqref{Eq:Eo_phiZ_osilr} and \eqref{Eq:Eoeilr-FphiJ}, noticing $V\geq|U|\geq |U_*|$, we similarly get that
\begin{Eq*}
&\varE_l[\phi^{(\bZ)},0](U,V)-M_2|U|^{2\chi_K(\bZ)}+\frac{R}{2}\int_{\rDb}|u|^{-1-R-\varepsilon_2}v^{-1+R}\kl(|u|^3|\hD_\Lb\phi^{(\bZ)}|^2+v^3|\hD_L\phi^{(\bZ)}|^2\kr)\\
\lesssim&  \kl(\int_V^{V_U}\int_{U}^{U_v} |u|^{2-R-\varepsilon_2}v^{3+R}P_0^2\d u\d v\kr)^\frac{1}{2}\kl(\int_{\rDb}|u|^{2-R-\varepsilon_2}v^{-1+R}|\hD_\Lb\phi^{(\bZ)}|^2\kr)^{\frac{1}{2}}\\
&\quad+\int_U^{U_V} \kl(\int_V^{V_u}|u|^{2-R-\varepsilon_2}v^{2+R}\|\phi^{(\bZ)}\|_{L_\omega^{p_\delta}}^2\|\alphab\|_{L_\omega^{q_\delta}}^2\d v\kr)^{\frac{1}{2}}\kl(\int_{\rH_{u}^{V,V_u}}|u|^{2-R-\varepsilon_2}v^R|\slashed D\phi^{(\bZ)}|^2\kr)^{\frac{1}{2}}\d u\\
&\quad+\kl(\int_U^{U_V}\int_V^{V_u} |u|^{1-R-\varepsilon_2}v^{4+R}P_0^2\d v\d u\kr)^\frac{1}{2}\kl(\int_{\rDb}|u|^{-1-R-\varepsilon_2}v^{2+R}|\hD_L\phi^{(\bZ)}|^2\kr)^{\frac{1}{2}}\\
&\quad+\int_V^{V_U} \kl(\int_U^{U_v}|u|^{-R-\varepsilon_2}v^{4+R}\|\phi^{(\bZ)}\|_{L_\omega^{p_\delta}}^2\|\alpha\|_{L_\omega^{q_\delta}}^2\d u\kr)^{\frac{1}{2}}\kl(\int_{\rHb_{U,U_v}^v}|u|^{-R-\varepsilon_2}v^{2+R}|\slashed D\phi^{(\bZ)}|^2\kr)^{\frac{1}{2}}\d v\\
\lesssim& |U|^{\chi_K(\bZ)}\kl(\int_{\rDb}|u|^{-1-R-\varepsilon_2}v^{-1+R}\kl(|u|^3|\hD_\Lb\phi^{(\bZ)}|^2+v^3|\hD_L\phi^{(\bZ)}|^2\kr)\kr)^{\frac{1}{2}}+|U_*|^{\frac{-1+\varepsilon_2}{2}+\delta}|U|^{2\chi_K(\bZ)}.
\end{Eq*}
By applying $ab\leq \varepsilon a^2+4\varepsilon^{-1}b^2$ to the first term in right hand side and choosing $R$ large enough, we finally reach that
\begin{Eq*}
\varE_l[\phi^{(\bZ)},0](U,V)\leq \kl(1+1+C_2|U_*|^{\frac{-1+\varepsilon_2}{2}+\delta}\kr)M_2|U|^{2\chi_K(\bZ)},
\end{Eq*}
with some constant $C_2$ independent with $U_*$ and $R$.

\subsection{Estimate of the current terms}
Again, we only give the discussion of the case that $\chi(\bZ)=2$. 
Here we use the notations $\{I_{\mu;j}(Z_1,Z_2)\}_{j=1}^8$ defined in \eqref{Eq:Efo_JZ2}. 
Then, using \eqref{Eq:Eo_phiZ_osilr}, \eqref{Eq:Eo_DphiZ_osilr}, \eqref{Eq:Ro_F0Z} and \eqref{Eq:Eo_FrZ_osilr}, we are able to find that
\begin{Eq*}
\kl(\sum_{j=1}^3+\sum_{j=5}^8\kr)\|I_{j;X}\|_{L_\omega^2}\lesssim&
\begin{cases}
|u|^{\chi_K(\bZ)-2+R+\varepsilon_2+\delta}v^{-2-R},&X=\Lb,\\
|u|^{\chi_K(\bZ)-1+R+\varepsilon_2+\delta}v^{-3-R},&X=e_i,~(i=1,2),\\
|u|^{\chi_K(\bZ)+R+\varepsilon_2+\delta}v^{-4-R},&X=L,
\end{cases}\\
\|I_{4;X}\|_{L_\omega^2}\lesssim& |u|^{\frac{-1+\varepsilon_2+R}{2}}v^{-1-\frac{R}{2}}\| \hD_X \phi^{(\bZ)}\|_{L_\omega^2}.
\end{Eq*}
Combining these and  using \eqref{Eq:Eoeilr-FphiJ}, we similarly get that
\begin{Eq*}
\varE_l[J^{(\bZ)}](U,V)
=&\int_{\rHb_{U,U_V}^V}|u|^{2-2R-2\varepsilon_2}v^{2+2R} \kl(|u|^2|J^{(\bZ)}_{\Lb}|^2+v^2|\sJ^{(\bZ)}|^2\kr)\\
&\quad+\int_{\rH_U^{V,V_U}}|u|^{2-2R-2\varepsilon_2}v^{2+2R}\kl(|u|^2|\sJ^{(\bZ)}|^2+v^2|J^{(\bZ)}_L|^2\kr)\\
\lesssim&\int_U^{U_*} |u|^{2\chi_K(\bZ)+2\delta}\d u+|U|\int_{\rHb_{U,U_V}^{V}}|u|^{-R-\varepsilon_2}v^R\kl(|u|^2|\hD_\Lb \phi^{(\bZ)}|^2+v^2|\sD \phi^{(\bZ)}|^2\kr)\\
&\quad+\int_V^\infty |U|^{2\chi_K(\bZ)+2+2\delta}v^{-2}\d v+|U|\int_{\rH_{U}^{V,V_U}}|u|^{-R-\varepsilon_2}v^R\kl(|u|^2|\sD\phi^{(\bZ)}|^2+v^{2}|\hD_\Lb\phi^{(\bZ)}|^2\kr)\\
\lesssim & |U_*|^{-2\delta}|U|^{2\chi_K(\bZ)+1+4\delta}.
\end{Eq*}

\subsection{Close of \Bz{Bz:l}}
Through the discussions in last few subsections, we find that
\begin{Eq*}
\varE_l[\phi^{(\bZ)},\rF^{(\bZ)}](U,V)
\leq & \kl(3+C_2|U_*|^{\frac{-1+\varepsilon_2}{2}+2\delta}\kr)M_2|U|^{2\chi_K(\bZ)},\\
\varE_l[J^{(\bZ)}](U,V)\leq&(C_2|U_*|^{-2\delta})M_2|U|^{2\chi_K(\bZ)+1+4\delta},
\end{Eq*}
with the constant $C_2$ does not depend on $U_*$ and $R$. 
Thus, defining $U_2:=\min\kl\{U_1,-1-{C_2}^{\frac{1}{2\delta}}\kr\}$, we find
\begin{Eq*}
\varE_l[\phi^{(\bZ)},\rF^{(\bZ)}](U,V)\leq & 4M_2|U|^{2\chi_K(\bZ)},\\
\varE_l[J^{(\bZ)}](U,V)\leq&M_2|U|^{2\chi_K(\bZ)+1+4\delta},
\end{Eq*}
for all $U\leq U_*\leq U_2$, $V\geq V_*$.
This closes the \Bz{Bz:l}.

\subsection{Proof of \Tm{Tm:Main2}}
Now, for the $R$ and $U_2$ fixed above, 
for any $U\leq U_*\leq U_2$, we use the \eqref{Eq:Ee} with $\rD=\rD_{l;U,U_*}^{V_*,V_U}$.
Using \eqref{Eq:Eoeilr-FphiJ}
and a simple modification of the processes in last few subsections, we easily find that
\begin{Eq*}
&\int_{\varC_U}|r|^{2-\varepsilon_2} |(\alphab^{(\bZ)},\rrho^{(\bZ)},\sigma^{(\bZ)},\alpha^{(\bZ)},\hD_\Lb \phi^{(\bZ)}, \sD\phi^{(\bZ)},\hD_{L} \phi^{(\bZ)})|^2\lesssim |U|^{2\chi_K(\bZ)},
\end{Eq*}
for any $\chi(\bZ)\leq 2$ and $\varC_U:=\{t=0,~2V_*\leq r\leq -2U\}$, since $v=|u|=2^{-1}r$ on $\varC_U$.
Now, using relations between \emph{Lie} derivatives and usual derivatives,
and that all $\partial_\mu$ can be expressed as the combination of $Z\in\{T,S,\Omega\}\subset\bZ$, 
we easily find
\begin{Eq*}
\sum_{|\beta|\leq 2}\int_{\varC_U}r^{2-\varepsilon_2+2|\beta|}\kl|\kl(DD^\beta\phi,\partial^\beta\rF\kr)\kr|^2\d x\lesssim 1,
\end{Eq*}
with the constant independent of $U$. 
Now, passing $U\rightarrow\infty$, we get the outer part of \eqref{Eq:Main2}.
On the other hand, using the result shows in the interior region with $U_*'=-V_*$,
and the fact that $\kl<r\kr>\approx 1$ while $r\leq 2V_*$, we finish the proof of \eqref{Eq:Main2}.

Finally, we take the integral of $\partial^\mu J_\mu$ on $\rD_{-V,V}^{0,V}$.
Then we find
\begin{Eq*}
\kl.\int_{\{|x|<2V\}}J_0\d x\kr|_{t=0}=\int_{\rHb_{-V,V}^V}J_\Lb=\kl.\int_{-V}^V \int_{\rS^2}\Im\kl((r\phi)\cdot \overline{ D_\Lb(r\phi)}\kr)\d\omega\d u\kr|_{v=V}.
\end{Eq*}
Passing $V\rightarrow\infty$, this gives
\begin{Eq*}
\kl.\int_{\rR^3}J_0\d x\kr|_{t=0}=\int_{-\infty}^\infty \int_{\rS^2}\Im\kl(\Phi\cdot \overline{ D_\Lb\Phi}\kr)\d\omega\d u,
\end{Eq*}
and proves that \eqref{Eq:crosd} implies  \eqref{Eq:Do_q0}. Now we finish the proof.

\section{Appendix}\label{Sn:A}

\subsection{Technical tools}

\begin{lemma}\label{La:key_Eob1}
Assume that $f,g\in C^1([s_1,s_2])$ with $g$ nonnegative and increasing, 
we have
\begin{align}
|f(s_2)|^2g^{-1}(s_2)+g^{-2}(s_2)\int_{s_1}^{s_2}|f|^2g'\d s\lesssim &|f(s_1)|^2g^{-1}(s_2)+\int_{s_1}^{s_2}|f'|^2(g')^{-1}\d s,\label{Eq:key_Eob1_2}\\
|f(s_2)|^2g^{-1}(s_2)+\int_{s_1}^{s_2}|f|^2g^{-2}g'\d s\lesssim &|f(s_1)|^2g^{-1}(s_1)+\int_{s_1}^{s_2}|f'|^2(g')^{-1}\d s.\label{Eq:key_Eob1_1}
\end{align}
\end{lemma}
\begin{corollary}\label{Cy:key_Eob2}
Assume $f,g$ as that in \La{La:key_Eob1} with additionally $f(s_1)=0$, we have
\begin{Eq}\label{Eq:key_Eob2}
|f(s_2)|^2g^{-1}(s_2)+\int_{s_1}^{s_2}|f|^2g^{-2}g'\d s\lesssim \int_{s_1}^{s_2}|f'|^2(g')^{-1}\d s,
\end{Eq}
even if $g(s_1)=0$.
\end{corollary}
\begin{corollary}\label{Cy:key_Eob3}
Assume $f$ as that in \La{La:key_Eob1} and $m>0$. For any $0<s_1,s_2$  we have
\begin{align}
\sup_{s\in(0,s_2]}s^{-m}|f(s)|^2\lesssim m^{-1}\int_{0}^{s_2} s^{-m-1}|(s\partial_s)^{\leq 1} f|^2\d s,\label{Eq:key_Eob3_2}\\
\sup_{s\in[s_1,\infty)}s^m|f(s)|^2 \lesssim m^{-1}\int_{s_1}^\infty s^{m-1}|(s\partial_s)^{\leq 1} f|^2\d s.\label{Eq:key_Eob3_1}\\
\end{align}
\end{corollary}

\begin{proof}[Proof of \La{La:key_Eob1}]
First we see
\begin{Eq*}
|f(s_2)|^2g(s_2)+\int_{s_1}^{s_2} |f|^2g'\d s
\leq &2|f(s_2)|^2g(s_2)+2\int_{s_1}^{s_2} |f||f'|g\d s\\
\leq &2g(s_2)|f(s_1)|^2+6g(s_2)\int_{s_1}^{s_2} |f||f'|\d s\\
\leq &2g(s_2)|f(s_1)|^2+72g^2(s_2)\int_{s_1}^{s_2} |f'|^2(g')^{-1}\d s+\frac{1}{2}\int_{s_1}^{s_2} |f|^2g'\d s.
\end{Eq*}
This gives the proof of \eqref{Eq:key_Eob1_2}. Here we also know
\begin{Eq*}
|f(s_2)|^2g^{-1}(s_2)+\int_{s_1}^{s_2}|f|^2g^{-2}g'\d s
= &|f(s_2)|^2g^{-1}(s_2)-\int_{s_1}^{s_2} |f|^2(g^{-1})'\d s\\
\leq&|f(s_1)|^2g^{-1}(s_1)+2\int_{s_1}^{s_2} |f||f'|g^{-1}\d s\\
\leq&|f(s_1)|^2g^{-1}(s_1)+8\int_{s_1}^{s_2} |f'|^2g'\d s+\frac{1}{2}\int_{s_1}^{s_2} |f|^2g^{-2}g'\d s.
\end{Eq*}
This gives the proof of \eqref{Eq:key_Eob1_1}. 
\end{proof}
\begin{proof}[Proof of \Cy{Cy:key_Eob2}]
We first use \eqref{Eq:key_Eob1_2} with $s_2=s_1'$ and know
\begin{Eq*}
|f(s_1')|^2g^{-1}(s_1') \lesssim \int_{s_1}^{s_1'}|f'|^2(g')^{-1}\d s. 
\end{Eq*} 
Taking $s_1'\rightarrow s_1$, we know
\begin{Eq*}
\lim_{s_1'\rightarrow s_1}|f(s_1')|^2g(s_1')^{-1}=0. 
\end{Eq*}
Now, \eqref{Eq:key_Eob2} follows from \eqref{Eq:key_Eob1_1}.
\end{proof}
\begin{proof}[Proof of \Cy{Cy:key_Eob3}]
For any given $f$, we define $f_n:=\xi(ns) f$ with $\xi$ is a smooth cut-off function satisfying $\xi|_{s\leq 1}=0$ and $\xi|_{s\geq 2}=1$.
Then, using \eqref{Eq:key_Eob2} to $f_n$ with $g(s)=s^m$ and passing $n\rightarrow\infty$, we get \eqref{Eq:key_Eob3_2}. The \eqref{Eq:key_Eob3_1} follows in a similar way.
\end{proof}

\begin{lemma}\label{La:key_42}
For any $f(s,\omega)$ with $(s,\omega)\in [s_1,s_2]\times\rS^2$, we have
\begin{Eq}\label{Eq:key_42}
\sup_{s\in[s_1,s_2]}\|f\|_{L_\omega^4}^2\lesssim\inf_{s\in[s_1,s_2]}\|f\|_{L_\omega^4}^2+\|\partial_{s,\omega}^{\leq 1} f\|_{L_{s,\omega}^{2}([s_1,s_2]\times \rS^2)}^2.
\end{Eq}
\begin{proof}
Set the region to be $[s_1,s_2]\times \rS^2$. 
On one hand, we find
\begin{Eq*}
\|f\|_{L_s^\infty L_\omega^4}^4=&\sup_{s\in[s_1,s_2]}\||f|^4\|_{L_\omega^1}\\
\lesssim& \inf_{s\in[s_1,s_2]}\||f|^4\|_{L_\omega^1}+\||\partial_sf||f|^3\|_{L_s^1L_\omega^1}\\
\lesssim&\inf_{s\in[s_1,s_2]}\|f\|_{L_\omega^4}^4+\|\partial_sf\|_{L_{s,\omega}^2}\|f\|_{L_{s,\omega}^6}^3.
\end{Eq*}
On the other hand, 
using the \emph{Gagliardo-Nirenberg} inequality, 
we have
\begin{Eq*}
\|f\|_{L_{s,\omega}^6}^3
\lesssim\kl\|\|f\|_{L_\omega^4}^{\frac{2}{3}}\|\partial_\omega^{\leq 1} f\|_{L_\omega^2}^{\frac{1}{3}}\kr\|_{L_s^6}^3
\lesssim\|f\|_{L_s^\infty L_\omega^4}^{2}\|\partial_\omega^{\leq 1} f\|_{L_s^2L_\omega^2}.
\end{Eq*}
Combining these two inequality, we get
\begin{Eq*}
\|f\|_{L_s^\infty L_\omega^4}^4\leq& C\kl(\inf_{s\in[s_1,s_2]}\|f\|_{L_\omega^4}^4+\|\partial_sf\|_{L_{s,\omega}^2}\|f\|_{L_s^\infty L_\omega^4}^{2}\|\partial_\omega^{\leq 1} f\|_{L_s^2L_\omega^2}\kr)\\
\leq&C'\kl(\inf_{s\in[s_1,s_2]}\|f\|_{L_\omega^4}^4+\|\partial_sf\|_{L_{s,\omega}^2}^4+\|\partial_\omega^{\leq 1} f\|_{L_s^2L_\omega^2}^4\kr)+\frac{1}{2}\|f\|_{L_s^\infty L_\omega^4}^4.
\end{Eq*}
This finishes the proof.
\end{proof}
\end{lemma}

\subsection{Proof of \Cm{Cm:Eo_tAphi_o_tH}}\label{Sn:PoCEo_tAphi_o_tH}
In this subsection, we omit the tilde on $\tL$, $\tA$, etc., since that we will only discuss in $(\tt,\tx)$ coordinate system.
We also omit the restrict notation $u=0$ since we will only discuss on $\rH_{0}^{0,1}$.
Finally, due to the denseness, without loss of generality, we assume $(\alpha,\phi)$ vanishes near $\{v=0\}$.

\subsubsection{Preparation}
At the beginning, we give the equations that connecting $A_\Lb$, $A_L$, $\sA$ and $\phi$.
Firstly, from \emph{Lorentz} gauge \eqref{Eq:gci_tD}, we know
\begin{Eq}\label{Eq:Ro_LbAL_ug}
\Lb A_L=-LA_\Lb+2r^{-1}(A_L-A_\Lb)+2\sdiv\sA.
\end{Eq}
Then, on one hand, by the expansion formula of $A_L$ and \eqref{Eq:MKGe-swe}, we have
\begin{Eq*}
\square A_L=&2(\nabla_\nu L^\mu)(\nabla^\nu A_\mu)+A_\mu\square L^\mu+L^\mu\square A_\mu\\
=&r^{-2}(A_L-A_\Lb)+2r^{-1}\sdiv \sA-J_L.
\end{Eq*}
On the other hand, by spherical expansion formula of $\square$ and \eqref{Eq:Ro_LbAL_ug}, we have
\begin{Eq*}
\square A_L=&-r^{-1}L(r\Lb A_L)+r^{-1}LA_L+r^{-2}\Delta_{\rS^2}A_L\\
=&r^{-2}L(rL(rA_\Lb))-r^{-1}LA_L-r^{-2}A_\Lb-2r^{-1}L(r\sdiv \sA)+r^{-2}\Delta_{\rS^2}A_L.
\end{Eq*}
Mixing them, we get the compatibility condition, that is
\begin{Eq}\label{Eq:c_eq_ALb}
L(rL(rA_\Lb))=&\sum_{\chi(\bL)\leq 1}r^{\chi(\bL)}\bL A_L-\Delta_{\rS^2}A_L
+\sum\Low{\chi(\bL)\leq1\\\chi(\bOmega)\leq 1}r^{\chi(\bL)}C^i_{\bL,\bOmega}(\omega)\bL\bOmega A_{e_i}-r^2J_L.
\end{Eq}

Meanwhile, using \eqref{Eq:MKGe-swe}, we also have
\begin{Eq}\label{Eq:c_eq_sA}
L\Lb (rA_{e_i})=&r^{-1}\sum\Low{\chi(\bOmega)\leq 1} C_{i,\bOmega}(\omega)\bOmega(A_L-A_\Lb)
+r^{-1}\sum\Low{\chi(\bOmega)\leq 2} C^j_{i,\bOmega}(\omega)\bOmega A_{e_j}-rJ_{e_i},
\end{Eq}
and
\begin{Eq}\label{Eq:c_eq_phi}
(L+\I A_L)\Lb(r\phi)=&-r^{-1}\Delta_{\rS^2}\phi-\I A_L\phi-\I rA_\Lb L\phi+2\I r\sm^{ij}A_{e_i}e_j\phi
+rA_LA_\Lb\phi-r\sm^{ij}A_{e_i}A_{e_j}\phi.
\end{Eq}

Now, for the given $\alpha$, we take $A_L=0$, let $\sA$ to be the unique solution (without singularity) of 
\begin{Eq}\label{Eq:S_sA_b_alpha}
r\alpha_i=&L(rA_{e_i})-re_i A_L=L(rA_{e_i}),
\end{Eq}
and let $A_\Lb$ to be the unique solution (without singularity) of \eqref{Eq:c_eq_ALb}.
Then, such $A$ is in respect with $\alpha$ and satisfying the compatibility condition.

Next, assume that there exists another $A'$ satisfying these conditions. 
We need to show that there exists a $\xi$ such that $A=A'-\d\xi$. 
Here we mention that on $\rH_{0}^{0,1}$, $\Lb\xi$ should be understood as the unique solution (without singularity) of 
\begin{Eq*}
L(r\Lb\xi)=L\xi+r^{-1}\Delta_{\rS^2}\xi,
\end{Eq*}
since that $\square\xi=\partial^\mu(A_\mu-A_\mu')=0$ under Lorentz gauge.

Then, we take $\xi=\int_0^vA_L'\d v'$ such that $A_L'-L\xi=0=A_L$. 
By the obvious uniqueness of \eqref{Eq:S_sA_b_alpha} and \eqref{Eq:c_eq_ALb}, 
we also have $A_{e_i}'-e_i\xi=A_{e_i}$ and $A_\Lb'-\Lb\xi=A_\Lb$.
This finishes the proof of uniqueness.
From now on, we fix the $(A,\phi)$ to be the collection with $A_L=0$.
This brings a computational convenience that $D_L=L$ in next subsections.

Here we also show some behaviors of $A$ and $\phi$ near the origin $\{v=0\}$.
Noticing that the $\Lb$ direction of $(0,0,\omega)$ is exactly the $L$ direction of $(0,0,-\omega)$,
for any $k\in\rN_0$ and $\bOmega$, we know
\begin{Eq}\label{Eq:co_1}
\kl[\Lb^n\bOmega(A_\Lb,A_L,\sA,\phi)\kr](0,0,\omega)=\kl[L^n\bOmega(A_\Lb,A_L,\sA,\phi)\kr](0,0,-\omega)=0
\end{Eq}
since we have assumed that $(\alpha,\phi)$ vanishes near $\{v=0\}$.

\subsubsection{The iteration quantities}
For $N=0,1,2$, we define
\begin{Eq}\label{Eq:itf_1}
\varE[A_\Lb;N]:=&\hspace{-45pt}\sum\Low{\chi_\Lb(\bY)\leq N\\\chi_L(\bY)\leq 3-\chi_\Lb(\bY)\\\chi_\Omega(\bY)\leq 7-2\chi_\Lb(\bY)-2\Ma{\chi_L(\bY)}}\hspace{-45pt}\int_0^1\int_{\rS^2}v^{-4+2\chi_{\Lb,L}(\bY)}|\bY A_\Lb|^2\d\omega\d v,\\
\varE[A_L;N]:=&\hspace{-15pt}\sum\Low{\chi_\Lb(\bY)\leq N\\\chi_L(\bY)\leq 3-\chi_\Lb(\bY)\\\chi_\Omega(\bY)\leq 7-2\chi_{\Lb,L}(\bY)}\hspace{-15pt}\int_0^1\int_{\rS^2}v^{-4+2\chi_{\Lb,L}(\bY)}\kl|\bY  A_L\kr|^2\d\omega\d v,\\
\varE[\sA;N]:=&\hspace{-45pt}\sum\Low{\chi_\Lb(\bY)\leq N\\\chi_L(\bY)\leq 3-\chi_\Lb(\bY)\\\chi_\Omega(\bY)\leq 8-2\chi_\Lb(\bY)-2\Ma{\chi_L(\bY)}}\hspace{-45pt}\int_0^1\int_{\rS^2}v^{-4+2\chi_{\Lb,L}(\bY)}\kl|\bY  \sA\kr|^2\d\omega\d v,\\
\varE[\phi;N]:=&\hspace{-45pt}\sum\Low{\chi_\Lb(\bY)\leq N\\\chi_L(\bY)\leq 3-\chi_\Lb(\bY)\\\chi_\Omega(\bY)\leq 7-2\chi_\Lb(\bY)-2\Ma{\chi_L(\bY)}}\hspace{-45pt}\int_0^1\int_{\rS^2}v^{-4+2\chi_{\Lb,L}(\bY)}|\bY \phi|^2\d\omega\d v,\\
\varE[J_L;N]:=&\hspace{-15pt}\sum\Low{\chi_\Lb(\bY)\leq N\\\chi_L(\bY)\leq 2-\chi_\Lb(\bY)\\\chi_\Omega(\bY)\leq 5-2\chi_{\Lb,L}(\bY)}\hspace{-15pt}\int_0^1\int_{\rS^2}v^{2\chi_{\Lb,L}(\bY)}|\bY J_L|^2\d\omega\d v,\\
\varE[\sJ;N]:=&\hspace{-15pt}\sum\Low{\chi_\Lb(\bY)\leq N\\\chi_L(\bY)\leq 2-\chi_\Lb(\bY)\\\chi_\Omega(\bY)\leq 4-2\chi_{\Lb,L}(\bY)}\hspace{-15pt}\int_0^1\int_{\rS^2}v^{2\chi_{\Lb,L}(\bY)}|\bY \sJ|^2\d\omega\d v,\\
\varF[A_\Lb;N]:=&\sup_{v\in(0,1]}\hspace{-45pt}\sum\Low{\chi_\Lb(\bY)\leq N\\\chi_L(\bY)\leq 2-\chi_\Lb(\bY)\\\chi_\Omega(\bY)\leq 7-2\chi_\Lb(\bY)-2\Ma{\chi_L(\bY)}}\hspace{-45pt}v^{-3+2\chi_{\Lb,L}(\bY)}\int_{\rS^2}|\bY A_\Lb|^2\d\omega,\\
\varF[A_L;N]:=&\sup_{v\in(0,1]}\hspace{-15pt}\sum\Low{\chi_\Lb(\bY)\leq N\\\chi_L(\bY)\leq 2-\chi_\Lb(\bY)\\\chi_\Omega(\bY)\leq 7-2\chi_{\Lb,L}(\bY)}\hspace{-15pt}v^{-3+2\chi_{\Lb,L}(\bY)}\int_{\rS^2}\kl|\bY  A_L\kr|^2\d\omega,\\
\varF[\sA;N]:=&\sup_{v\in(0,1]}\hspace{-45pt}\sum\Low{\chi_\Lb(\bY)\leq N\\\chi_L(\bY)\leq 2-\chi_\Lb(\bY)\\\chi_\Omega(\bY)\leq 8-2\chi_\Lb(\bY)-2\Ma{\chi_L(\bY)}}\hspace{-45pt}v^{-3+2\chi_{\Lb,L}(\bY)}\int_{\rS^2}\kl|\bY  \sA\kr|^2\d\omega,\\
\varF[\phi;N]:=&\sup_{v\in(0,1]}\hspace{-45pt}\sum\Low{\chi_\Lb(\bY)\leq N\\\chi_L(\bY)\leq 2-\chi_\Lb(\bY)\\\chi_\Omega(\bY)\leq 7-2\chi_\Lb(\bY)-2\Ma{\chi_L(\bY)}}\hspace{-45pt}v^{-3+2\chi_{\Lb,L}(\bY)}\int_{\rS^2}|\bY \phi|^2\d\omega.
\end{Eq}

\subsubsection{The estimate of $\varE[A_\Lb;N]$}
Using the calculation trick
\begin{Eq*}
\sum_{n\leq N}s^n|\partial_s^n f|\approx \sum_{n\leq N}s^{n-m}|\partial_s^n (s^mf)|,
\end{Eq*}
 we can dominate $\varE[A_\Lb;N]$ as
\begin{Eq*}
\varE[A_\Lb;N]\lesssim&\Bigg(\hspace{-20pt}\sum\Low{\chi_\Lb(\bY)\leq N\\1\leq\chi_L(\bY)\leq 3-\chi_\Lb(\bY)\\\chi_\Omega(\bY)\leq 7-2\chi_{\Lb,L}(\bY)}+\hspace{-5pt}\sum\Low{\chi_\Lb(\bY)\leq N\\\chi_L(\bY)=0\\\chi_\Omega(\bY)\leq 5-2\chi_\Lb(\bY)}\hspace{-15pt}\Bigg)\int_0^1\int_{\rS^2}v^{-6+2\chi_{\Lb,L}(\bY)}|\bY (rA_\Lb)|^2\d\omega\d v.
\end{Eq*}
Now, we split an $L$ from $\bY$ in the integral of the first summation,
and use \eqref{Eq:key_Eob2} with \eqref{Eq:co_1} for the integral in second summation.
Then, similar to above, we find
\begin{Eq*}
\varE[A_\Lb;N]\lesssim&\hspace{-15pt}\sum\Low{\chi_\Lb(\bY)\leq N\\\chi_L(\bY)\leq 2-\chi_\Lb(\bY)\\\chi_\Omega(\bY)\leq 5-2\chi_{\Lb,L}(\bY)}\hspace{-15pt}\int_0^1\int_{\rS^2}v^{-4+2\chi_{\Lb,L}(\bY)}|\bY L(rA_\Lb)|^2\d\omega\d v\\
\lesssim&\sum\Low{\cdots}\int_0^1\int_{\rS^2}v^{-6+2\chi_{\Lb,L}(\bY)}|\bY (rL(rA_\Lb))|^2\d\omega\d v.
\end{Eq*}
Here and hereafter, we use the convention $\sum\limits_{\cdots}$ which means the summation with index same as the previous one.
Similar to the above process and using \eqref{Eq:c_eq_ALb}, we get
\begin{Eq*}
\varE[A_\Lb;N]
\lesssim&\hspace{-20pt}\sum\Low{\chi_\Lb(\bY)\leq N\\\chi_L(\bY)\leq 1-\Mi{\chi_\Lb(\bY)}\\\chi_\Omega(\bY)\leq 5-2\chi_{\Lb,L}(\bY)}\hspace{-20pt}\int_0^1\int_{\rS^2}v^{-4+2\chi_{\Lb,L}(\bY)}|\bY L(rL(rA_\Lb))|^2\d\omega\d v\\
\lesssim&\sum\Low{\cdots}\int_0^1\int_{\rS^2}v^{-4+2\chi_{\Lb,L}(\bY)}\kl|\bY \kl((rL)^{\leq 1}A_L,\Omega^2A_L,(rL)^{\leq 1}\Omega^{\leq 1}\sA,r^2J_L\kr)\kr|^2\d\omega\d v\\
\lesssim& \varE[(A_L,\sA,J_L);N].\\
\end{Eq*}

\subsubsection{The estimate of $\varE[A_L;N]$}
At first, it is easy to see that $\varE[A_L;0]=0$ since $A_L=0$.

Then, similar to above processes and using \eqref{Eq:Ro_LbAL_ug}, for $N=1,2$, we have 
\begin{Eq*}
\varE[A_L;N]\lesssim&\hspace{-10pt}\sum\Low{\chi_\Lb(\bY)\leq N-1\\\chi_L(\bY)\leq 2-\chi_\Lb(\bY)\\\chi_\Omega(\bY)\leq 5-2\chi_{\Lb,L}(\bY)}\hspace{-10pt}\int_0^1\int_{\rS^2}v^{-4+2\chi_{\Lb,L}(\bY)}\kl|\bY  \Lb(r A_L)\kr|^2\d\omega\d v+\varE[A_L;0]\\
\lesssim&\sum_{\cdots}\int_0^1\int_{\rS^2}v^{-4+2\chi_{\Lb,L}(\bY)}\kl|\bY \kl((rL)^{\leq 1}A_\Lb,A_L,\Omega^{\leq 1}\sA\kr)\kr|^2\d\omega\d v\\
\lesssim&\varE[(A_\Lb,A_L,\sA);N-1].
\end{Eq*}

\subsubsection{The estimate of $\varE[\sA;N]$}
Similar to the above processes, with the additional use of \eqref{Eq:S_sA_b_alpha} with $A_L=0$ and \eqref{Eq:Eo_talpha_o_tH}, we get
\begin{Eq*}
\varE[\sA;0]\lesssim&\hspace{-15pt}\sum\Low{\chi_\Lb(\bY)=0\\\chi_L(\bY)\leq 2\\\chi_\Omega(\bY)\leq 6-2\chi_L(\bY)}\hspace{-15pt}\int_0^1\int_{\rS^2}v^{-4+2\chi_{L}(\bY)}\kl|\bY L(r\sA)\kr|^2\d\omega\d v\\
\lesssim&\hspace{-15pt}\sum\Low{\chi_\Lb(\bY)=0\\\chi_L(\bY)\leq 2\\\chi_\Omega(\bY)\leq 6-2\chi_L(\bY)}\hspace{-15pt}\int_0^1\int_{\rS^2}v^{-2+2\chi_{L}(\bY)}\kl|\bY \alpha\kr|^2\d\omega\d v<\infty.
\end{Eq*}

As for $N=1,2$, similar to above and by use of \eqref{Eq:c_eq_sA}, we have
\begin{Eq*}
\varE[\sA;N]\lesssim&\hspace{-45pt}\sum\Low{\chi_\Lb(\bY)\leq N\\\chi_L(\bY)\leq 3-\chi_\Lb(\bY)\\\chi_\Omega(\bY)\leq 8-2\chi_\Lb(\bY)-2\Ma{\chi_L(\bY)}}\hspace{-45pt}\int_0^1\int_{\rS^2}v^{-6+2\chi_{\Lb,L}(\bY)}\kl|\bY  (r\sA)\kr|^2\d\omega\d v\\
\lesssim&\hspace{-15pt}\sum\Low{\chi_\Lb(\bY)\leq N-1\\\chi_L(\bY)\leq 1-\chi_\Lb(\bY)\\\chi_\Omega(\bY)\leq 4-2\chi_{\Lb,L}(\bY)}\hspace{-15pt}\int_0^1\int_{\rS^2}v^{-4+2\chi_{\Lb,L}(\bY)}\kl|\bY  (rL\Lb (r\sA))\kr|^2\d\omega\d v+\varE[\sA;0]\\
\lesssim&\sum\Low{\cdots}\int_0^1\int_{\rS^2}v^{-4+2\chi_{\Lb,L}(\bY)}\kl|\bY  \kl(\Omega^{\leq 1}(A_\Lb,A_L),\Omega^{\leq 2}\sA,r^2\sJ\kr) \kr|^2\d\omega\d v+\varE[\sA;0]\\
\lesssim&\varE[(A_\Lb,A_L,\sA,\sJ);N-1].
\end{Eq*}

\subsubsection{The estimate of $\varE[\phi;0]$}
Using \eqref{Eq:key_Eob3_2} and \eqref{Eq:Eo_tsAphi_o_tH}, we know that 
\begin{Eq*}
\sup_{v\in(0,1]}\hspace{-20pt}\sum\Low{n\leq 2\\|\beta|\leq 5-2\Ma{n}}\hspace{-20pt}v^{-1}\int_{\rS^2}|v^{-1+n}D_L^n D_\Omega^{\beta}\phi|^2\d\omega
\lesssim&\sum_{\cdots}\int_0^1\int_{\rS^2}v^{-2}\kl|(vL)^{\leq 1}\kl(v^{-1+n}D_L^n D_\Omega^{\beta} \phi\kr)\kr|^2\d\omega\d v\\
\lesssim&\underbrace{\hspace{-20pt}\sum\Low{n\leq 3\\|\beta|\leq 7-2\Ma{n}}\hspace{-20pt}\int_0^1\int_{\rS^2}v^{-4+2n}|D_L^n D_\Omega^{\beta} \phi|^2\d\omega\d v}_{P_1}<\infty.
\end{Eq*}

In order to use this result,
we introduce an obvious observation.
Because $D_L=L$, we have
\begin{Eq*}
\sum\Low{\chi_\Lb(\bY)=0 \\\chi_L(\bY)\leq 3\\\chi_\Omega(\bY)\leq 7-2\Ma{\chi_L(\bY)}}\hspace{-30pt}v^{-2+\chi_L(\bY)}\kl|{\bY\phi}\kr|
\lesssim &\hspace{-20pt}\sum\Low{n\leq 3\\|\beta|\leq 7-2\Ma{n}}\hspace{-20pt}v^{-3+n}\kl|{D_L^n(D_\Omega-\I r\sA)^{7-2\Ma{n}}(r\phi)}\kr|\\
\lesssim&\hspace{-20pt}\sum\Low{n\leq 3\\|\beta|\leq 7-2\Ma{n}}\hspace{-20pt}\kl|v^{-2+n}\kl(D_L^nD_\Omega^\beta\phi,L^n\Omega^\beta\sA\kr)\kr|\cdot \hspace{-20pt}\sum\Low{n\leq2\\|\beta|\leq 5-2\Ma{n}}\hspace{-20pt}\kl<v^{1+n}\kl(D_L^nD_\Omega^\beta\phi,L^n\Omega^\beta\sA\kr)\kr>^5.
\end{Eq*}

Substituting these into $\varE[\phi;0]$, using the \emph{H\"older} inequality and the \emph{Sobolev} inequality on sphere, and noticing that $v\in(0,1]$, we easily find
\begin{Eq*}
\varE[\phi;0]\lesssim& \kl(\varE[\sA;0]+P_1\kr)\cdot\kl<\varF[\sA;0]+P_1\kr>^5.
\end{Eq*}

\subsubsection{The estimate of $\varE[\phi;N]$ $(N=1,2)$}
Similar to above, for part of $\varE[\phi;N]$ with $N=1,2$, we know
\begin{Eq*}
&\hspace{-45pt}\sum\Low{1\leq\chi_\Lb(\bY)\leq N\\\chi_L(\bY)\leq 3-\chi_\Lb(\bY)\\\chi_\Omega(\bY)\leq 7-2\chi_\Lb(\bY)-2\Ma{\chi_L(\bY)}}\hspace{-40pt}\int_0^1\int_{\rS^2}v^{-6+2\chi_{\Lb,L}(\bY)}\kl|\bY  (r\phi)\kr|^2\d\omega\d v\\
\lesssim&\underbrace{\hspace{-15pt}\sum\Low{\chi_\Lb(\bY)\leq N-1\\\chi_L(\bY)\leq 1-\chi_\Lb(\bY)\\\chi_\Omega(\bY)\leq 3-2\chi_{\Lb,L}(\bY)}\hspace{-15pt}\int_0^1\int_{\rS^2}v^{-2+2\chi_{\Lb,L}(\bY)}\kl|L\bY  \Lb(r\phi)\kr|^2\d\omega\d v}_{P_2[N-1]}.
\end{Eq*}

To estimate $P_2[N]$, we take the derivative on both sides of \eqref{Eq:c_eq_phi}. Then, for $N=0,1$, we have
\begin{Eq*}
&\hspace{-15pt}\sum\Low{\chi_\Lb(\bY)\leq N\\\chi_L(\bY)\leq 1-\chi_\Lb(\bY)\\\chi_\Omega(\bY)\leq 3-2\chi_{\Lb,L}(\bY)}\hspace{-15pt}v^{-1+\chi_{\Lb,L}(\bY)}\kl|L\bY\Lb(r\phi)\kr|\\
\lesssim& \hspace{-45pt}\sum\Low{\chi_\Lb(\bY)\leq N\\\chi_L(\bY)\leq 3-\chi_\Lb(\bY)\\\chi_\Omega(\bY)\leq 7-2\chi_\Lb(\bY)-2\Ma{\chi_L(\bY)}}\hspace{-45pt}\kl|v^{-2+\chi_{\Lb,L}(\bY)}\bY (A_\Lb,A_L,\sA,\phi)\kr|
\cdot\hspace{-45pt}\sum\Low{\chi_\Lb(\bY)\leq N\\\chi_L(\bY)\leq 2-\chi_\Lb(\bY)\\\chi_\Omega(\bY)\leq 5-2\chi_\Lb(\bY)-2\Ma{\chi_L(\bY)}}\hspace{-45pt}\kl<v^{1+\chi_{\Lb,L}(\bY)}\bY(A_\Lb,A_L,\sA,\phi)\kr>^2.
\end{Eq*}
Substituting it into $P_2[N]$, 
similar to the process of estimating $\varE[\phi;0]$,
for $N=0,1$, we easily find
\begin{Eq*}
P_2[N]\lesssim \varE[(A_\Lb,A_L,\sA,\phi);N]\cdot \kl<\varF[(A_\Lb,A_L,\sA,\phi);N]\kr>^2.
\end{Eq*}
Back to $\varE[\phi;N]$ with $N=1,2$, we finally find
\begin{Eq*}
\varE[\phi;N]\lesssim& \varE[(A_\Lb,A_L,\sA,\phi);N-1]\cdot \kl<\varF[(A_\Lb,A_L,\sA,\phi);N-1]\kr>^2+\varE[\phi;0].
\end{Eq*}

\subsubsection{The estimate of $\varE[(J_L,\sJ);N]$}
For $J_L$ part, similar to above, we calculate that
\begin{Eq*}
&\hspace{-15pt}\sum\Low{\chi_\Lb(\bY)\leq N\\\chi_L(\bY)\leq 2-\chi_\Lb(\bY)\\\chi_\Omega(\bY)\leq 5-2\chi_{\Lb,L}(\bY)}\hspace{-15pt}v^{\chi_{\Lb,L}(\bY)}|\bY J_L|\\
\lesssim &\hspace{-45pt}\sum\Low{\chi_\Lb(\bY)\leq N\\\chi_L(\bY)\leq  3-\chi_{\Lb}(\bY)\\\chi_\Omega(\bY)\leq 7-2\chi_\Lb(\bY)-2\Ma{\chi_L(\bY)}}\hspace{-45pt}\kl|v^{-2+\chi_{\Lb,L}(\bY)}\bY(A_L,\phi)\kr|
\cdot\hspace{-45pt}\sum\Low{\chi_\Lb(\bY)\leq N\\\chi_L(\bY)\leq  2-\chi_{\Lb}(\bY)\\\chi_\Omega(\bY)\leq 5-2\chi_\Lb(\bY)-2\Ma{\chi_L(\bY)}}\hspace{-45pt}\kl<v^{1+\chi_{\Lb,L}(\bY)}\bY(A_L,\phi)\kr>^2.
\end{Eq*}
Then, similar to above, we find
\begin{Eq*}
\varE[J_L;N]\lesssim \varE[(A_L,\phi);N]\kl<\varF[(A_L,\phi);N]\kr>^2,
\end{Eq*}
and similarly,
\begin{Eq*}
\varE[\sJ;N]\lesssim \varE[(\sA,\phi);N]\kl<\varF[(\sA,\phi);N]\kr>^2.
\end{Eq*}

\subsubsection{The estimate of $\varF[(A_\Lb,A_L,\sA,\phi);N]$}
For $A_\Lb$ part, using \eqref{Eq:key_Eob3_2}, we find
\begin{Eq*}
\varF[A_\Lb;N]=&\sup_{v\in(0,1]}\hspace{-45pt}\sum\Low{\chi_\Lb(\bY)\leq N\\\chi_L(\bY)\leq 2-\chi_\Lb(\bY)\\\chi_\Omega(\bY)\leq 7-2\chi_\Lb(\bY)-2\Ma{\chi_L(\bY)}}\hspace{-45pt}v^{-1}\int_{\rS^2}|v^{-1+\chi_{\Lb,L}(\bY)}\bY A_\Lb|^2\d\omega\\
\lesssim&\sum_{\cdots}\int_0^1\int_{\rS^2}v^{-2}\kl|(vL)^{\leq 1}\kl(v^{-1+\chi_{\Lb,L}(\bY)}\bY A_\Lb\kr)\kr|^2\d\omega\d v \lesssim \varE[A_\Lb;N].
\end{Eq*}
Similarly, we also have
\begin{Eq*}
\varF[A_L;N]\lesssim \varE[A_L;N],\qquad \varF[\sA;N]\lesssim \varE[\sA;N],\qquad \varF[\phi;N]\lesssim \varE[\phi;N].
\end{Eq*}

\subsubsection{The end of the proof}
Mixing these estimates, we get that all quantities defined in \eqref{Eq:itf_1} are finite.
Meanwhile, a direct calculation shows that for any $\beta$, there are
\begin{Eq*}
\partial^\beta f=\sum_{\chi(\bY)\leq|\beta|}C_{\beta,\bY}(\omega)r^{-|\beta|+\chi_{\Lb,L}(\bY)}\bY f.
\end{Eq*}
Then, we finally find
\begin{Eq*}
&\sum\Low{|\alpha|\leq 1\\|\beta|\leq 2}\int_{\rH_0^{0,1}}|(\spartial,L)^\alpha\partial^\beta (A,\phi)|^2\\
\lesssim&\sum\Low{|\alpha|\leq 1\\\chi(\bY)\leq|\beta|\leq 2}\int_0^1\int_{\rS^2}v^2\kl|(\spartial,L)^\alpha \kl(C_{\beta,\bY}(\omega) r^{-|\beta|+\chi_{\Lb,L}(\bY)}\bY(A,\phi)\kr)\kr|^2\d \omega\d v\\
\lesssim&\hspace{-10pt}\sum\Low{\chi_{\Lb}(\bY)\leq 2\\\chi_{L,\Omega}(\bY)\leq 3-\chi_{\Lb}(\bY)}\hspace{-10pt}\int_0^1\int_{\rS^2}v^{-4+2\chi_{\Lb,L}(\bY)}|\bY (A,\phi)|^2\d\omega\d v\\
\lesssim&\sum\Low{\cdots}\int_0^1\int_{\rS^2}v^{-4+2\chi_{\Lb,L}(\bY)}|\bY (A_\Lb,A_L,\sA,\phi)|^2\d\omega\d v
\lesssim \varE[(A_\Lb,A_L,\sA,\phi);2]<\infty.
\end{Eq*}
This finishes the proof. 

\subsection{Proof of \Cm{Cm:Eo_FphiZ_oeni}}\label{Sn:PoCEo_FphiZ_oeni}
As that in \Sn{Sn:Taierp1}, we allow the constant in $\lesssim$ to depend on the initial data, 
and use the convention $\INF{f}$ which means $\lim_{v\rightarrow\infty}f$.

\subsubsection{The iteration form}
For $N=0,1,2$, we define
\begin{Eq}\label{Eq:itf_21}
\varE[\rho;N]:=&\hspace{-45pt}\sum\Low{n\leq N\\\chi_\Lb(\bY)\leq 2-n,~\chi_L(\bY)=0\\\chi_\Omega(\bY)\leq 7-2n-2\Ma{\chi_\Lb(\bY)}}\hspace{-45pt}\int_{-\infty}^{-1}\int_{\rS^2}|u|^{2\chi_\Lb(\bY)-2n-\varepsilon_1}\kl|\INFt{(r^2L)^n\bY (r^2\rrho)}\kr|^2\d\omega\d u,\\
\varE[\sigma;N]:=&\hspace{-45pt}\sum\Low{n\leq N\\\chi_\Lb(\bY)\leq 2-n,~\chi_L(\bY)=0\\\chi_\Omega(\bY)\leq 7-2n-2\Ma{\chi_\Lb(\bY)}}\hspace{-45pt}\int_{-\infty}^{-1}\int_{\rS^2}|u|^{2\chi_\Lb(\bY)-2n-\varepsilon_1}\kl|\INFt{(r^2L)^n\bY (r^2\sigma)}\kr|^2\d\omega\d u,\\
\varE[\alphab;N]:=&\hspace{-30pt}\sum\Low{n\leq N\\\chi_\Lb(\bY)\leq 2-n,~\chi_L(\bY)=0\\\chi_\Omega(\bY)\leq 6-2n-2\chi_\Lb(\bY)}\hspace{-30pt}\int_{-\infty}^{-1}\int_{\rS^2}|u|^{2+2\chi_\Lb(\bY)-2n-\varepsilon_1}\kl|\INFt{(r^2L)^n\bY (r\alphab)}\kr|^2\d\omega\d u,\\
\varE[\phi;N]:=&\hspace{-45pt}\sum\Low{n\leq N\\\chi_\Lb(\bY)\leq 3-n,~\chi_L(\bY)=0\\\chi_\Omega(\bY)\leq 7-2n-2\Ma{\chi_\Lb(\bY)}}\hspace{-45pt}\int_{-\infty}^{-1}\int_{\rS^2}|u|^{2\chi_\Lb(\bY)-2n-\varepsilon_1}\kl|\INFt{(r^2D_L)^nD_\bY(r\phi)}\kr|^2\d\omega\d u,\\
\varE[J_\Lb;N]:=&\hspace{-30pt}\sum\Low{n\leq N\\\chi_\Lb(\bY)\leq 2-n,~\chi_L(\bY)=0\\\chi_\Omega(\bY)\leq 5-2n-2\chi_\Lb(\bY)}\hspace{-30pt}\int_{-\infty}^{-1}\int_{\rS^2}|u|^{3+2\chi_\Lb(\bY)-2n-2\varepsilon_1}\kl|\INFt{(r^2L)^n\bY (r^2J_\Lb)}\kr|^2\d\omega\d u,\\
\varE[\sJ;N]:=&\hspace{-45pt}\sum\Low{n\leq N\\\chi_\Lb(\bY)\leq 2-n,~\chi_L(\bY)=0\\\chi_\Omega(\bY)\leq 6-2n-2\Ma{\chi_\Lb(\bY)}}\hspace{-45pt}\int_{-\infty}^{-1}\int_{\rS^2}|u|^{1+2\chi_\Lb(\bY)-2n-2\varepsilon_1}\kl|\INFt{(r^2L)^n\bY (r^3\sJ)}\kr|^2\d\omega\d u,\\
\varF[\phi;N]:=&\sup_{u\in(-\infty,-1]}\hspace{-45pt}\sum\Low{n\leq N\\\chi_\Lb(\bY)\leq 2-n,~\chi_L(\bY)=0\\\chi_\Omega(\bY)\leq 7-2n-2\Ma{\chi_\Lb(\bY)}}\hspace{-45pt}|u|^{1+2\chi_\Lb(\bY)-2n-\varepsilon_1}\int_{\rS^2}\kl|\INFt{(r^2D_L)^nD_\bY(r\phi)}\kr|^2\d\omega.
\end{Eq}
For $N=0,1$, we define
\begin{Eq}\label{Eq:itf_22}
\varF[\rho;N]:=&\sup_{u\in(-\infty,-1]}\hspace{-45pt}\sum\Low{n\leq N\\\chi_\Lb(\bY)\leq 1-n,~\chi_L(\bY)=0\\\chi_\Omega(\bY)\leq 7-2n-2\Ma{\chi_\Lb(\bY)}}\hspace{-45pt}|u|^{1+2\chi_\Lb(\bY)-2n-\varepsilon_1}\int_{\rS^2}\kl|\INFt{(r^2L)^n\bY (r^2\rrho)}\kr|^2\d\omega,\\
\varF[\sigma;N]:=&\sup_{u\in(-\infty,-1]}\hspace{-45pt}\sum\Low{n\leq N\\\chi_\Lb(\bY)\leq 1-n,~\chi_L(\bY)=0\\\chi_\Omega(\bY)\leq 7-2n-2\Ma{\chi_\Lb(\bY)}}\hspace{-45pt}|u|^{1+2\chi_\Lb(\bY)-2n-\varepsilon_1}\int_{\rS^2}\kl|\INFt{(r^2L)^n\bY (r^2\sigma)}\kr|^2\d\omega,\\
\varF[\alphab;N]:=&\sup_{u\in(-\infty,-1]}\hspace{-30pt}\sum\Low{n\leq N\\\chi_\Lb(\bY)\leq 1-n,~\chi_L(\bY)=0\\\chi_\Omega(\bY)\leq 6-2n-\chi_\Lb(\bY)}\hspace{-30pt}|u|^{3+2\chi_\Lb(\bY)-2n-\varepsilon_1}\int_{\rS^2}\kl|\INFt{(r^2L)^n\bY (r\alphab)}\kr|^2\d\omega,\\
\varF[\alpha;N]:=&\sup_{u\in(-\infty,-1]}\hspace{-45pt}\sum\Low{n\leq N\\\chi_\Lb(\bY)\leq 1-n,~\chi_L(\bY)=0\\\chi_\Omega(\bY)\leq 6-2n-2\Ma{\chi_\Lb(\bY)}}\hspace{-45pt}|u|^{-1+2\chi_\Lb(\bY)-2n-\varepsilon_1}\int_{\rS^2}\kl|\INFt{(r^2L)^n\bY (r^3\alpha)}\kr|^2\d\omega,\\
\varF[J_L;N]:=&\sup_{u\in(-\infty,-1]}\hspace{-45pt}\sum\Low{n\leq N\\\chi_\Lb(\bY)\leq 1-n,~\chi_L(\bY)=0\\\chi_\Omega(\bY)\leq 5-2n-2\Ma{\chi_\Lb(\bY)}}\hspace{-45pt}|u|^{2\chi_\Lb(\bY)-2n-2\varepsilon_1}\int_{\rS^2}\kl|\INFt{(r^2L)^n\bY (r^4J_L)}\kr|^2\d\omega,\\
\varF[\sJ;N]:=&\sup_{u\in(-\infty,-1]}\hspace{-45pt}\sum\Low{n\leq N\\\chi_\Lb(\bY)\leq 1-n,~\chi_L(\bY)=0\\\chi_\Omega(\bY)\leq 6-2n-2\Ma{\chi_\Lb(\bY)}}\hspace{-45pt}|u|^{2+2\chi_\Lb(\bY)-2n-2\varepsilon_1}\int_{\rS^2}\kl|\INFt{(r^2L)^n\bY (r^3\sJ)}\kr|^2\d\omega.
\end{Eq}

\subsubsection{The estimate of $\varE[\rho;0]$}
To begin with, by \eqref{Eq:MKGe-ud_FZ}, we have
\begin{Eq*}
\INFt{\Lb(r^2\rho)}=\INFt{[r\sdiv](r\alphab)-\Im\kl((r\phi)\overline{D_\Lb(r\phi)}\kr)}=\rsdiv \Alphab-\Im\kl(\Phi\overline{D_u\Phi}\kr).
\end{Eq*}
Using \eqref{Eq:crosd} and the fact that $\rho_r=\rho-\q r^{-2}$, we find 
\begin{Eq*}
\lim_{u\rightarrow-\infty}\INFt{r^2\rrho}=\lim_{u\rightarrow-\infty}\INFt{r^2\rho}-\q=\lim_{u\rightarrow\infty}\INFt{r^2\rho}.
\end{Eq*}
Combing this with \eqref{Eq:Eo_rho_oini}, we get
\begin{Eq*}
\lim_{u\rightarrow-\infty}\int_{\rS^2}\kl|\INFt{r^2\Omega^{\leq 5}\rrho}\kr|^2\d\omega=\lim_{u\rightarrow\infty}\int_{\rS^2}\kl|\INFt{r^2\Omega^{\leq 5}\rho}\kr|^2\d\omega=0.
\end{Eq*}

Now, similar to the estimate of $\varE[A_\Lb;N]$ in \Sn{Sn:PoCEo_tAphi_o_tH}, using \eqref{Eq:key_Eob2} and \eqref{Eq:MKGe-ud_FZ}, we know
\begin{Eq*}
\varE[\rho;0]
\lesssim& \Bigg(\hspace{-10pt}\sum\Low{1\leq\chi_\Lb(\bY)\leq 2\\\chi_L(\bY)=0\\\chi_\Omega(\bY)\leq 7-2\chi_\Lb(\bY)}\hspace{-10pt}+\sum\Low{\chi_{\Lb,L}(\bY)=0\\\chi_\Omega(\bY)\leq 5}\Bigg)\int_{-\infty}^{-1}\int_{\rS^2}|u|^{2\chi_\Lb(\bY)-\varepsilon_1}\kl|\INFt{\bY (r^2\rrho)}\kr|^2\d\omega\d u\\
\lesssim& \sum\Low{\chi_\Lb(\bY)\leq 1\\\chi_L(\bY)=0\\\chi_\Omega(\bY)\leq 5-2\chi_\Lb(\bY)}\int_{-\infty}^{-1}\int_{\rS^2}|u|^{2+2\chi_\Lb(\bY)-\varepsilon_1}\kl|\INFt{\bY \Lb(r^2\rrho)}\kr|^2\d\omega\d u\\
\lesssim& \sum_{\cdots} \int_{-\infty}^{-1}\int_{\rS^2}|u|^{2+2\chi_\Lb(\bY)-\varepsilon_1}\kl|\INFt{\bY\kl(\Omega^{\leq 1}(r\alphab),r^2J_\Lb\kr)}\kr|^2\d\omega\d u\\
\lesssim&\varE[(\alphab,J_\Lb);0].
\end{Eq*}

\subsubsection{The estimate of $\varE[\rho;N]$ $(N=1,2)$}
Similar to above, by \eqref{Eq:Eo_rho_oini}, we also know
\begin{Eq*}
\int_{\rS^2}\kl|\INFt{(r^2L)\Omega^{\leq 3}(r^2\rrho)}\kr|_{u=-1}^2\d\omega\lesssim 1.
\end{Eq*}
On the other hand, by \eqref{Eq:MKGe-ud_FZ}, we find
\begin{Eq*}
\int_{\rS^2}\kl|\INFt{(r^2L)^2\Omega^{\leq 1}(r^2\rrho)}\kr|_{u=-1}^2\d\omega
\lesssim&\int_{\rS^2}\kl|\INFt{(r^2L)\Omega^{\leq 1}(\Omega^{\leq 1}\kl(r^3\alpha),r^4J_L\kr)}\kr|_{u=-1}^2\d\omega
\lesssim \varF[(\alpha,J_L);1].
\end{Eq*}
Now, similar to above, using \eqref{Eq:key_Eob1_1} and \eqref{Eq:MKGe-ud_FZ}, for $N=1,2$, we know
\begin{Eq*}
\varE[\rho;N]
\lesssim&\hspace{-40pt}\sum\Low{n\leq N\\\chi_\Lb(\bY)\leq 1-\Mi{n},~\chi_L(\bY)=0\\\chi_\Omega(\bY)\leq 5-2n-2\chi_\Lb(\bY)}\hspace{-40pt}\int_{-\infty}^{-1}\int_{\rS^2}|u|^{2+2\chi_\Lb(\bY)-2n-\varepsilon_1}\kl|\INFt{(r^2L)^n\bY \Lb(r^2\rrho)}\kr|^2\d\omega\d u\\ 
&\quad+\varE[\rho;0]+1+\varF[(\alpha,J_L);N-1]\\
\lesssim&\sum_{\cdots}\int_{-\infty}^{-1}\int_{\rS^2}|u|^{2+2\chi_\Lb(\bY)-2n-\varepsilon_1}\kl|\INFt{(r^2L)^n\bY\kl(\Omega^{\leq 1}(r\alphab),r^2J_\Lb\kr)}\kr|^2\d\omega\d u\\ 
&\quad+\varE[\rho;0]+1+\varF[(\alpha,J_L);N-1]\\
\lesssim&\varE[(\alphab,J_\Lb);N]+\varE[\rho;0]+1+\varF[(\alpha,J_L);N-1].
\end{Eq*}

\subsubsection{The estimate of $\varE[\sigma;N]$} 
First for $N=0$, similar to above, by \eqref{Eq:MKGe-ud_FZ}, \eqref{Eq:crosd} and \eqref{Eq:Eo_sigma_oini}, we know
\begin{Eq*}
\lim_{u\rightarrow-\infty}\int_{\rS^2}\kl|\INFt{r^2\Omega^{\leq 5}\sigma}\kr|^2\d\omega=\lim_{u\rightarrow\infty}\int_{\rS^2}\kl|\INFt{r^2\Omega^{\leq 5}\sigma}\kr|^2\d\omega=0.
\end{Eq*}
Then, similar to the process of estimating $\varE[\rho;0]$, using \eqref{Eq:key_Eob2} and \eqref{Eq:MKGe-ud_FZ}, we also know
\begin{Eq*}
\varE[\sigma;0]\lesssim&\varE[\alphab;0].
\end{Eq*}

On the other hand, by \eqref{Eq:Eo_sigma_oini}, we know
\begin{Eq*}
\sum_{n\leq 2}\int_{\rS^2}\kl|\INFt{(r^2L)^{n}\Omega^{\leq 5-2n}(r^2\sigma)}\kr|_{u=-1}^2\d\omega\lesssim 1.
\end{Eq*}
Then, similar to above, for $N=1,2$, we know
\begin{Eq*}
\varE[\sigma;N]\lesssim&\varE[\alphab;N]+1.
\end{Eq*}

\subsubsection{The estimate of $\varE[\alphab;N]$}
Noticing $\mathcal{E}_{\varepsilon_1}<\infty$, we know
\begin{Eq*}
\varE[\alphab;0]\lesssim 1.
\end{Eq*}
As for $N=1,2$, similar to above and using \eqref{Eq:MKGe-ud_FZ}, we know
\begin{Eq*}
\varE[\alphab;N]\lesssim&\hspace{-20pt}\sum\Low{n\leq N-1\\\chi_\Lb(\bY)\leq 1-n,~\chi_L(\bY)=0\\\chi_\Omega(\bY)\leq 4-2n-2\chi_\Lb(\bY)}\hspace{-20pt}\int_{-\infty}^{-1}\int_{\rS^2}|u|^{2\chi_\Lb(\bY)-2n-\varepsilon_1}\kl|\INFt{(r^2L)^{n}\bY(r^2L)(r\alphab)}\kr|^2\d\omega\d u+\varE[\alphab;0]\\
\lesssim&\sum_{\cdots}\int_{-\infty}^{-1}\int_{\rS^2}|u|^{2\chi_\Lb(\bY)-2n-\varepsilon_1}\kl|\INFt{(r^2L)^n\bY\kl(\Omega(r^2\rrho,r^2\sigma),r^3\sJ\kr)}\kr|^2\d\omega\d u+\varE[\alphab;0]\\
\lesssim& \varE[(\rho,\sigma,\alphab,\sJ);N-1].
\end{Eq*}

\subsubsection{The estimate of $\varE[\phi;0]$}
Noticing $\mathcal{E}_{\varepsilon_1}<\infty$, we know that
\begin{Eq*}
\hspace{-15pt}\sum\Low{n\leq 3\\|\beta|\leq 7-2\Ma{n}}\hspace{-15pt}\int_{-\infty}^{-1}\int_{\rS^2}|u|^{2n-\varepsilon_1}\kl|\INFt{D_\Lb^nD_\Omega^\beta(r\phi)}\kr|^2\d\omega\d u\lesssim 1.
\end{Eq*}
To estimate $\varE[\phi;0]$, we only need to notice that by \eqref{Eq:Cocd}, we know $[D_\Lb,D_\Omega]=\I r\alphab$. 
Then, similar to the process of estimating $\varE[\phi;0]$ in \Sn{Sn:PoCEo_tAphi_o_tH}, we easily get that
\begin{Eq*}
\varE[\phi;0]\lesssim (1+\varE[\alphab;0])\kl(1+\varF[\alphab;0]\kr)^5.
\end{Eq*}

\subsubsection{The estimate of $\varE[\phi;N]$ $(N=1,2)$}
The estimate of $\varE[\phi;N]$ $(N=1,2)$ is much more complicated.
To use the equation \eqref{Eq:MKGe-ud_phi}, we try to place a $D_\Lb$ in the front of $r\phi$ in $\varE[\phi;N]$.

If there is no $\Lb$ in $\bY$, we will use \eqref{Eq:key_Eob1_1} with \eqref{Eq:Eo_phi_oini}.
This step will add a $D_\Lb$ in front of $(r^2D_L)^n$.
Note that the total amount of derivatives increases by $1$ in this step.

Otherwise, if there is an $\Lb$ in $\bY$, or there is a $D_\Lb$ created by last step,
we are hoping to move it to the front of $r\phi$.
However, by \eqref{Eq:Cocd}, we know there will be extra terms with $r^2\rho$, $r\alphab$ or their derivatives produced in this step.
Nevertheless, the total amount of derivatives in the extra terms decreases by $2$.

Now, similar to the estimate of $\varE[\phi;0]$,
after a finite steps, for $N=1,2$, we can get that
\begin{Eq*}
\varE[\phi;N]
\lesssim& \bigg(\underbrace{\hspace{-20pt}\sum\Low{1\leq n\leq N\\\chi_\Lb(\bY)\leq 2-n,~\chi_L(\bY)=0\\\chi_\Omega(\bY)\leq 5-2n-2\chi_\Lb(\bY)}\hspace{-20pt}\int_{-\infty}^{-1}\int_{\rS^2}|u|^{2+2\chi_\Lb(\bY)-2n-\varepsilon_1}\kl|\INFt{(r^2D_L)^nD_\bY D_\Lb(r\phi)}\kr|^2\d\omega\d u}_{P_3[N]}\\
&\quad \phantom{\bigg(}+\varE[\alphab;N]+\varE[\phi;N-1]\bigg)\times\big(1+\varF[(\rho,\alphab,\phi);N-1]\big)^m
\end{Eq*}
with some $m$ large enough.
%

Next, we try to move one $r^2D_L$ to the front of $D_\Lb(r\phi)$ in $P_3[N]$. 
By a similar discussion as above, for $N=1,2$, we know
\begin{Eq*}
P_3[N]\lesssim&\bigg(\underbrace{\hspace{-20pt}\sum\Low{n\leq N-1\\\chi_\Lb(\bY)\leq 1-n,~\chi_L(\bY)=0\\\chi_\Omega(\bY)\leq 3-2n-2\chi_\Lb(\bY)}\hspace{-20pt}\int_{-\infty}^{-1}\int_{\rS^2}|u|^{2\chi_\Lb(\bY)-2n-\varepsilon_1}\kl|\INFt{(r^2D_L)^nD_\bY (r^2D_LD_\Lb)(r\phi)}\kr|^2\d\omega\d u}_{P_4[N]}\\
&\quad\phantom{\bigg(}+\varE[\phi;N-1]\bigg)\times \big(1+\varF[(\rho,\alphab,\alpha,\phi);N-1]\big)^m
\end{Eq*}
with some $m$ large enough.

Finally, using the equation \eqref{Eq:MKGe-ud_phi}, similar to above, we find
\begin{Eq*}
P_4[N]\lesssim \varE[(\rho,\phi);N-1](1+\varF[(\rho,\phi);N-1]).
\end{Eq*}
It means
\begin{Eq*}
\varE[\phi;N]\lesssim (\varE[\alphab;N]+\varE[\phi;N-1])\big(1+\varF[(\rho,\alphab,\alpha,\phi);N-1]\big)^m
\end{Eq*}
with some $m$ large enough.

\subsubsection{The estimate of $\varE[(J_\Lb,\sJ);N]$ and $\varF[(J_L,\sJ);N]$}
For $N=0,1,2$, we know
\begin{Eq*}
\hspace{-25pt}\sum\Low{n\leq N\\\chi_\Lb(\bY)\leq 2-n,~\chi_L(\bY)=0\\\chi_\Omega(\bY)\leq 5-2n-2\chi_\Lb(\bY)}\hspace{-25pt}\kl|(r^2L)^n\bY(r^2J_\Lb)\kr|
\lesssim\hspace{-45pt}\sum\Low{n\leq N\\\chi_\Lb(\bY)\leq 3-n,~\chi_L(\bY)=0\\\chi_\Omega(\bY)\leq 7-2n-2\Ma{\chi_\Lb(\bY)}}\hspace{-45pt}\kl|(r^2D_L)^nD_\bY (r\phi)\kr|\cdot \hspace{-35pt}\sum\Low{n\leq \Mi{N}\\\chi_\Lb(\bY)\leq 2-n,~\chi_L(\bY)=0\\\chi_\Omega(\bY)\leq 5-2n-2\Ma{\chi_\Lb(\bY)}}\hspace{-45pt} |(r^2D_L)^{n}D_{\bY} (r\phi)|.
\end{Eq*}
Then, using \emph{H\"older} inequality, and \emph{Sobolev} inequality on sphere, we easily know 
\begin{Eq*}
\varE[J_\Lb;N]\lesssim \varE[\phi;N] \varF[\phi;N].
\end{Eq*}
Similarly, we also know
\begin{Eq*}
\varE[\sJ;N]\lesssim \varE[\phi;N] \varF[\phi;N],\qquad
\varF[J_L;N]\lesssim \varF[\phi;N+1]^2,\qquad
\varF[\sJ;N]\lesssim \varF[\phi;N]^2.
\end{Eq*}

\subsubsection{The estimate of $\varF[(\phi,\rho,\sigma,\alphab);N]$}
Using \eqref{Eq:key_Eob3_1}, for $N=0,1,2$, we have
\begin{Eq*}
\varF[\phi;N]\approx& \sup_{u\in(-\infty,-1]}\hspace{-45pt}\sum\Low{n\leq N\\\chi_\Lb(\bY)\leq 2-n,~\chi_L(\bY)=0\\\chi_\Omega(\bY)\leq 7-2n-2\Ma{\chi_\Lb(\bY)}}\hspace{-45pt}|u|\int_{\rS^2}\kl||u|^{\chi_\Lb(\bY)-n-\frac{1}{2}\varepsilon_1}\INFt{(r^2D_L)^nD_\bY (r\phi)}\kr|^2\d\omega\\
\lesssim&\sum_{\cdots}\int_{-\infty}^{-1}\int_{\rS^2}\kl|\INFt{(uD_\Lb)^{\leq 1}\kl(|u|^{\chi_\Lb(\bY)-n-\frac{1}{2}\varepsilon_1}(r^2D_L)^nD_\bY (r\phi)\kr)}\kr|^2\d\omega\d u.
\end{Eq*}
Now, using \eqref{Eq:Cocd}, we know
\begin{Eq*}
\varF[\phi;N]\lesssim&\varE[\phi;N]\cdot \kl(1+\sum_{n\leq N-1}|u|^{-2n}\kl\|\INFt{(r^2L)^n(r^2\rho)}\kr\|_{L_\omega^\infty}^2\kr)\\
\lesssim&\varE[\phi;N](1+\varF[\rho;N-1])
\end{Eq*}
with $\varF[\rho;-1]:=0$. By a similar process, we also have
\begin{Eq*}
\varF[\rho;N]\lesssim \varE[\rho;N],
\qquad \varF[\sigma;N]\lesssim \varE[\sigma;N],
\qquad \varF[\alphab;N]\lesssim \varE[\alphab;N].
\end{Eq*}

\subsubsection{The estimate of $\varF[\alpha;N]$}
By \eqref{Eq:Eo_alpha_oini}, we know
\begin{Eq*}
\sum_{n\leq 1}\int_{\rS^2}\kl|\INFt{(r^2L)^n\Omega^{4-2n}(r^3\alpha)}\kr|_{u=-1}^2\d\omega\lesssim 1.
\end{Eq*}

Then, using \eqref{Eq:key_Eob1_1}, \eqref{Eq:spts_3} and \eqref{Eq:MKGe-ud_FZ}, we know
\begin{Eq*}
&\sup_{u\in(-\infty,-1]}\hspace{0pt}\sum\Low{n\leq N\\\chi_{\Lb,L}(\bY)=0\\\chi_\Omega(\bY)\leq 4-2n}|u|^{-1-2n-\varepsilon_1}\int_{\rS^2}\kl|\INFt{(r^2L)^n\bY (r^3\alpha)}\kr|^2\d\omega\\
\lesssim&\sum\Low{n\leq N\\\chi_{\Lb,L}(\bY)=0\\\chi_\Omega(\bY)\leq 4-2n}\int_{-\infty}^{-1}\int_{\rS^2}|u|^{-2n-\varepsilon_1}\kl|\INFt{(r^2L)^n\bY \Lb(rv^2\alpha)}\kr|^2\d\omega\d u+1\\
\lesssim&\sum_{\cdots}\int_{-\infty}^{-1}\int_{\rS^2}|u|^{-2n-\varepsilon_1}\kl|\INFt{(r^2L)^n\bY \kl(\Omega(r^2\rrho,r^2\sigma),r^3\sJ\kr)}\kr|^2\d\omega\d u+1\\
\lesssim&\varE[(\rho,\sigma,\sJ);N]+1.
\end{Eq*}
Meanwhile, by  \eqref{Eq:spts_3} and \eqref{Eq:MKGe-ud_FZ}, we also know
\begin{Eq*}
&\sup_{u\in(-\infty,-1]}\sum\Low{\chi_\Lb(\bY)= 1\\\chi_L(\bY)=0\\\chi_\Omega(\bY)\leq 4}|u|^{1-\varepsilon_1}\int_{\rS^2}\kl|\INFt{\bY (r^3\alpha)}\kr|^2\d\omega\\
\lesssim& \sup_{u\in(-\infty,-1]}\sum\Low{\chi_{\Lb,L}(\bY)=0\\\chi_\Omega(\bY)\leq 4}|u|^{1-\varepsilon_1}\int_{\rS^2}\kl|\INFt{\bY \Lb(rv^2\alpha)}\kr|^2\d\omega\\
\lesssim& \sup_{u\in(-\infty,-1]}\sum\Low{\chi_{\Lb,L}(\bY)=0\\\chi_\Omega(\bY)\leq 4}|u|^{1-\varepsilon_1}\int_{\rS^2}\kl|\INFt{\bY \kl(\Omega(r^2\rrho,r^2\sigma),r^3\sJ\kr)}\kr|^2\d\omega\\
\lesssim&\varF[(\rho,\sigma,\sJ);N].
\end{Eq*}
Mixing them, we finally have
\begin{Eq*}
\varF[\alpha;N]\lesssim \varE[(\rho,\sigma,\sJ);N]+\varF[(\rho,\sigma,\sJ);N]+1.
\end{Eq*}

\subsubsection{End of the proof}
Now, using iteration, we easily find that all quantities defined in \eqref{Eq:itf_21} and \eqref{Eq:itf_22} can be controlled by some constant depended on the initial data.

Now, using the calculation trick
\begin{Eq*}
\kl|\kl(s^2\partial_s\kr)^N(v^{-m}f)\kr|\lesssim\sum_{\max\{N-m,0\}\leq n\leq N}v^{N-m-n}|\kl(s^2 \partial_s\kr)^nf|,
\end{Eq*}
and the fact that $\INF{v^{\leq 1} f}=0$ while $\INF{v^2 f}$ is finite,
we know
\begin{Eq*}
\sum_{\chi(\bZ)\leq 2}|u|^{-\chi_K(\bZ)}\kl|\INFt{r^2\rrho^{(\bZ)}}\kr|
\lesssim&\sum_{\chi(\bZ)\leq 2}|u|^{-\chi_K(\bZ)}\kl|\INFt{\bZ (r^2\rrho)}\kr|\\
\lesssim& \sum\Low{n\leq 2\\\chi_{\Lb,\Omega}(\bY)\leq 2-n\\\chi_L(\bY)=0}|u|^{\chi_\Lb(\bY)-n}\kl|\INF{(r^2L)^n\bY(r^2\rrho)}\kr|.
\end{Eq*}
It means that for all $\chi(\bZ)\leq 2$, there is
\begin{Eq*}
\int_{-\infty}^{-1}\int_{\rS^2} |u|^{-2\chi_K(\bZ)-\varepsilon_1}\kl|\INF{r^2\rrho^{(\bZ)}}\kr|^2\d\omega\d u\lesssim \varE[\rho;2]\lesssim 1.
\end{Eq*}
Through similar discussions, we also know
\begin{Eq*}
&\int_{-\infty}^{-1}\int_{\rS^2} |u|^{-2\chi_K(\bZ)-\varepsilon_1}\kl|\INF{\kl(r^2\sigma^{(\bZ)},r|u|\alphab^{(\bZ)}\kr)}\kr|^2\d\omega\d u\lesssim \varE[(\sigma,\alphab);2]\lesssim 1,\\
&\int_{-\infty}^{-1}\int_{\rS^2} |u|^{-2\chi_K(\bZ)-\varepsilon_1}\kl|\INF{(|u|D_\Lb)^{\leq 1}D_\Omega^{\leq 1}(r\phi^{(\bZ)})}\kr|^2\d\omega\d u\lesssim \varE[\phi;2]\lesssim 1,\\
&\int_{-\infty}^{-1}\int_{\rS^2} |u|^{-2\chi_K(\bZ)+1-2\varepsilon_1}\kl|\INF{vr\kl(|u|J^{(\bZ)}_\Lb,v\sJ^{(\bZ)}\kr)}\kr|^2\d\omega\d u\lesssim \varE[(J_\Lb,\sJ);2]\lesssim 1.
\end{Eq*}
These give the \eqref{Eq:Eo_FphiZ_oeni}.
Meanwhile, for $\chi(\bZ)\leq 1$, there are
\begin{Eq*}
&\sup_{u\in(-\infty,-1]}|u|^{-2\chi_K(\bZ)-1-\varepsilon_1}\int_{\rS^2}\kl|\INF{\kl(r|u|^2\alphab^{(\bZ)},r^2|u|(\rrho^{(\bZ)},\sigma^{(\bZ)}),r^3\alpha^{(\bZ)}\kr)}\kr|^2\d\omega\\
\lesssim&\varF[(\alphab,\rho,\sigma,\alpha);1]\lesssim 1.
\end{Eq*}
This gives the \eqref{Eq:Eo_FZ_onis}.
Moreover, using \eqref{Eq:key_Eob3_1}, for $\chi(\bZ)\leq 2$, we know
\begin{Eq*}
&\sup_{u\in(-\infty,-1]}|u|\int_{\rS^2}\kl|\INF{D_\Omega^{\leq 1}\kl(|u|^{-\chi_K(\bZ)-\frac{1}{2}\varepsilon_1}r\phi^{(\bZ)}\kr)}\kr|^2\d\omega\\
\lesssim&\int_{-\infty}^{-1}\int_{\rS^2} \kl|\INF{(|u|D_\Lb)^{\leq 1}D_\Omega^{\leq 1}\kl(|u|^{-\chi_K(\bZ)-\frac{1}{2}\varepsilon_1}r\phi^{(\bZ)}\kr)}\kr|^2\d\omega\d u\\
\lesssim&\int_{-\infty}^{-1}\int_{\rS^2} |u|^{-2\chi_K(\bZ)-\varepsilon_1}\kl|\INF{(|u|D_\Lb)^{\leq 1}D_\Omega^{\leq 1}(r\phi^{(\bZ)})}\kr|^2\d\omega\d u\lesssim 1.
\end{Eq*}
This and the \emph{Sobolev} inequality give \eqref{Eq:Eo_phiZ_onis}.
Now we finish the proof.

\subsection*{Acknowledgment}
The authors would like to thank the anonymous referee for the careful reading and valuable comments.
W. Dai is supported by the  Natural Science Foundation of Zhejiang province LQN25A010005.
S. Yang is supported by the National Key R\&D Program of China 2021YFA1001700 and the National Science Foundation of China 12171011,  12141102. 
He thanks Mihalis Dafermos for encouragement and helpful discussion on this problem.

\end{document}